\documentclass{amsart}
\usepackage{bbm}
\usepackage{mathabx}
\usepackage{tensor}
\usepackage{amssymb}
\usepackage{amsmath}
\usepackage{amsfonts}
\usepackage{amsthm}
\usepackage{stmaryrd}
\usepackage[all]{xy}
\usepackage{mathrsfs}
\usepackage{graphicx}
\usepackage{hyperref}
\usepackage{color}
\usepackage{tikz-cd}
\usepackage{tkz-euclide}
\usepackage{tikz}
\usetikzlibrary{matrix}
\usepackage{enumitem}
\usepackage{amsmath,amscd}
\usepackage{pifont}
\usetikzlibrary{arrows}
\usepackage[all]{xy}
\usepackage{graphicx}
\usepackage{placeins}
\usepackage{xspace}
\usetikzlibrary{calc,intersections,through,backgrounds}

\newtheorem{thm}{Theorem}[section]
\newtheorem{theorem}[thm]{Theorem}
\newtheorem{conjecture}[thm]{Conjecture}

\newtheorem{corollary}[thm]{Corollary}
\newtheorem{lemma}[thm]{Lemma}
\newtheorem{proposition}[thm]{Proposition}

\theoremstyle{definition}

\newtheorem{construction}[thm]{Construction}

\newtheorem{definition}[thm]{Definition}
\newtheorem{example}[thm]{Example}

\newtheorem{notation}[thm]{Notation}

\newtheorem{remark}[thm]{Remark}

\newtheorem{hypothesis}[thm]{Hypothesis}

\numberwithin{equation}{section}

\numberwithin{mytheorem}{subsection}

\numberwithin{mytheorem}{subsection}
\numberwithin{myconjecture}{subsection}
\numberwithin{mydefinition}{subsection}
\numberwithin{myremark}{subsection}
\numberwithin{mysituation}{subsection}
\numberwithin{myhypothesis}{subsection}
\numberwithin{myquestion}{subsection}
\numberwithin{mynotation}{subsection}
\numberwithin{myfact}{subsection}
\numberwithin{myexamples}{subsection}
\numberwithin{myexample}{subsection}
\numberwithin{myconstruction}{subsection}
\numberwithin{mycaution}{subsection}
\numberwithin{myproposition}{subsection}
\numberwithin{mylemma}{subsection}
\numberwithin{mycorollary}{subsection}

\def\AAA{\mathbb{A}}
\def\BB{\mathbb{B}}
\def\CC{\mathbb{C}}

\def\FF{\mathbb{F}}
\def\GG{\mathbb{G}}

\def\KK{\mathbb{K}}
\def\LL{\mathbb{L}}
\def\MM{\mathbb{M}}
\def\NN{\mathbb{N}}

\def\QQ{\mathbb{Q}}
\def\RR{\mathbb{R}}
\def\SS{\mathbb{S}}
\def\TT{\mathbb{T}}

\def\YY{\mathbb{Y}}
\def\ZZ{\mathbb{Z}}

\def\sp{\mathrm{sp}}

\def\NP{\mathrm{NP}}
\def\IHP{\mathrm{IHP}}
\def\GNP{\mathrm{GNP}}

\def\sgn{\mathrm{sgn}}

\def\Int{\mathrm{Int}}
\def\dom{\mathrm{Dom}}
\def\ran{\mathrm{Im}}

\def\area{\mathrm{Area}}

\def\Frob{\mathrm{Frob}}

\def\univ{\textrm{univ}}
\def\Zp{\ZZ_p}
\newcommand{\Tr}{\mathrm{Tr}}

\DeclareMathOperator{\Cone}{Cone}

\DeclareMathOperator{\Iso}{Iso}

\pgfkeys{tikz/mymatrixenv/.style={decoration=brace,every left delimiter/.style={xshift=3pt},every right delimiter/.style={xshift=-3pt}}}
\pgfkeys{tikz/mymatrix/.style={matrix of math nodes,left delimiter=[,right delimiter={]},inner sep=2pt,column sep=1em,row sep=0.5em,nodes={inner sep=0pt}}}
\pgfkeys{tikz/mymatrixbrace/.style={decorate,thick}}

\begin{document}\large

\title{Generic Newton polygon for exponential sums in two variables with triangular base} 

\author{Rufei Ren}
\address{University of California, Irvine, Department of
	Mathematics, 340 Rowland Hall, Irvine, CA 92697}
\email{rufeir@math.uci.edu}
\date{\today}
\begin{abstract}
	Let $p$ be a prime number. Every two-variable polynomial $f(x_1, x_2)$ over a finite field of characteristic $p$ defines an Artin--Schreier--Witt tower of surfaces whose Galois group is isomorphic to $\ZZ_p$. 
	Our goal of this paper is to study the Newton polygon of the $L$-functions associated to a finite character of $\ZZ_p$ and a generic polynomial whose convex hull is a fixed triangle $\Delta$. We denote this polygon by $\GNP(\Delta)$. We prove a lower bound of $\GNP(\Delta)$, which we call the improved Hodge polygon $\IHP(\Delta)$, and we conjecture that $\GNP(\Delta)$ and $\IHP(\Delta)$ are the same. We show that if $\GNP(\Delta)$ and $\IHP(\Delta)$ coincide at a certain point, then they coincide at infinitely many points. 
	
	When $\Delta$ is an isosceles right triangle with vertices $(0,0)$, $(0, d)$ and $(d, 0)$ such that $d$ is not divisible by $p$ and that the residue of $p$ modulo $d$ is small relative to $d$, we prove that $\GNP(\Delta)$ and $\IHP(\Delta)$ coincide at infinitely many points. As a corollary, we deduce that the slopes of $\GNP(\Delta)$ roughly form an arithmetic progression with increasing multiplicities. 
\end{abstract}

	\maketitle
	
	\setcounter{tocdepth}{1}
	\tableofcontents
	
	\section{Introduction}
	We shall state our main results and their motivation after recalling the notion of
	$L$-functions for Witt coverings. 
	Let $p$ be a prime number.
	 Let $$f(x_1,x_2):=\sum_{P\in \ZZ^2_{\geq 0}} a_{P}x_1^{P_x} x_2^{P_y}$$ be a two-variable polynomial in $\overline\FF_p[x_1, x_2]$ and write $$\hat{f}(x_1,x_2):=\sum\limits_{P\in \ZZ_{\geq 0}^2} \hat{a}_{P}x_1^{P_x} x_2^{P_y}$$ for its Teichm\"uller lift, where $\hat a_P$ denotes the Teichm\"uller lift of $a_P$.
	 We use $\FF_p(f)$ to denote the extension of $\FF_p$ generated by all coefficients of $f$ and set $n(f):=[\FF_p(f):\FF_p]$. 
	The convex hull of the set of points $(0,0)\cup\big\{P\;|\;a_{P}\neq 0\big\}$ is called \emph{the polytope} of $f$ and denoted by $\Delta_f$.
	
	Let $(\GG_{m})^2$ be the two-dimensional torus over $\mathbb{F}_{p^{n(f)}}$.
	The main subject of our study is the $L$-function associated to finite characters $\chi: \Zp \to \CC_p^\times$ of conductor $p^{m_\chi}$ given by
	$$L_f^*(\chi,s): =\prod\limits_{x\in |(\GG_{m})^2|}
	\frac{1}{1-\chi\big(\Tr_{\QQ_{p^{n(f)\deg(x)}}/\QQ_p}(\hat{f}(\hat{x}))\big)s^{\deg(x)}},$$
		where $|(\GG_{m})^2|$ is the set of closed points of $(\GG_{m})^2$ and $\hat{x}$ is the Teichm\"uller lift of a closed point $x$ in $(\GG_{m})^2$. 	
		The characteristic power series $C^*_f(\chi, s)$ is a product of reciprocals of $L$-functions:
		\begin{equation}\label{C(chi, s)}
		C^*_f(\chi, s)=\prod\limits_{j=0}^{\infty}L^*_f(\chi, p^{jn(f)}s)^{-(j+1)}.
		\end{equation}
	We can alternatively express $L_f^*(\chi, s)$ in terms of $C_f^*(\chi, s)$ as  $$ L^*_f(\chi, s)=\Big(\frac{C^*_f(\chi, s)C^*_f(\chi, p^{2n(f)}s)}{C^*_f(\chi, p^{n(f)}s)^2}\Big)^{-1}.$$
	Therefore, $C^*_f(\chi, s)$ and $L^*_f(\chi, s)$ determine each
		other. 
		
	\begin{definition}\label{definition for NP chi}
	From \cite{liu-wei}, we know that $$L^*_f(\chi, s)^{-1}:=\sum_{i=0}^{2p^{2(m_\chi-1)} \area (\Delta_f)} v_is^i$$ is a polynomial of degree $2p^{2(m_\chi-1)} \area (\Delta_f)$ in $\ZZ_p[\zeta_{p^{m_\chi}}][s]$, where $\zeta_{p^{m_{\chi}}}$ is a primitive $p^{m_{\chi}}$-th root of unity. We call the lower convex hull of the set of points $\big(i,p^{m_\chi-1}(p-1)v_{p^{n(f)}}(v_i)\big)$ the \emph{normalized Newton polygon of $L^*_f(\chi, s)^{-1}$}, which is denoted by $\NP(f, \chi)_{L^{-1}}$. Here, $v_{p^{n(f)}}(-)$ is the $p$-adic valuation normalized so that $v_{p^{n(f)}}(p^{n(f)})=1.$  Similarly,
	we write $\NP(f, \chi)_C$ for the normalized Newton polygon of $C^*_f(\chi, s)$.
	\end{definition}
		
		In \cite{Davis-wan-xiao}, Davis, Wan and Xiao studied the $p$-adic Newton slopes  of $L^*_f(\chi, s)$ when $f(x)$ is a one-variable polynomial whose degree $d$ is coprime to $p$. They concluded that, for each character $\chi:\Zp\to \CC_p^\times$ of relatively large conductor, $\NP(f, \chi)_{L^{-1}}$ depends only on its conductor. We briefly introduce their proof as follows.
		
		They proved a lower bound of $\NP(f, \chi)_C$ when $\chi$ is the so-called universal character and an upper bound by the Poincar\'e  duality of roots of $L^*_f(\chi_1, s)$ for a \emph{particular character} $\chi_1$ of conductor $p$. The lower bound is called the \emph{Hodge polygon} in their paper. Then they verified that the upper bound coincides with the lower bound at $x=kd$ for any non-negative integer $k$. Since the Newton polygon of $C^*_f(\chi, s)$ is confined between these two bounds, it also passes through their intersections. See more details in \cite{Davis-wan-xiao}. 
		
		We also mention here that the aforementioned proof strongly inspired the proof of spectral halo conjecture by Liu, Wan, and Xiao in \cite{liu-wan-xiao}; we refer to \cite{ren-wan-xiao-yu} for the discussion on the analogy of the two proofs.
	Motivated by the attempt of extending spectral halo type results beyond the case of modular forms, it is natural to ask whether one can generalize the main results of \cite{Davis-wan-xiao} to more general cases of exponential sums and Artin--Schreier--Witt towers.
	For example, in a joint work with Wan, Xiao, and Yu, we examined the case when the Galois group of the Artin--Schreier--Witt tower is canonically isomorphic to $\ZZ_{p^\ell}$.
		
		In this paper, we mainly deal with the generic Newton polygon of $L$-functions for \emph{two-variable} polynomials.  
	We want to apply the methods in \cite{Davis-wan-xiao} to this case. Therefore, it is crucial for us to give a lower bound and an upper bound for  $C^*_f(\chi, s)$. However, the
		Hodge polygon provided by Liu and Wan in \cite{liu-wan} is no longer optimal, and is in general strictly lower than the upper bound we obtain by Poincare duality. Our main contribution in this paper is to find an improved lower bound for $\NP(f,\chi)_C$, which we call the \emph{improved Hodge bound}  $\IHP(\Delta)$. We conjecture that our improved Hodge polygon is optimal, and is equal to the generic Newton polygon, that is the lowest Newton polygon for all polynomials $f$ with the same convex hull.
		
		When $\Delta_f$ is an isosceles right triangle with vertices $(0,0)$, $(d,0)$ and $(0,d)$, we will give an equivalent condition to verify the coincidence of improved Hodge polygon with the Newton polygon (at infinitely many points), and we will show that this condition is met for a generic polynomial with convex hull $\Delta_f$.

	We now turn to stating our main results more rigorously.
	\begin{notation}\label{cone}
		For a two-dimensional convex polytope $\Delta$ which contains $(0,0)$, we denote its \emph{cone} by 
		$$\Cone(\Delta):=\Big\{P\in \RR^2\;\big |\; kP\in \Delta\ \textrm{for some}\ k>0 \Big\},$$ and put $$\MM(\Delta):=\Cone(\Delta)\cap \ZZ^2$$ to be the set of lattice points in $\Cone(\Delta)$.
		
		Moreover, we write $\TT_k(\Delta)$ (resp. $\TT'_k(\Delta)$) for the set consisting of all points in $\MM(\Delta)$ with weight $w$ (See Definition~\ref{weight function}) strictly less than $k$ (resp. less than or equal to $k$), and denote its cardinality by $\mathbbm{x}_k(\Delta)$ (resp. $\mathbbm{x}'_k(\Delta)$). 
	\end{notation}
	

	\begin{notation}
		For integers $a$ and $b$, we denote by $a\%b$ the residue of $a$ modulo $b$.
	\end{notation}

	\begin{definition}\label{generic Newton polygon}
	The \emph{generic Newton polygon} of $\Delta$ is defined by
	$$\GNP(\Delta):=\inf\limits_{\substack{\chi: \Zp/p^{m_\chi}\Zp\to \CC_p^\times\\\Delta_f=\Delta}}\Big(\NP(f, \chi)_{L^{-1}}\Big),$$
	where $\chi:\ZZ_p\to \CC_p^\times$ runs over all finite characters, and $f$ runs over all polynomials in $\overline \FF_p[x_1,x_2]$ such that $\Delta_f=\Delta$. The following are our main results.
\end{definition}
	
			\begin{theorem}\label{generic newton polygon}
			Let $\Delta$ be a right isosceles triangle with vertices $(0,0),(0,d),(d,0)$, where $d$ is a positive integer not divisible by $p$. Let $p_0$ be the residue of $p$ modulo $d$. Suppose $d\geq 24(2p_0^2+p_0)$. Then the generic Newton polygon $\GNP(\Delta)$ passes through points	
			 $(\mathbbm{x}_k(\Delta)+i, h_k(\Delta)+ki)$ for any $k\geq 0$ and $0\leq i\leq kd+1$, where 
			$$\mathbbm{x}_k(\Delta)=\frac{(kd+1)kd}{2}\quad \textrm{and}\quad h_k(\Delta)=\frac{(p-1)(k-1)k(k+1)d^2}{3}+k\sum_{P\in \TT_1(\Delta)}\lfloor pw(P)\rfloor.$$
		\end{theorem}
	
	The points  $(\mathbbm x_k, h_k(\Delta))$ are vertices for the improved Hodge polygon $\IHP(\Delta)$ (see Definition~\ref{IHP} and Proposition~\ref{better form for IHP}). So the essential content of the proof is to show that the generic Newton polygon $\GNP(\Delta)$ also passes through these points. The proof of Theorem~\ref{generic newton polygon} consists of two parts: first we show that, for a fixed polynomial $f$ with convex hull $\Delta_f$, if $\IHP(\Delta)$ coincides with the corresponding Newton polygon $\NP(f,\chi_1)_C$ at $\mathbbm x_1$, then these two polygons agree at all points $x= \mathbbm x_k(\Delta)+i$ for $k\geq 1$ and $0\leq i\leq kd+1$.  This is proved in Theorem~\ref{main}, which in fact holds with less constraints on $\Delta$.  Next, we prove that, for a generic polynomial $f$, $\NP(f, \chi_1)_C$ agrees with $\IHP(\Delta)$ at $x=\mathbbm x_1(\Delta)$. For this, we look at the leading term of $\widetilde v_{\mathbbm x_1(\Delta)}$ for the universal polynomial $f_\univ$ with convex hull $\Delta$ and show that this term is non-zero when $d \geq 24(2p_0^2+p_0)$. This is proved in Theorem~\ref{Thm for 4}, which in fact holds under a weaker condition on $p_0$.

	From \cite[Theorem~1.4]{liu-wei}, for a finite character $\chi$ of conductor $p^{m_\chi}$, we know that $L^*_f(\chi,s)^{-1}$ has degree of $p^{2(m_\chi-1)}d^2$. 
	\begin{theorem}\label{theorem for L}
		Under the hypotheses of Theorem~\ref{generic newton polygon},
		if we put  $(\alpha_1,\dots, \alpha_{p^{2(m_\chi-1)}d^2})$ to be the sequence of $p^{n(f)}$-adic Newton slopes of $L^*_f(\chi,s)^{-1}$ (in non-decreasing order), then for first $\frac{p^{2(m_\chi-1)}d^2+p^{(m_\chi-1)}d}{2}$-th slopes we have 
		\[\begin{cases}
		\alpha_{\mathbbm x'_i+1},\dots,\alpha_{\mathbbm x_{i+1}}\in (\frac{i}{p^{m-1}},\frac{i+1}{p^{m_\chi-1}})& \textrm{for}\ i=0, 1,\dots, p^{m_\chi-1}-1,\\
		\alpha_{\mathbbm x_i+1},\dots,\alpha_{\mathbbm x'_{i}}= \frac{i}{p^{m_\chi-1}}& \textrm{for}\ i=0,1,2,\dots, p^{m_\chi-1}-1,\\
		\alpha_{\mathbbm x_{p^{m_\chi-1}}+1},\dots,\alpha_{\mathbbm x'_{p^{m_\chi-1}}-2}=1.& 
		\end{cases}\]
		
	\end{theorem}

	In fact, points $(\mathbbm{x}_k(\Delta), h(\TT_k))$ are vertices of the improved Hodge polygon (see Definition~\ref{IHP} and Proposition~\ref{better form for IHP}).
	
	We do not know if Theorem~\ref{generic newton polygon} still holds for polytopes which are not right isosceles triangle. However, for an arbitrary multi-variable polynomial $f$ in $\overline\FF_p[\underline{x}]$, we are still able to get an improved Hodge polygon for $\NP(f, \chi)$. Especially, when $f$ is a two-variable polynomial, it is expected that the slopes of the improved Hodge polygon form certain generalized arithmetic progression. We plan to address this in a forthcoming paper."

	The Newton polygon for exponential sums was explicitly computed in the ``ordinary'' case by Adolphson--Sperber \cite{AS}, Berndt--Evans \cite{BR}, and Wan \cite{wan} in many special cases, and in general (namely the $T$-adic setup) by Liu--Wan \cite{liu-wan}. For the $\Delta$ we considered in Theorem~\ref{generic newton polygon}, the ordinary condition amounts to requiring $p \equiv 1 \pmod d$. Blache, Ferard, and Zhu in \cite{bfz} proved a lower bound for the Newton polygon of one-variable Laurent polynomial over $\FF_q$ of degree $(d_1, d_2)$, which is called a Hodge-Stickelberger polygon. They also showed that when $p$ approaches to infinite, the Newton polygon coincides the  Hodge-Stickelberger polygon. 
	
	Going beyond the ordinary case, there has been many researches on understanding the generic Newton polygon of $L_f(\chi, s)$ when $f$ is a polynomial of a single variable. 
	The first results are due to Zhu \cite{zhu03} and Scholten--Zhu \cite{sz}, when $p$ is large enough. In \cite{bf}, Blache and Ferard worked on the generic Newton polygon associated to characters of large conductors.
	In \cite{ouy}, Ouyang and Yang studied the one-variable polynomial $f(x) = x^d +a_1x$. A similar result can be found in \cite{ouz}, where Ouyang and Zhang studied the family of polynomials of the form  $f(x) = x^d +a_{d-1}x^{d-1}$. 	
	
	Our Theorem~\ref{generic newton polygon} maybe considered as the first step beyond the ordinary case when the base polynomial is multivariable. A similar result is obtained by Zhu in \cite{zhu} independently which shows that  $\GNP(\Delta_f)$ and $\IHP(\Delta_f)$ coincide for characters of $\ZZ_p$ of conductor $p$.

%
%

	\subsection*{Acknowledgments}
	The author would like to thank his advisor Liang Xiao for the extraordinary support in this paper and also thank Douglass Haessig, Hui June Zhu, and Daqing Wan for helpful discussion. 	

\section{Dwork trace formula}\label{section 2} 
Let $p$ be an odd prime and let $f(x_1,x_2):=\sum_{P\in \ZZ^2_{\geq 0}} a_{P}x_1^{P_x} x_2^{P_y}$ be a two-variable polynomial in $\overline\FF_p[x_1, x_2]$. Denote $F_p(f)$ to be the finite field generated by the coefficients of $f$, which we call the \emph{coefficient field} of $f$. The convex hull of the set of points $\{(0,0)\}\cup\big\{P\; |\;a_{P}\neq 0\big\}$ is called \emph{the polytope} of $f$ and denoted by $\Delta_f$.

Our discussion will focus on a fixed $f$ until  Proposition~\ref{reduce to Ef}. 
We put $\FF_q=\FF_p(f)$ and $n=[\FF_q:\FF_p]$. 
Let $\hat a_P \in \ZZ_q$ be the Teichm\"uller lift of $a_{P}$. We call $\hat f(x_1,x_2):=\sum_{P\in \ZZ^2_{\geq 0}} \hat a_{P}x_1^{P_x} x_2^{P_y}$ the \emph{Teichm\"uller lift} of $f(x)$.

For convenience, we put $v_{p}(-)$ (resp. $v_{q}(-)$) be the $p$-adic valuation normalized so that $v_{p}(p)=1$ (resp. $v_{q}(q)=1$).


\subsection{$T$-adic exponential sums.}

\begin{notation}
	We recall that the \emph{Artin--Hasse exponential series} is defined by
	\begin{equation}\label{Artin-Hasse}
		E(\pi) = \exp\big( \sum_{i=0}^\infty \frac{\pi^{p^i}}{p^i} \big) = \prod\limits_{p \nmid i,\ i \geq 1} \big( 1-\pi^i\big)^{-\mu(i)/i} \in 1+ \pi + \pi^2 \ZZ_p[\![ \pi ]\!].
	\end{equation}
	Putting $ E(\pi)= T+1$ gives an isomorphism $\ZZ_p\llbracket\pi \rrbracket \cong \ZZ_p\llbracket T\rrbracket$. 
	\begin{definition}\label{T-adic valuation}
			For each power series in $\ZZ_q\llbracket T\rrbracket$, say $g(T)$, we define its \emph{$T$-adic valuation} as the largest $k$ such that $g\in T^k\ZZ_q\llbracket T\rrbracket$ and denote it by $v_T(g)$.
	\end{definition}

	
\end{notation}

\begin{definition}	
	For each $k\geq 1$, the \emph{$T$-adic exponential sum} of $f$ over $\FF_{q^k}^\times$ is
	\[
	S_f^*(k, T): = \sum_{(x_1,x_2)\in (\FF_{q^k}^\times)^2} (1+T)^{\Tr_{\QQ_{q^k} / \QQ_p}(\hat{f}(\hat{x}_1,\hat {x}_2))} \in \ZZ_p[\![T]\!].
	\]
	\begin{definition}
		The \emph{$T$-adic $L$-function of $f$} is defined by 
		$$
		L_f^*(T,s) = \exp \Big( \sum_{k=1}^\infty S_f^*(k,T)\frac{s^k} {k} \Big)$$
		and its corresponding \emph{$T$-adic characteristic power series} is defined by 
		\begin{eqnarray}
		\label{E:Cfstar}
		C_f^*(T,s) &:= &\exp \Big( \sum_{k=1}^\infty -(q^k-1)^{-2} S_f^*(k,T)\frac{s^k}{k} \Big)
		\\
		\nonumber
		&=& \displaystyle \sum_{k=0}^\infty u_k( T) s^k \in \ZZ_p\llbracket T. s\rrbracket,
		\end{eqnarray}
		We put $u_k(T)=u_{k,j}T^j\in \Zp[\![T]\!]$. 
			\end{definition}
		Moreover, they determine each other by relations:
			\begin{equation}\label{C(T, s)}
		C^*_f(T, s)=\Big(\prod\limits_{j=0}^{\infty}L^*_f(T, q^js)^{j+1}\Big)^{-1}
		\end{equation}
		and
		\begin{equation}
	 L^*_f(T, s)=\Big(\frac{C^*_f(T, s)C^*_f(T, q^2s)}{C^*_f(T, qs)^2}\Big)^{-1}.
		\end{equation}

	It is clear that for a finite character $\chi: \ZZ_{p} \to \CC_p^\times$, we have
	\[
	L_f^*(\chi,s) =L_f^*(T,s)\big|_{T = \chi(1)-1} \quad \textrm{and}\quad
	C_f^*(\chi,s) =C_f^*(T,s)\big|_{T = \chi(1)-1},
	\]
where $L_f^*(\chi,s)$ and $C_f^*(\chi,s)$ are defined in the introduction.
	
\end{definition}

\begin{notation}
	Recall that we put $E(\pi)=T+1$. We put
	\begin{equation}	\label{E:Ef(x)}
	\begin{split}
	E_f(x_1,x_2):=& \prod\limits_{P\in \ZZ^2_{\geq 0}} E(\hat{a}_{P} \pi x_1^{P_x}x_2^{P_y})\\
	 =&\sum\limits_{P\in \ZZ^2_{\geq 0}} e_{P}(T)x_1^{P_x}x_2^{P_y}\in \ZZ_q\llbracket T\rrbracket \llbracket x_1,x_2 \rrbracket.
	\end{split}
	\end{equation}

\end{notation}


\subsection{Dwork's trace formula}
Recall that $\Delta_f$ is the convex hull of $f(x_1,x_2)$ and $\MM(\Delta_f)$ is defined in Notation~\ref{cone} as a set consisting of all the lattice points in the $\Cone(\Delta_f)$.
Let $D$ be the smallest positive integer such that $w(\MM(\Delta_f))\subset \frac{1}{D}\ZZ.$

\begin{definition}\label{Banach space}
We fix a D-th root $T^{1/D}$ of T. Define
$$\textbf{B}=\Big\{\sum\limits_{P\in \MM(\Delta_f)}b_{P}(T^{1/D} x_1)^{P_x}(T^{1/D} x_2)^{P_y}\;\Big|\;b_{P}\in \ZZ_q\llbracket T^{1/D}\rrbracket, v_T(b_{P})\to +\infty, \textrm{when}\ w(P)\to \infty\Big\}.$$
\end{definition}
Let $\psi_p$ denote the operator on $\bold{B}$ such that
$$\psi_p\Big(\sum\limits_{P\in \MM(\Delta_f)}b_{P}x_1^{P_x}x_2^{P_y}\Big): = \sum\limits_{P\in \MM(\Delta_f)}b_{(pP)}x_1^{P_x}x_2^{P_y}.$$ 

Recall that $n = [\FF_q:\FF_p]$.
\begin{definition}
	Define
	\begin{equation}
	\label{E:psi}
	\psi := \sigma_{\Frob}^{-1}\circ\psi_p \circ E_{f}(x_1, x_2): \bold{B} \longrightarrow \bold{B},
	\end{equation}
	and its $n$-th iterate
	$$\psi^n=\psi_p^n\circ \prod_{i=0}^{n-1}E_f^{\sigma_{\Frob}^i}(x_1^{p^i}, x_2^{p^i}),$$
	where $\sigma_\Frob$ represents the arithmetic Frobenius acting on the coefficients, and for any $g\in \bold{B}$ we have
	$E_{f}(x_1,x_2)(g):=E_{f}(x_1,x_2)\cdot g$ . 
\end{definition} 
One can easily check that 
\[
\psi_p \circ E_{f}(x_1,x_2)\big(x_1^{P_x}x_2^{P_y}\big) = \sum_{Q\in \MM(\Delta_f)} e_{pQ-P}(T) x_1^{Q_x}x_2^{Q_y},
\]
where $e_{pQ-P}(T)$ is defined in \eqref{E:Ef(x)}.
%

\begin{theorem}[Dwork Trace Formula]
	
	For every integer $k>0$, we have
	$$(q^k-1)^{-2}S_f^*(k,\underline{\pi})=\Tr_{\bold{B}/\ZZ_q[\![ \pi]\!]}\big(\psi^{nk}\big).$$
\end{theorem}
\begin{proof} This was proved by \cite[Lemma~4.7]{liu-wei}.  
\end{proof}
One can see \cite{wan} for a a thorough treatment of the universal Dwork trace formula. 

\begin{proposition}[Analytic trace formula]\label{determinant}
	The theorem above has an equivalent multiplicative form:
	\begin{equation}\label{dwork}
	\begin{split}
		C_f^*(T, s)=& \det\big(I-s \psi^n  \;|\; \bold{B}/\ZZ_q[\![ \pi]\!] \big).
	\end{split}
	\end{equation}
\end{proposition}
\begin{proof}
	Also see \cite[Theorem~4.8]{liu-wei}.
\end{proof}
\begin{definition}
	The normalized Newton polygon of $C_f^*(T, s)$, denoted by $\NP(f, T)_C$, is the lower convex hull of the set of points $\Big\{\left(i,\frac{v_T(u_i)}{n}\right)\Big\}$.
\end{definition}

%
%



\begin{notation}\label{Delta}	
	In this paper, we fix $\Delta$ to be a triangle with vertices at $(0, 0)$, $\mathbf{P_1}:=(a_1,b_1)$ and $\mathbf{P_2}:=(a_2,b_2)$.
\end{notation}

\begin{definition}\label{weight function}
	For each lattice point $P$ in $\ZZ^2$, assume that $Q$ is the intersection of the lines $\overline{OP}$ and $\overline{\mathbf{P_1P_2}}$. Then
	we call $$w(P):=\tfrac{\overrightarrow{OP}}{\overrightarrow{OQ}}$$ \emph{the weight} of $P$.
	
	The weight function $w$ is linear, i.e. Any two points $P$ and $Q$ in $\ZZ_{\geq 0}^2$ satisfy \begin{equation}\label{linear}
	w(P+Q)=w(P)+w(Q).
	\end{equation}
	
	Equality \eqref{linear} does not always hold for a general polytope.	

	%
\end{definition} 
 We shall frequently work with multisets, i.e. sets of possibly repeating elements. They are often marked by a superscript star to be distinguished from regular sets, e.g. $S^\star$.
	The disjoint union of two multiset $S^\star$ and ${S'}^\star$ is denoted by $S^\star\uplus {S'}^\star$ as a multiset. 
\begin{definition}
Let $\SS$ be a subset of $\MM(\Delta)$. Then we write ${\SS^\star}^m$ (resp. ${\SS^\star}^{\infty}$) for the union of $m$ (resp. countably infinite) copies of $\SS$ as a multiset. 
\end{definition}

\begin{notation}\label{Iso and Prem}
	For any sets $\SS_1^\star$ and $\SS_2^\star$ in ${\MM(\Delta)^\star}^\infty$ of the same cardinality, we denote by $\Iso(\SS_1^\star, \SS_2^\star)$ the set of all bijections (as multisets) from $\SS_1^\star$ to $\SS_2^\star$. When $\SS_1^\star=\SS_2^\star=\SS^\star$, we denote $\Iso(\SS^\star):=\Iso(\SS^\star, \SS^\star)$.
\end{notation}

\begin{definition}\label{definition of h}
	For a bijection $\tau$ in $\Iso(\SS_1^\star, \SS_2^\star)$, we define
	\begin{equation}\label{defintion of h SS}
	h(\SS_1^\star, \SS_2^\star, \tau):=\sum\limits_{P\in \SS_1^\star} \big\lceil w(p\tau(P)-P)\big\rceil. 
	\end{equation}
	
	For any submultiset ${\SS_1'}^\star$ of $\SS_1^\star$, we write $\tau|_{{\SS_1'}^\star}$ for the restriction  of $\tau$ to ${\SS_1'}^\star$. Moreover, the minimum of $h(\SS_1^\star, \SS_2^\star, \tau)$ is denoted by 
	\begin{equation}\label{h function}
	h(\SS_1^\star, \SS_2^\star):=\min_{\tau\in \Iso(\SS_1^\star, \SS_2^\star)}(h(\SS_1^\star, \SS_2^\star, \tau)),
	\end{equation}
	where $\tau$ varies among all bijections from $\SS_1^\star$ to $\SS_2^\star$.
\end{definition}

\begin{definition}\label{minimal bijection}
	We call a bijection from $\SS_1^\star$ to $\SS_2^\star$ \emph{minimal}, if it reaches the minimum in \eqref{h function}. When $\SS_1^\star=\SS_2^\star$, we call it a  \emph{minimal permutation} of $\SS^\star$ and abbreviate $h(\SS^\star, \SS^\star, \bullet)$ (resp. $h(\SS^\star, \SS^\star)$) to $h(\SS^\star,\bullet)$ (resp. $h(\SS^\star)$).
\end{definition}
\begin{remark}
	If $S_i^\star$ for $i =1,2$ belongs $\MM(\Delta)$, we suppress the star from the notation.
\end{remark}


%
%
%

\begin{definition}\label{IHP}
	The \emph{improved Hodge polygon} of $\Delta$, denoted by $\IHP(\Delta)$, is the lower convex hull of the set of points $\Big\{\left(\ell,\min\limits_{\SS^\star\in \mathscr M_\ell(n)} \frac{h(\SS^\star)}{n}\right)\Big\}$,
	where $\mathscr M_\ell(n)$ represents for the set consisting of all multi-subsets of ${\MM(\Delta)^\star}^n$ of cardinality $n\ell$, note $\mathscr M_\ell(1)=\mathscr M_\ell$.
\end{definition}

we shall prove in Proposition~\ref{better form for IHP} later that	$$\min\limits_{\SS^\star\in \mathscr M_\ell(n)}h(\SS^\star)=n\cdot\min\limits_{\SS\in \mathscr M_\ell} h(\SS),$$ and hence the $\IHP(\Delta)$ is independent of $n$. In particular, $\IHP(\Delta)$ is the convex hull of the set of points $$\left(\ell,\min\limits_{\SS\in \mathscr M_\ell} h(\SS)\right).$$

\begin{notation}
	We denote by \[\left[ \begin{array}{cccccccccc}
	m_0 & m_1 &\cdots&m_{\ell-1} \\
	n_0 & n_1 &\cdots&n_{\ell-1}  \end{array} \right]_M\]
	the $\ell\times \ell$-submatrix formed by elements of a matrix $M$ whose row indices belong to $\{m_0,m_1,\dots,m_{\ell-1}\}$ and whose column indices belong to $\{n_0,n_1,\dots,n_{\ell-1}\}$.
\end{notation}
Put $\Delta_f=\Delta$. Recall that we define $\TT_1'$ in Notation~\ref{cone}.
\begin{lemma}\label{L:estimate of Ef(x)}
	We have $e_O(T)=1$ and $v_T(e_{Q}(T))\geq \big\lceil w(Q)\big\rceil$ for all $Q \in \MM(\Delta).$
\end{lemma}

\begin{proof}
	(1) It follows from the definition of $e_O(T)$ in \eqref{E:Ef(x)}.
	
	(2) Let $$\SS(f):=\{P\in \TT_1'\;|\; a_P\neq 0\}=\{Q_1, Q_2, \dots, Q_{t}\},$$
	where $\{a_{P}\}$ is the set of coefficients of $ f(x_1,x_2)$ and $t$ is the cardinality of $\SS(f)$. 
	
	Expanding each $E(\hat{a}_{Q_i} \pi x_1^{(Q_i)_x}x_2^{(Q_i)_y})$ to be a power series in variables $x_1$ and $x_2$, we get $$E_f(x_1,x_2)=\prod_{i=1}^{t}E(\hat{a}_{Q_i} \pi x_1^{(Q_i)_x}x_2^{(Q_i)_y})=\sum_{\vec{j}\in \ZZ_{\geq 0}^{t}}c_{\vec{j}}\prod_{i=1}^{t} (\hat{a}_{Q_i} \pi x_1^{(Q_i)_x}x_2^{(Q_i)_y})^{j_i}, $$
	where $\{\hat{a}_{P}\}$ is the set of coefficients of $\hat f(x_1,x_2)$ and $c_{\vec{j}}$ belongs to $\ZZ_q$. 
	
	It is not hard to get that 
	\[\begin{split}
	e_{Q}(T)=&\sum_{\Big\{\vec{j}\;\Big|\;\sum\limits_{i=1}^{t}j_iQ_i=Q\Big\}}c_{\vec{j}}\prod_{i=1}^{t} (\hat{a}_{Q_i} \pi)^{j_i}\\
	=&\sum_{\Big\{\vec{j}\;\Big|\;\sum\limits_{i=1}^{t}j_iQ_i=Q\Big\}}\Big(c_{\vec{j}}\prod_{i=1}^{t} (\hat{a}_{Q_i})^{j_i}  \pi^{\sum\limits_{i=1}^{t}j_i}\Big).
	\end{split} \]
	Since $w(Q_i)\leq 1$ for each $Q_i\in \SS(f)$, then for each $\vec{j}$ such that $\sum\limits_{i=1}^{t}j_iQ_i=Q$, we have $$v_T\Big(c_{\vec{j}}\prod_{i=1}^{t} (\hat{a}_{Q_i})^{j_i}  \pi^{\sum\limits_{i=1}^{t}j_i}\Big)=\sum\limits_{i=1}^{t}j_i\geq \sum\limits_{i=1}^{t}j_iw(Q_i)=w(Q),$$
	where $T=E(\pi)-1$.
	Therefore, we immediately get that $v_T(e_{Q}(T))\geq w(Q)$. Since $v_T(e_{Q}(T))$ is an integer, we have 
	\[v_T(e_{Q}(T))\geq \big\lceil w(Q)\big\rceil.\qedhere\]
\end{proof}

\begin{notation}\label{Pi}
	We label points in $\MM(\Delta)$ such that $\MM(\Delta)=\{P_1, P_2, \dots\}$.
\end{notation}
\begin{proposition}\label{IHP is below NPchi}
The normalized Newton polygon $\NP(f, T)_C$ lies above $\IHP(\Delta_f)$. 
\end{proposition}

\begin{proof}
	We write $N$ for the standard matrix of $\psi_p\circ E_f$ corresponding to the basis  $$\{x_1^{(P_1)_x}x_2^{(P_1)_y}, x_1^{(P_2)_x}x_2^{(P_2)_y},\cdots\}$$ of the Banach space $\mathbf B$. By \cite[Corollary~3.9]{ren-wan-xiao-yu}, we know that the standard matrix of $\psi^n$ corresponding to the same basis is equal to $\sigma_{\Frob}^{n-1}(N)\circ \sigma_{\Frob}^{n-2}(N)\circ\cdots \circ N.$ Then by \cite[Proposition~4.6]{ren-wan-xiao-yu}, for every $\ell \in \NN$ we have \begin{equation}
	\label{E:expression of char power series}
	\begin{split}
	u_\ell(T)
	&=\sum_{\substack{\{P_{m_{0,0}}, P_{m_{0,1}}, \dots, P_{m_{0,\ell-1}}\}\in \mathscr M_\ell\\\{P_{m_{1,0}}, P_{m_{1,1}}, \dots, P_{m_{1,\ell-1}}\}\in \mathscr M_\ell\\\vdots\\\{P_{m_{{n-1},0}}, P_{m_{{n-1},1}}, \dots, P_{m_{{n-1},\ell-1}}\}\in \mathscr M_\ell}} \det\bigg(\prod\limits_{j=0}^{n-1}\left[ \begin{array}{cccccccccc}
	m_{j+1,0} & m_{j+1,1} &\cdots&m_{j+1,\ell-1} \\
	m_{j,0} & m_{j,1} &\cdots&m_{j,\ell-1}  \end{array} \right]_{\sigma_{\Frob}^{j}(N)}\bigg),     \\
	\end{split}
	\end{equation}
	where $m_{n,i}:=m_{0,i}$ for each $0\leq i\leq \ell-1$.
	
	Then for $\SS_j=\{P_{m_{j,0}}, P_{m_{j,1}}, \dots, P_{m_{j,\ell-1}}\}$, we have
	\begin{equation}\label{explicit of u}
	\begin{split}
	&v_T\Big(\det\big(\prod\limits_{j=0}^{n-1}\left[ \begin{array}{cccccccccc}
	m_{j+1,0} & m_{j+1,1} &\cdots&m_{j+1,\ell-1} \\
	m_{j,0} & m_{j,1} &\cdots&m_{j,\ell-1}  \end{array} \right]_{\sigma_{\Frob}^{j}(N)}\big)\Big)\\
	=&v_T\Big(\prod_{j=0}^{n-1}\sum_{\tau_j=\Iso(\SS_{j+1},\SS_j)}\sgn(\tau_j)\prod_{P\in \SS_{j+1}}\sigma_{\Frob}^j(e_{p\tau_j(P)-P})\Big) \\
	\geq& \sum_{j=0}^{n-1}h(\SS_{j+1},\SS_j)\\
	\geq& h(\biguplus_{j=0}^{n-1} \SS_{j}^\star),
	\end{split}
	\end{equation} 
	where $\SS_n=\SS_0$.
	Therefore, it is easily seen that 
	\begin{equation}\label{xk}
	v_T(u_\ell(T))\geq \min_{\SS^\star\in \mathscr M_\ell(n)} h(\SS^\star).\qedhere
	\end{equation}

\end{proof}

	 \section{Improved Hodge polygon for a triangle $\Delta$}
Recall that $\Delta$ is a triangle with vertices $(0, 0)$, $\mathbf{P_1}:=(a_1,b_1)$ and $\mathbf{P_2}:=(a_2,b_2)$ and as we defined in Notation~\ref{cone}, $\mathbbm x_k = \mathbbm x_k(\Delta)$ (resp. $\mathbbm x'_k = \mathbbm x'_k(\Delta)$) is the number of lattice points in $\MM(\Delta)$ whose weight is strictly less than (resp. less than or equal to) $k$. For the rest of this paper, we restrict $p$ to be a prime satisfying $$p\nmid a_2b_1-a_1b_2\quad \textrm{and}\quad p>\frac{2(b_1a_2-b_2a_1)}{\gcd(a_1-a_2,b_1-b_2)}+1.$$
The goal of this section is to show that if $\NP(f, \chi)_C$ (See Definition~\ref{definition for NP chi}) and $\IHP(\Delta)$ coincide at a certain point, then they will coincide at infinitely many points. More precisely, we have the following.
\begin{theorem}\label{main}
Let $f(x_1, x_2)$ be a two-variable polynomial with convex hull $\Delta$. Suppose  that there exists a nontrivial finite character $\chi_{1}:\Zp\to \CC_p^\times$ and an integer $k_1>0$ such that $\NP(f, \chi_{1})_C$ coincides with $\IHP(\Delta)$ at $x=\mathbbm{x}_{k_1}$. Then for any finite character $\chi$ and positive integer $k$, $\IHP(\Delta)$ and $\NP(f, \chi)_C$ coincide at $\mathbbm{x}_k+i_k$ for all $0\leq i_k\leq\mathbbm x_k'-\mathbbm x_k$. 

Moreover, the leading coefficients  $ u_{\mathbbm{x}_k, nh(\TT_k)}$ $(\mathrm{resp. }\ u_{\mathbbm{x}'_k, nh(\TT'_k)})$ of the $\mathbbm x_k$-th $(\mathrm{resp. }\ \mathbbm x'_k\textrm{-th})$ terms of the characteristic power series (see (2.2) for precise definition) are $\ZZ_p$-units.

\end{theorem}
The proof of this theorem will occupy the rest of this section.

\begin{notation}
		 Without loss of generality, we can assume that $a_2b_1-a_1b_2>0$.
		  We call the parallelogram with vertices $O, \mathbf P_1, \mathbf P_2\ \textrm{and}\ \mathbf P_1+\mathbf P_2$ (excluding the upper and right sides) the \emph{fundamental parallelogram} of $\Delta$, and denote it by $\square_{\Delta}$, i.e. the shadow region in Figure~\ref{fig: triangle}.
		  
		  We put $\square_{\Delta}^\Int$ to be the set of lattice points in $\square_{\Delta}$, which contains $a_2b_1-a_1b_2$ points.
		  Let $\Lambda$ be the lattice generated by $\bf{P_1}$ and $\bf{P_2}$. For each point $P$ in $\Cone(\Delta)$, we write $P\%$ for its residue in $\square_{\Delta}$ modulo $\Lambda$. 
\end{notation}
		\begin{figure}[h]
		\centering
		\begin{tikzpicture}[scale=0.8]
		\filldraw[color=blue!20] 
		(0,0) -- (3,6) -- (9,8)--(6,2)--cycle;
		\draw[->] (-0,0) -- (10,0) node[right] {$x$};
		\draw[->] (0,-0) -- (0,10) node[above] {$y$};
		\draw[dashed] (3,6) -- (9,8)--(6,2) node[right] {};
		\draw(6,1.8)node[below]{$\bf{P_2}:=(a_2,b_2)$};
		\draw(3,6)node[left]{$\bf{P_1}:=(a_1,b_1)$};
		\end{tikzpicture}
		\caption{The fundamental parallelogram.} \label{fig: triangle}
	\end{figure}
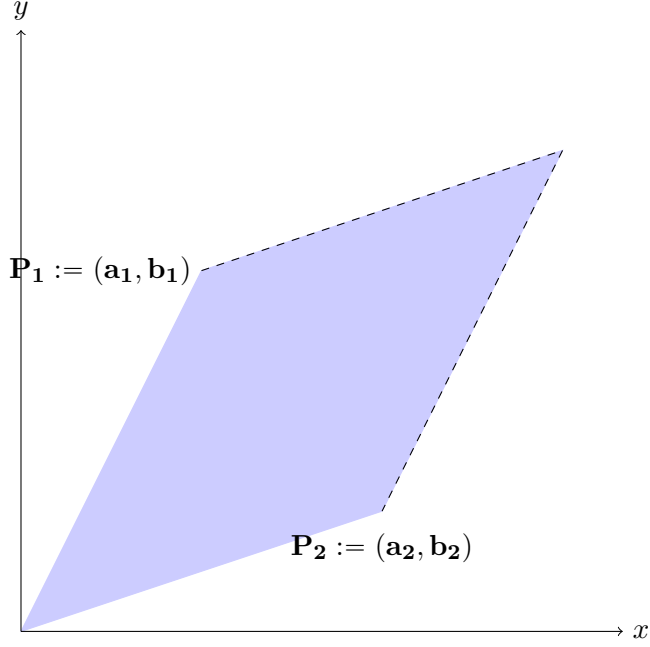

		\begin{lemma}\label{counting points in a parallelogram}
		 The map  \begin{eqnarray*}
			\eta:&\square_{\Delta}^\Int\to \square_{\Delta}^\Int\\
			&P\mapsto (pP)\% 
		\end{eqnarray*}
	 is a permutation.
	\end{lemma}
	\begin{proof}
	since $p$ and $a_2b_1-a_1b_2$ are coprime, there exist integers $p'$ and $n_1$ such that $pp'-1= (a_2b_1-a_1b_2)n_1$. 
		For a point $P=(P_x,P_y)$, we have $$(pp'-1)P=n_1(b_2P_x-a_2P_y)\mathbf{P_1}+n_1(-b_1P_x+a_1P_y)\mathbf{P_2}\in \Lambda.$$
		It implies that composite $$ \square_{\Delta}^\Int\xrightarrow{P\mapsto pP\%}\square_{\Delta}^\Int\xrightarrow{P\mapsto p'P\%}\square_{\Delta}^\Int$$ is the identity map. 
		\end{proof}


The key to proving Theorem~\ref{main} is to gain precise control of the improved Hodge polygon in Proposition~\ref{IHP is below NPchi}. In \cite{wan}, Wan made use of the following coarser estimate of this Hodge polygon, for each multiset $\SS^\star$ of ${\MM(\Delta)^\star}^\infty$ we have $$h_1(\SS^\star)\geq h_1(\SS^\star):=(p-1)\sum_{P\in \SS^\star}w(P).$$
It is however important for our method to understand the difference between $h_1(\SS^\star)$ and $h(\SS^\star)$ (or more generally $h(\SS^\star, \tau)$.
			
	 	 For each $r\in \RR$, we put $R(r):=\lceil r\rceil -r$; and for a permutation $\tau$ of $\SS^\star$, we set \begin{equation}\label{U star}
	 	 U^\star(\SS^\star, \tau):=\Big\{R\left(w\left(p\tau(P)-P\right)\right)\;|\;P\in \SS^\star\Big\}^\star,
	 	 \end{equation}	
	 	 \begin{equation}\label{U star1}
	 	  U^\star(\SS^\star, \tau)^{\leq a}:=\Big\{r\in U^\star(\SS^\star, \tau)\;|\;r\leq a\Big\}^\star\quad\textrm{and}\quad  U^\star(\SS^\star, \tau)^{< a}:=\Big\{r\in U^\star(\SS^\star, \tau)\;|\;r< a\Big\}^\star
	 	 \end{equation}	
	 	 to measure the distance of these weights to the next integer values.
	 	 
	 	 Write
	 	 \begin{equation}\label{definition of h1 and h2}
	 	 h_2(\SS^\star,\tau):=\sum\limits_{r\in U^\star(\SS^\star,\tau)} r\quad \textrm{and}\quad h_2(\SS^\star)=\min_{\tau\in \Iso(\SS^\star)}\Big\{h_2(\SS^\star, \tau)\Big\}
	 	\end{equation}
	where $\Iso(\SS^\star)$ is set consisting all permutations of $\SS^\star$ (see Notation~\ref{Iso and Prem}). 
	 	\begin{lemma}\label{h=h1+h2}
	 		We have
	 		\begin{equation}\label{equivalent equation for h} 
	 		h(\SS^\star,\tau)=h_1(\SS^\star)+h_2(\SS^\star,\tau) \quad \textrm{and}\quad h(\SS^\star)=h_1(\SS^\star)+h_2(\SS^\star).
	 		\end{equation}
	 	\end{lemma}

	 	\begin{proof}
	 		By the definition of $h(\SS^\star, \tau)$ in \eqref{defintion of h SS}, we have \[\begin{split}
	 		h(\SS^\star, \tau)&=\sum_{P\in \SS^\star} \big\lceil w\big(p\tau(P)-P\big)\big\rceil\\
	 		& =\sum_{P\in \SS^\star} w\big(p\tau(P)-P\big)+\sum_{P\in \SS^\star}\Big\{\big\lceil w\big(p\tau(P)-P\big)\big\rceil- w\big(p\tau(P)-P\big)\Big\}\\
	 		&=\sum_{P\in \SS^\star} pw\big(\tau(P)\big)-\sum_{P\in \SS^\star} w(P)+ \sum_{P\in \SS^\star}R\big( w\big(p\tau(P)-P\big)\big)\\
	 		&=(p-1)\sum_{P\in \SS^\star}w(P)+h_2(\SS^\star, \tau)\\
	 		&=h_1(\SS^\star)+h_2(\SS^\star, \tau).
	 		\end{split}\]
	 		
	 		Taking the minimal over all $\tau \in \Iso(\SS^\star)$ implies \[h(\SS^\star)=h_1(\SS^\star)+h_2(\SS^\star).\qedhere\]
	 	\end{proof}

\begin{lemma}\label{U star and h2}
	For any two permutations $\tau_1, \tau_2$ of $\SS^\star$, if $U^\star(\SS^\star, \tau_1)=U^\star(\SS^\star,\tau_2),$ then  $$h_2(\SS^\star,\tau_1)=h_2(\SS^\star,\tau_2)\quad \textrm{and}\quad h(\SS^\star,\tau_1)=h(\SS^\star,\tau_2).$$
\end{lemma}
		\begin{proof}
			This lemma follows from the definition of $h_2$ and Lemma~\ref{h=h1+h2}.
		\end{proof}

	\begin{lemma}\label{set to h2}
		Suppose $\SS^\star_1, \SS^\star_2\subset {\MM(\Delta)^\star}^\infty$ satisfy 	\begin{equation}\label{same multisets}
		\Big\{w(P)\;|\; P\in\SS^\star_1\Big\}^\star=\Big\{w(P)\;|\; P\in\SS^\star_2\Big\}^\star
		\end{equation} as multisets. Then we have 
		\begin{eqnarray}\label{same h}
		h_2(\SS^\star_1)=h_2(\SS^\star_2)\quad \textrm{and}\quad h(\SS^\star_1)=h(\SS^\star_2).
		\end{eqnarray}
	\end{lemma}
	\begin{proof}
	Let $\xi: \SS^\star_1\to\SS^\star_2$ be a bijection such that the induced map of weights from $\Big\{w(P)\;|\; P\in\SS^\star_1\Big\}^\star$ to $\Big\{w(P)\;|\; P\in\SS^\star_2\Big\}^\star$ realizes the equality~\eqref{same multisets}. Then $\xi$ induces a bijection from $\Iso(\SS^\star_1)$ to $\Iso(\SS^\star_2)$. Moreover, we have$$ h_2(\SS^\star_1, \tau) = h_2(\SS^\star_2, \xi^{-1}\tau \xi).$$ 
		
	We immediately get the first equality in \eqref{same h}. Combining it with Lemma~\ref{h=h1+h2}, we can easily check the second equality.
	\end{proof}

	We prove Theorem~\ref{main} in two steps.
	
	\noindent\textbf{Step I.}
	 Recall that a permutation $\tau$ of $\SS^\star$ that achieves the minimum of $h(\SS^\star)$ or equivalently the minimum of $h_2(\SS^\star)$ is called minimal.
	The first core result in this section is Proposition~\ref{Thm 2}. It shows for a given $\SS^\star\subset{\MM(\Delta)^\star}^\infty$ how to construct an explicit minimal permutation $\overline{\tau}$ of $\SS^\star$. 
	

	First, we construct a minimal permutation inductively for a general subset $\SS^\star$ of ${\MM(\Delta)^\star}^\infty$ as follows.
	\begin{construction}\label{construction for overline sigma}
		
		We choose a pair of points $(P_0, Q_0)$ in $\SS^\star\times \SS^\star$ such that $R(w(pQ_0-P_0))$ reaches the minimum among all pairs $(P, Q)$ in $\SS^\star\times \SS^\star$. Define $\overline{\tau}(P_0):=Q_0$. 
		
		Then we take out $P_0$ from the first $\SS^\star$ and $Q_0$ from the second $\SS^\star$. We pick another pair of points $(P_1, Q_1)$ from $\SS^\star\backslash\{P_0\}\times \SS^\star\backslash\{Q_0\}$ such that $R(w(pQ_1-P_1))$ reaches the minimum among all pairs $(P, Q)$ in $\SS^\star\backslash\{P_0\}\times \SS^\star\backslash\{Q_0\}$, and define $\overline{\tau}(P_1):=Q_1$. Similarly, we pick a ``minimal'' pair of points $(P_2,Q_2)$ from $\SS^\star\backslash\{P_0, P_1\}\times \SS^\star\backslash\{Q_0, Q_1\}$. Define $\overline{\tau}(P_2):=Q_2$.
		Iterating this process defines $\overline{\tau}$. 
		
	\end{construction}
	
%
%
	\begin{lemma}\label{can be odered better}
		Let $\overline \tau$ be a minimal permutation constructed in Construction~\ref{construction for overline sigma}, and let $\tau$ be an arbitrary permutation of $\SS^\star$. Suppose  
		\begin{equation}\label{Ustar equal}
		U^\star(\SS^\star, \overline \tau)^{<r_0}=U^\star(\SS^\star, \tau)^{<r_0}\quad \textrm{for some rational number}\ r_0.
		\end{equation}
		Then by taking finite times of the following operations:
		\begin{enumerate}
			\item swapping the images of two points of the same weight; and 
			\item  swapping the preimages of two points of the same weight,
		\end{enumerate}
		we obtain a permultation $\tau'$ from $\tau$ such that $\tau'(P)=\overline\tau(P)$ for all $P\in \SS^\star$ satisfying $R\big(w(p\tau'(P)-P)\big)\leq r_0$ and $U^\star(\SS^\star, \tau') = U^\star(\SS^\star, \tau)$. 
		
		In particular, $U^\star (\SS^\star, \overline \tau)$ defined in \eqref{U star} is independent of the choice made in Construction~\ref{construction for overline sigma}.
	\end{lemma}
	\begin{proof}
		Assuming \eqref{Ustar equal} holds for every $r\leq r_0$. Then we induce a permutation $\widetilde\tau$ from $\tau$ by taking finite times of operations (1) and (2) such that $\widetilde\tau(P)=\overline\tau(P)$ for all $P\in \SS^\star$ satisfying $R\big(w(p\widetilde\tau(P)-P)\big)< r_0$ and $U^\star(\SS^\star, \widetilde\tau) = U^\star(\SS^\star, \tau).$ It is not hard to check that it is enough for us to show this lemma works for $\widetilde{\tau}$. In other words, we can assume that \begin{equation}\label{extra assumption}
		\tau(P)=\overline\tau(P) \textrm{ for each } P\in \SS^\star \textrm{ satisfying }  R\big(w(p\overline\tau(P)-P)\big)< r_0.
		\end{equation}
		Suppose that $\tau(P)=\overline\tau(P)$ for all $P\in \SS^\star$ satisfying $R\big(w(p\tau(P)-P)\big)\leq r_0$. Then we are done. Otherwise there exists a point $P$ in $\SS^\star$ such that $R\big(w(p\tau(P)-P)=r_0$, but $\tau(P)\neq\overline\tau(P)$.  By Construction~\ref{construction for overline sigma}, it is not hard to see that at least one of the following cases will happen:
		\begin{itemize}
			\item[(a)] $R\big(w(p\overline\tau(P)-P\big)=r_0$.
			\item[(b)] There exists a point $P_1$ in $\SS^\star$ such that $w(P)=w(P_1)$ and $\overline\tau(P_1)=\tau (P)$.
		\end{itemize}
		When case (a) happens, we put $Q_1=\overline\tau(P)$ and define $\tau_1:\SS^\star\to \SS^\star$ to be the same permutation as $\tau$ except we swap the preimages of $Q_1$ and $\tau(P)$. 
		
		Otherwise, we define $\tau_1$ to be the same permutation as $\tau$ except we swap the images of $P$ and $P_1$, where $P_1$ is defined in (b). 
		
		By \eqref{extra assumption}, we know that either $$R\big(w(pQ_1-\tau^{-1}(Q_1))\big)\geq r_0\quad \textrm{or}\quad R\big(w(p\tau(P_1)-P_1)\big)\geq r_0,$$ 
		which implies that $\tau_1 $ also satisfies \eqref{extra assumption} and $$\#\Big\{P\in U^\star(\SS^\star, \tau_1)^{=r_0}\;\big|\; \tau_1(P)=\overline\tau(P)\Big\}\geq \#\Big\{P\in U^\star(\SS^\star,  \tau)^{=r_0}\;\big|\; \tau_1(P)=\overline\tau(P)\Big\}+1.$$ If $\tau_1$ does not satisfy the wanted property, then we run the same argument with $\tau_1$ in place of $\tau$ to obtain another permutation $\tau_2$. Iterating this process eventually gives us a permutation $\tau'$ of $\SS^\star$. It is easy to check that $\tau'(P)=\overline\tau(P)$ for all $P\in \SS^\star$ satisfying $R\big(w(p\tau'(P)-P)\big)\leq r_0$. 
		Since both operations (1), (2) do not change the set $U^\star(\SS^\star, \tau)$, we have $U^\star(\SS^\star, \tau') = U^\star(\SS^\star, \tau).$    
		
		We now prove the last statement of the lemma. Let $\overline \tau_1$ and $\overline \tau_2$ be two permutations constructed in Construction~\ref{construction for overline sigma} by different choices of pairs of points. Suppose that $U^\star (\SS^\star, \overline \tau)\neq U^\star (\SS^\star, \overline \tau)$. Then there is a rational number $r_0$, such that   
		$$U^\star(\SS^\star, \overline \tau_1)^{<r_0}=U^\star(\SS^\star, \overline \tau_2)^{<r_0}\quad \textrm{and}\quad U^\star(\SS^\star, \overline \tau_1)^{\leq r_0}\neq U^\star(\SS^\star, \overline \tau_2)^{\leq r_0}.$$ Without loss of generality, we assume that $$U^\star(\SS^\star, \overline \tau_1)^{\leq r_0}\subsetneq U^\star(\SS^\star, \overline \tau_2)^{\leq r_0}.$$
		From the argument in this lemma above, we know that there exists a permutation $\tau'$ of $\SS^\star$ such that $\tau'(P)=\overline\tau_1(P)$ for all $P\in \SS^\star$ satisfying $R\big(w(p\tau'(P)-P)\big)\leq r_0$ and $$U^\star(\SS^\star, \overline \tau_2)=U^\star(\SS^\star, \tau').$$
		Therefore,  we have $$U^\star(\SS^\star, \overline \tau_1)^{\leq r_0}=U^\star(\SS^\star, \tau')^{\leq r_0}=U^\star(\SS^\star, \overline\tau_2)^{\leq r_0},$$ a contradiction.
	\end{proof}
	
	\begin{corollary}
		For $\tau$ and $\tau'$ as in the last lemma, we have  $h_2(\SS^\star,\tau')=h_2(\SS^\star,\tau)$.
	\end{corollary}
 
	\begin{proof}
		Since $U^\star(\SS^\star, \tau') = U^\star(\SS^\star, \tau)$, then it follows directly from the definition of $h_2$ in \eqref{definition of h1 and h2}.
	\end{proof}
	\begin{proposition}\label{Thm 2}
		The permutation $\overline{\tau}$ in Construction~\ref{construction for overline sigma} is a minimal permutation of $\SS^\star$, i.e. $$h(\SS^\star,\overline{\tau})=h(\SS^\star).$$
	\end{proposition}

	\begin{proof}
		By Lemma~\ref{h=h1+h2}, it is enough to prove that $\overline{\tau}$ minimizes $h_2(\SS^\star, \bullet)$ among all permutations of $\SS^\star$.
		
		Assume that $\tau$ is a minimal permutation of $\SS^\star$. If $U^\star(\SS^\star, \overline{\tau})= U^\star(\SS^\star,\tau)$, we are done by Lemma~\ref{U star and h2}. Otherwise we shall construct below a finite sequence of permutations $(\tau_0=\tau, \tau_1, \dots, \tau_m)$ satisfying 
		\begin{enumerate}[label={(\arabic*)}]
			\item for each $0\leq i\leq m$, we have $h_2(\SS^\star, \tau_i)=h_2(\SS^\star).$
			\item $U^\star(\SS^\star, \tau_m)=U^\star(\SS^\star, \overline{\tau}).$
		\end{enumerate}
		Combining Property~(2) with Lemma~\ref{U star and h2}, we would deduce
		$$h(\SS^\star,\tau_m)=h(\SS^\star,\overline{\tau}),$$ which completes the proof.
		
		Now we come to the inductive construction of the sequence $(\tau_0=\tau, \tau_1, \dots, \tau_m)$. 	We take induction on $i$. First, we put $\tau_0=\tau$, where $\tau$ is a minimal permutation of $\SS$. Hence, it satisfies $h_2(\SS^\star, \tau_0)=h_2(\SS^\star).$	
		
		Suppose that we have defined $\tau_i$. If $U^\star(\SS^\star, \tau_i)=U^\star(\SS^\star, \overline{\tau})$, we terminate this induction. Otherwise let $t_i$ be the smallest number in the multiset $U^\star(\SS^\star, \overline{\tau})\backslash U^\star(\SS^\star, \tau_i)$. 
		Let $\tau_i':\SS^\star\to \SS^\star$  be the permutation constructed from $\tau_i$ in Lemma~\ref{can be odered better}. It is not hard to check that 
		\begin{itemize}
			\item $h_2(\SS^\star,\tau_i)=h_2(\SS^\star,\tau'_i)$.
			\item 	$\tau_i'(P)=\overline\tau(P)$ holds for each $P\in \SS^\star$ satisfying $w(P-\tau_i'(P))\leq t_i$.
			\item $U^\star(\SS^\star, \overline \tau)^{\leq t_i}\supsetneq U^\star(\SS^\star, \tau_i')^{\leq t_i}$.
		\end{itemize}
		Therefore, there exists a point $P_i$ in $\SS^\star$ with 
		$$Q_i=\overline{\tau}(P_i),\quad \tau_i'(P'_i)=Q_i\quad \textrm{and}\quad \overline\tau(P'_i)=Q'_i$$ such that 
		\begin{itemize}
			\item $R(w(pQ_i-P_i))=t_i$,
			\item $R(w(pQ'_i-P'_i))\geq t_i$,
			\item $R(w(pQ_i-P'_i))> t_i$, and
			\item 	$R(w(pQ'_i-P_i))> t_i$.
		\end{itemize}
		We define $\tau_{i+1}$ to be the same permutation as $\tau'_i$ except we  swap the images of $P_i$ and $P_i'$. This process gives us a sequence with elements from $\Iso(\SS^\star)$, whose length, say $m$, is less than or equal to $\#(\SS^\star)$.
		
		Then we check $h_2(\SS^\star,\tau_{i+1})= h_2(\SS^\star,\tau_i)$. From the induction we know that $h_2(\SS^\star, \tau_i)=h_2(\SS^\star)$. Hence, we ha
		$$h_2(\SS^\star, \tau_{i+1})\geq h_2(\SS^\star, \tau_i).$$ 
		
		On the other hand, from the definition of $h_2$, we have
		\begin{equation}\label{difference of h2}
		\begin{split}
		h_2(\SS^\star,\tau_{i+1})-h_2(\SS^\star,\tau_i)=&h_2(\SS^\star,\tau_{i+1})-h_2(\SS^\star,\tau_i')\\
		=&R\big(w(p\tau_{i+1}(P_i)-P_i)\big)+R(w\big(p\tau_{i+1}(P'_i)-P'_i)\big)\\
		-&R\big(w(p\tau_i'(P'_i)-P'_i)\big)-R\big(w(p\tau_i'(P_i)-P_i)\big)\\
		<& t_i+R\big(w(p\tau_{i+1}(P'_i)-P'_i)\big)-t_i-R\big(w(p\tau_i(P_i)-P_i)\big)\\
		\leq &1-R\big(w(p\tau_i(P_i)+P_i)\big)\\
		\leq &1.
		\end{split}
		\end{equation}
		From the linearity of $w$, we have   $$w\big(p\tau_{i+1}(P_i)-P_i\big)+w\big(p\tau_{i+1}(P'_i)-P'_i\big)
		=w\big(p\tau_i'(P'_i)-P'_i\big)+w\big(p\tau_i'(P_i)-P_i\big).$$ It implies that $h_2(\SS^\star,\tau_{i+1})-h_2(\SS^\star,\tau_i)$ is an integer. Combining it with \eqref{difference of h2}, we have  $$h_2(\SS^\star, \tau_{i+1})\leq h_2(\SS^\star, \tau_i).$$

		Combining these inequalities, we obtain \begin{equation}\label{equal for h}
		h_2(\SS^\star,\tau_{i+1})=h_2(\SS^\star,\tau_i)=h_2(\SS^\star).
		\end{equation}

		Finally, from the termination condition of this induction, we know that \[U^\star(\SS^\star, \tau_m)=U^\star(\SS^\star, \overline{\tau}).\qedhere\]
	\end{proof}

	\begin{remark}
		Notice that this result crucially depends on the linearity of $w$ as in \eqref{linear}. For general polytopes, we have a similar definition for the weight of a point, however, it is no longer linear. We will address this case in a future work.
	\end{remark}
	For any subset $\SS$ of $\MM(\Delta)$, we write $\SS+P$ for the shift of $\SS$ by $P$.

	\begin{corollary}\label{linear1}
		Let $\SS_2$ be a subset of $\MM(\Delta)$, and let $Q_1, Q_2, \dots, Q_k$ be points of $\MM(\Delta)$ with integer weights. Put $$\SS^\star_1=\SS_2^\star\uplus \biguplus\limits_{i=1}^{k}(\square_{\Delta}^\Int+Q_i)^\star,$$ then $h_2(\SS^\star_1)= h_2(\SS^\star_2)$.
	\end{corollary}
	\begin{proof}
		Recall that for each point $P\in \MM(\Delta)$, we denote by $P\%$ its residue in $\square_\Delta$. Lemma~\ref{counting points in a parallelogram} defines a permutation of $\square_{\Delta}^\Int$ such that $\eta(P)=(pP)\%$.  
		We write \begin{eqnarray*}
			\eta_i:\square_{\Delta}^\Int+Q_i\to \square_{\Delta}^\Int+Q_i\\
			P\mapsto \eta^{-1}(P-Q_i)+Q_i.
		\end{eqnarray*}
		It is easy to check that $R(p\eta_i(P)-P)=0$ for all point $P$ in $\square_{\Delta}^\Int+Q_i.$
		
		Then we apply Construction~\ref{construction for overline sigma} to $\SS_2$ and get a minimal permutation $\tau_{\SS_2}$ of $\SS_2$. It is not hard to see that $\tau_{\SS_2}, \eta_1,\dots,\eta_k$ together construct a permutation $\tau$ of $\SS^\star_1$ which agrees with Construction~\ref{construction for overline sigma}. Therefore, we have 
		\begin{align*}
		h_2(\SS^\star_1)=&h_2(\SS_2,	\tau_{\SS_2})+\sum_{i=1}^kh_2(\square_{\Delta}^\Int+Q_i, \eta_{i})\\
		=&h_2(\SS_2,	\tau_{\SS_2})\\
		=&h_2(\SS_2).\qedhere
		\end{align*} 
	\end{proof}
	\begin{corollary}\label{power}
		Let $\SS_2$ be a subset of $\MM(\Delta)$, and let $Q_1, Q_2, \dots, Q_k$ be points of $\MM(\Delta)$ of integer weights. Put
		$\SS^\star_1=\biguplus\limits_{i=1}^{k} (\SS_2+Q_i)$. Then we have $h_2(\SS^\star_1)= kh_2(\SS_2)$.
	\end{corollary}
	\begin{proof}

		Let $\tau_{0}$ be the minimal permutation of $\SS_2$ constructed in Construction~\ref{construction for overline sigma}. For each $1\leq i\leq k$,  put \begin{eqnarray*}
			\tau_i:\SS_2+Q_i\to \SS_2+Q_i\\
			P\mapsto \tau_0(P-Q_i)+Q_i.
		\end{eqnarray*}
		It is not hard to see that $\tau_1,\tau_2,\dots,\tau_k$ together induce a permutation $\tau$ of $\SS^\star_1$ which agrees with Construction~\ref{construction for overline sigma} and for each $1\leq i\leq k$ we have $$h_2(\SS_2+Q_i, \tau_{i})=h_2(\SS_2, \tau_0)=h_2(\SS_2).$$ Therefore, 
		\begin{align*}
		h_2(\SS^\star_1)=&h_2(\SS^\star_1, \tau)\\ 
		=&\sum_{i=1}^kh_2(\SS_2+Q_i, \tau_{i})\\
		=&kh_2(\SS_2).\qedhere
		\end{align*} 
	\end{proof}
\begin{lemma}\label{linear for h2}
	we have
	$$h_2(\TT_k)=kh_2(\TT_1).$$
\end{lemma}
\begin{proof}
	We can decompose $\TT_k$ into a disjoint union of sets as follows:
	\begin{equation}\label{Sn decomposition}
	\TT_k=\bigsqcup\limits_{i=0}^{k-1}(\TT_1+(k-1-i)\textbf{P}_\textbf{2}+i\textbf{P}_\textbf{1})\sqcup \bigsqcup\limits_{i=0}^{k-2}\bigsqcup\limits_{j=0}^{i}(\square_{\Delta}^\Int+i\textbf{P}_\textbf{1}+j\textbf{P}_\textbf{2}).
	\end{equation}
	
	Applying Corollary~\ref{linear1} and \ref{power} to \eqref{Sn decomposition}, we complete the proof of this lemma. 
\end{proof}

\begin{lemma}\label{xk and hk}
	Let $l_1+2$ (resp. $l_2+2$)  represent the number of lattice points on closed segment $O\mathbf{P_1}$ (resp. $O\mathbf{P_2}$). For each $k>0$, we have
	\begin{itemize}
		\item[$(1)$] $\mathbbm{x}_k=k\mathbbm{x}_1+\frac{k(k-1)}{2}(a_2b_1-a_1b_2).$
		\item[$(2)$] $h(\TT_k)=(p-1)\sum\limits_{i=0}^{k-2}\big[(a_2b_1-a_1b_2)(i+1)-\tfrac{1}{2}(l_1+l_2)\big](i+1)
		+k\big[h(\TT_1)+(p-1)(k-1)\mathbbm{x}_1\big].$
	\end{itemize}
\end{lemma} 
\begin{proof}
	(1) Since there are totally $a_2b_1-a_1b_2$ points in $\square_{\Delta}^\Int$, (1) follows directly from \eqref{Sn decomposition} above. 
	
	(2) A tautological computation shows that
	\[\begin{split}
	h_1(\TT_k)&=k\big[h_1(\TT_1)+(p-1)(k-1)\mathbbm{x}_1\big]\\
	&+(p-1)\sum_{i=0}^{k-2}\big[a_2b_1-a_1b_2-\tfrac{1}{2}(l_1+l_2)-1+i(a_2b_1-a_1b_2)\big](i+1).
	\end{split}\]
	For each $k\geq 1$,  we know from Lemma~\ref{linear for h2} that
	\[\begin{split}
	h(\TT_k)=&h_1(\TT_k)+h_2(\TT_k)\\
	=&k\big[h_1(\TT_1)+(p-1)(k-1)\mathbbm{x}_1\big]+kh_2(\TT_1)\\
	+&(p-1)\sum_{i=0}^{k-2}\big[(a_2b_1-a_1b_2)(i+1)-\tfrac{1}{2}(l_1+l_2)-1\big](i+1).\\
	\end{split}\] 
	Combining it with $h(\TT_1)=h_1(\TT_1)+h_2(\TT_1)$, we complete the proof.
\end{proof}

	\noindent\textbf{Step II.}
	The following proposition is the second core result of studying the improved Hodge polygon $\IHP(\Delta)$ at $x=\mathbbm{x}_k$. 

\begin{proposition}\label{better form for IHP}
	We have \begin{equation}\label{u times}
	\min\limits_{\SS^\star\in \mathscr M_\ell(n)}h(\SS^\star)=n\cdot\min\limits_{\SS\in \mathscr M_\ell} h(\SS).
	\end{equation}
	Therefore, we give a simpler expression of $\IHP(\Delta)$ as the lower convex hull of the set of points $$\left(\ell,\min\limits_{\SS\in \mathscr M_\ell} h(\SS)\right),$$ which is independent of $n$.
\end{proposition}
We will prove this proposition after two lemmas.

\begin{lemma}\label{basic lemma}
	For any two distinct points $P, Q\in \MM(\Delta)$ if $w(P)\neq w(Q)$, then $$|w(P-Q)|\geq\frac{\gcd(a_1-a_2,b_1-b_2)}{b_1a_2-b_2a_1}.$$
\end{lemma}
\begin{proof}
	It is easily known that $$w((1,0))=\frac{a_2-a_1}{b_1a_2-b_2a_1}\quad \textrm{and}\quad w((0,1))=\frac{b_1-b_2}{b_1a_2-b_2a_1}.$$ 
	Since each point in $\ZZ^2$ is a linear combination of $(1,0)$ and $(0,1)$, this lemma follows from the linearity of $w$
\end{proof}

\begin{lemma}\label{S}
	Let $\MM'$ be a subset of ${\MM(\Delta)^\star}^\infty$ and let $\mathscr S_\ell$ be the set consisting of all subsets of $\MM'$ of cardinality $\ell$. Choose a multiset $\SS_\textrm{min}^\star\in \mathscr S_\ell$ such that $$\sum_{P\in\SS_\textrm{min}^\star}w(P)=\min_{\SS^\star\in \mathscr S_\ell}\Big(\sum_{P\in\SS^\star}w(P)\Big).$$
	Suppose that $p>\frac{2(b_1a_2-b_2a_1)}{\gcd(a_1-a_2,b_1-b_2)}+1$. Then if $\widetilde{\SS}^\star\in \mathscr S_\ell$ satisfy
	\begin{equation}\label{a best S0}
	\sum_{P\in\widetilde{\SS}^\star}w(P)=\sum_{P\in\SS_\textrm{min}^\star}w(P)\quad(\textrm{resp.}\ \sum_{P\in\SS_\textrm{min}^\star}w(P)>\min_{\SS^\star\in \mathscr S_\ell}\Big(\sum_{P\in\SS^\star}w(P)\Big)),
	\end{equation}
	we have \begin{equation}\label{minimal of SS0}
	h(\widetilde{\SS}^\star)=h(\SS_\textrm{min}^\star)\quad(\textrm{resp.}\ h(\widetilde{\SS}^\star)>h(\SS_\textrm{min}^\star)).
	\end{equation}
	
	In other words, the minimal $ h(\SS^\star)$ is achieved by exactly those $\SS^\star$ for which the sum of weights of $\SS^\star$ is minimal.
	
\end{lemma}
\begin{proof}

	For a subset $\widetilde{\SS}^\star\in \mathscr S_\ell$, if $\Big\{w(P)\;|\; P\in\widetilde{\SS}^\star\Big\}^\star=\Big\{w(P)\;|\; P\in\SS_\textrm{min}^\star\Big\}^\star$, then by Lemma~\ref{set to h2}, we know $$h(\widetilde{\SS}^\star)=h(\SS_\textrm{min}^\star).$$
	Otherwise we construct a sequence  $(\SS^\star_0=\widetilde{\SS}^\star,\SS^\star_1, \dots,\SS^\star_m)$ in $\mathscr S_\ell$ of length less than or equal to $\ell$ such that 
	\begin{itemize}
		\item for any $0\leq i\leq m-1$, $h(\SS_i)>h(\SS_{i+1})$, and
		\item $\Big\{w(P)\;|\; P\in\SS_m^\star\Big\}^\star=\Big\{w(P)\;|\; P\in\SS_\textrm{min}^\star\Big\}^\star$.
	\end{itemize}
%
%
%
	The following is the construction:
	
	Assume that we have constructed $\SS^\star_i$. If $\Big\{w(P)\;|\; P\in\SS^\star_i\Big\}^\star=\Big\{w(P)\;|\; P\in\SS_\textrm{min}^\star\Big\}^\star$, then we stop. Otherwise there exists a rational number $t_i$ such that $${\Big\{w(P)<t_i\;|\; P\in\SS^\star_i\Big\}^\star}={\Big\{w(P)<t_i\;|\; P\in\SS_\textrm{min}^\star\Big\}^\star}$$ and  $${\Big\{w(P)\leq t_i\;|\; P\in\SS^\star_i\Big\}^\star} \subsetneq{\Big\{w(P)\leq t_i\;|\; P\in\SS_\textrm{min}^\star\Big\}^\star}.$$
	Then there exist points $P_i\in \SS_\textrm{min}^\star-\SS^\star_i$ and  $Q_i \in \SS^\star_i-\SS_\textrm{min}^\star$ such that $$w(P_i)=t_i< w(Q_i).$$
	Put $\SS^\star_{i+1}$ to be the set induced from $\SS^\star_i$ by simply substituting $Q_i$ with $P_i$.
	Then we get a sequence $(\SS^\star_0=\SS^\star,\SS^\star_1, \dots,\SS^\star_m)$ in $\mathscr S_\ell$ of length, say $m$, less than or equal to $\ell$, which satisfies the following conditions.
	\begin{enumerate}[label={(\arabic*)}]
		\item $\SS^\star_0=\widetilde{\SS}^\star$.
		\item $\Big\{w(P)\;|\; P\in\SS^\star_m\Big\}^\star=\Big\{w(P)\;|\; P\in\SS_\textrm{min}^\star\Big\}^\star$.
		\item For each $0\leq i\leq m-1$, we have $\SS^\star_{i+1}-\SS^\star_i=\{P_i\}$ and $\SS^\star_{i}-\SS^\star_{i+1}=\{Q_{i}\}$.
		\item The points above satisfy $w(P_i)<w(Q_i)$.
	\end{enumerate}

	From Lemma~\ref{set to h2}, we know that (2) implies that $$h(\SS^\star_m)=h(\SS_\textrm{min}^\star).$$
	Therefore, it is enough to show that $h(\SS^\star_i)>h(\SS^\star_{i+1})$ holds for each $0\leq i\leq m-1$.

	For $0\leq i\leq m-1$, let $\tau_i\in \Iso(\SS^\star_i)$ be a minimal permutation of $\SS^\star_i$, i.e. $h(\SS^\star_i,\tau_i)=h(\SS^\star_i)$. We denote by $\tau_{i+1}$ a permutation of $\SS^\star_{i+1}$ induced from $\tau_{i}$ by simply substituting $Q_i$ with $P_i$, i.e.
	\[\tau_{i+1}(P)=\begin{cases}
	\tau_i(Q_i)&\textrm{if}\ P= P_i\\
	P_i&\textrm{if}\ \tau_i(P)= Q_i\\
	\tau_i(P)& \textrm{otherwise.}
	\end{cases}\]
	
	Now we claim that $h(\SS^\star_i,\tau_i)-h(\SS^\star_{i+1},\tau_{i+1})> 0.$ We need to consider the following two cases.
	
	\textbf{Case 1:} When $\tau_i(Q_i)=Q_i$, we have
	\begin{align*}
	&h(\SS^\star_i,\tau_i)-h(\SS^\star_{i+1},\tau_{i+1})&  \\
	=& \big\lceil pw(Q_i)-w(Q_i)\big\rceil-\big\lceil pw(P_i)-w(P_i)\big\rceil&\\
	=&\big\lceil (p-1)w(Q_i)\big\rceil-\big\lceil (p-1)w(P_i)\big\rceil&\\
	\geq& \big\lceil (p-1)(w(Q_i)-w(P_i))\big\rceil-1&\\
	\geq& \Big\lceil \frac{(p-1)\gcd(a_1-a_2,b_1-b_2)}{b_1a_2-b_2a_1}\Big\rceil-1&\textrm{Lemma~\ref{basic lemma}}\\
	>&0.&
	\end{align*}

	\textbf{Case 2:} When $\tau_i(Q_i)\neq Q_i$, let $Q_i'=\tau^{-1}_i(Q_i).$ Then we have
	\begin{align*}
	&h(\SS^\star_i,\tau_i)-h(\SS^\star_{i+1},\tau_{i+1})\\
	=& \big\lceil pw\big(\tau_i(Q_i)\big)-w(Q_i)\big\rceil-\big\lceil pw\big(\tau_i(Q_i)\big)-w(P_i)\big\rceil
	+\big\lceil pw(Q_i)-w(Q_{i}')\big\rceil-\big\lceil pw(P_i)-w(Q_{i}')\big\rceil\\
	\geq& -\big\lceil w(Q_i)-w(P_i)\big\rceil+\big\lceil pw(Q_i)-pw(P_i)\big\rceil-1\\
	\geq& \big\lceil(p-1)\big(w(Q_i)-w(P_i)\big)\big\rceil-2\\
	\geq& \Big\lceil\frac{(p-1)\gcd(a_1-a_2,b_1-b_2)}{b_1a_2-b_2a_1}\Big\rceil-2\hspace{5cm}\textrm{Lemma~\ref{basic lemma}}\\
	>&0.
	\end{align*}
	Then this lemma follows from the following \textbf{strict} inequality \[h(\SS^\star_i)=h(\SS^\star_i,\tau_i)> h(\SS^\star_{i+1},\tau_{i+1})\geq h(\SS^\star_{i+1}).\qedhere\]
\end{proof}

\begin{proof}[Proof of Proposition~\ref{better form for IHP}]
	First, we fix a subset $\SS'\in \mathscr M_\ell$ such that  $$h(\SS')=\min_{\SS\in \mathscr M_\ell}\big(h(\SS)\big).$$
	 Let $\widetilde{\SS}^\star$ be an arbitrary submultiset in $\mathscr M_\ell(n)$ such that $$h(\widetilde{\SS}^\star)=\min_{\SS^\star\in \mathscr M_\ell(n)}\big(h(\SS^\star)\big).$$
	 By Lemma~\ref{S}, we know that 
	 $$\sum_{P\in\SS'}w(P)=\min_{\SS\in \mathscr M_\ell}\Big(\sum_{P\in\SS}w(P)\Big)
	 \quad \textrm{and}\quad 
	 \sum_{P\in\widetilde{\SS}^\star}w(P)=\min_{\SS^\star\in \mathscr M_\ell(n)}\Big(\sum_{P\in\SS^\star}w(P)\Big).$$
	 It is not hard to see that   $$\sum_{P\in ({\SS'}^\star)^n}w(P)= \sum_{P\in\widetilde{\SS}^\star}w(P).$$
	 Therefore, by Lemma~\ref{S} again, we have 
	 $h(\widetilde{\SS}^\star)=h(({\SS'}^\star)^n).$
	 By Corollary~\ref{power}, we know that  $$h(({\SS'}^\star)^n)=nh(\SS').$$
	 Combining these equalities above gives us \eqref{u times}.
\end{proof}

%
\begin{definition}\label{det TTk}
	For any subset $\SS$ of $\MM(\Delta)$, we write
	$$\det(\SS)_f=\sum\limits_{\tau\in \Iso(\SS)}\sgn(\tau) \prod\limits_{P\in \SS}e_{p\tau(P)-P}.$$
\end{definition}
Then as a corollary of Proposition~\ref{better form for IHP}, we get the following.
\begin{proposition}\label{important corollary}
	We have $$v_T\Big(\prod_{j=0}^{n-1}\sigma_\Frob^j(\det(\TT_k)_f)- u_{\mathbbm{x}_k, nh(\TT_k)}T^{nh(\TT_k)}\Big)\geq nh(\TT_k)+1,$$ where $u_{\mathbbm{x}_k, nh(\TT_k)}$ is defined in \eqref{E:Cfstar}.
\end{proposition}
\begin{proof}
	By Lemma~\ref{S}, we know that ${\TT_k^\star}^n$ is the only element in $ \mathscr M_{\mathbbm x_k}(n)$ which makes \eqref{xk} an equality. Therefore, we have $$v_T\Big(u_{\mathbbm{x}_k, nh(\TT_k)}T^{nh(\TT_k)}-\prod_{j=0}^{n-1}\big(\sum_{\tau_j=\Iso(\TT_k)}\sgn(\tau_j)\prod_{P\in \TT_k}\sigma_{\Frob}^j(e_{p\tau_j(P)-P})\big)\Big)\geq nh(\TT_k)+1.$$
	Then this proposition follows directly from checking the definition of $\det(\TT_k)_f$.
\end{proof}



%

\begin{notation}\label{Newton slopes}
	For each character $\chi:\Zp\to \CC_p^\times$ of conductor $p$, from \cite[Theorem~1.4]{liu-wei}, $L_f^*(\chi,s)^{-1}$ is a polynomial of degree $a_2b_1-a_1b_2$.  We denote its $q$-adic Newton slopes by $$\Big(\alpha_1,\alpha_2,\dots,\alpha_{a_2b_1-a_1b_2}\Big)$$ in a non-descending order and put $\Sigma(\chi):=\sum\limits_{j=1}^{\mathbbm{x}_1}\alpha_j$. 
\end{notation}
\begin{lemma}
		For each character $\chi:\Zp\to \CC_p^\times$ of conductor $p$, the normalized Newton polygon of $C_f^*(\chi,s))$, i.e. $\NP(\chi, s)_C$ is not above points $$\Big(\mathbbm{x}_k, \big[(a_2b_1-a_1b_2)\sum\limits_{i=1}^{k-1}i^2-\frac{1}{2}(l_1+l_2)\sum\limits_{i=1}^{k-1}i+\mathbbm{x}_1(k-1)k+k\Sigma(\chi)\big](p-1)\Big)$$ for all integers $k\geq 0$.
\end{lemma}

\begin{proof}
	First recall that $n = [\FF_q:\FF_p]$ is the degree of the coefficient field of $f$ (see section 2). 
	It is well known that the roots of $L_f^*(\chi,s)^{-1}$ are Weil numbers of weight $0$, $n$, or $2n$. We put them into three classes according to the Weil weights: 
	\begin{center}
		\begin{tabular}{ |c|c|c| } 
			\hline
			Weil weight & the number of roots of $L_f^*(\chi,s)^{-1}$\\ 
			\hline					
			$ 0$ & $1$ \\ 
			\hline
			$n$ & $l_1+l_2$ \\ 
			\hline 
			$2n$ &  $a_2b_1-a_1b_2-l_1-l_2-1$\\ 
			\hline
		\end{tabular}.
	\end{center}
Since $\alpha_i$'s are the $q$-adic Newton slopes of $L_f^*(\chi,s)^{-1}$, we know easily that they belong to $[0,2)$.
	Moreover, an algebraic number, say $z$, and its complex dual $\overline z$ are both roots of $L_f^*(\chi,s)^{-1}$ or both not. Suppose that they are roots of $L_f^*(\chi,s)^{-1}$ and $z$ as Weil weight $t$. Then we have $v_p(z)+v_p(\overline z)=t$. Therefore, the sum of all $q$-adic Newton slopes of $L_f^*(\chi,s)^{-1}$ can be computed as follows:
	\begin{equation}\label{slope sum of L}
	\begin{split}
	&(a_2b_1-a_1b_2-l_1-l_2-1)\times 1+(l_1+l_2)\times\frac{1}{2}+1\times 0.\\
	=&a_2b_1-a_1b_2-\tfrac{1}{2}(l_1+l_2)-1.
	\end{split}
	\end{equation}
	On the other hand, from \eqref{C(chi, s)}, i.e.
	\begin{equation}\label{C-function1}
	C^*_f(\chi, s)=\prod\limits_{j=0}^{\infty}L^*_f(\chi, q^js)^{-(j+1)},
	\end{equation}
	we know that
	\begin{equation}\label{a subset of roots of C*f}
	\big(\biguplus_{i=0}^{k-1}\{\alpha_1+i, \alpha_2+i,\dots, \alpha_{a_2b_1-a_1b_2}+i\}^\star\Big)^{i+1}\uplus\Big(\{\alpha_1+k-1, \alpha_2+k-1,\dots, \alpha_{\mathbbm x_1}+k-1\}^\star\Big)^{k}
	\end{equation}
	is included in the set of $q$-adic Newton slopes of $C^*_f(\chi,s)$ as multisets and its cardinality is equal to $\mathbbm x_k$.
	 Since elements in this set are not necessary to be the smallest $\mathbbm x_k$ Newton slopes of  $C^*_f(\chi,s)$, then the height of $\NP(\chi, s)_C$ at $x=\mathbbm{x}_k$ is not above the sum
	\begin{equation}\label{upper bound}
	\begin{split}
	&(p-1)\Big[\sum_{i=1}^{k-1}i\sum_{j=1}^{a_2b_1-a_1b_2} (i-1+\alpha_j) +k\sum_{j=1}^{\mathbbm{x}_1}(k-1+\alpha_j)\Big]\\
	=&(p-1)\Big[(a_2b_1-a_1b_2)\sum\limits_{i=1}^{k-1}i^2-(\frac{1}{2}(l_1+l_2)+1)\sum\limits_{i=1}^{k-1}i+\mathbbm{x}_1(k-1)k+k\Sigma(\chi)\Big],
	\end{split}
	\end{equation}
	where $p-1$ is from normalization in the definition of $\NP(\chi, s)_C$.
\end{proof}

\begin{lemma}\label{vertex for IHP}
	For each $k\geq 1$, 
	
	\noindent$(1)$ both $(\mathbbm{x}_k, h(\TT_k))$ and $(\mathbbm{x}'_k, h(\TT'_k))$ are vertices of $\IHP(\Delta)$, and 
	
	\noindent $(2)$  the segment with endpoints $(\mathbbm{x}_k, h(\TT_k))$ and $(\mathbbm{x}'_k, h(\TT'_k))$ is contained in
	 $\IHP(\Delta)$. 
\end{lemma}

\begin{proof}
	(1) Suppose the lemma were false. Then there exists an integer $k$ and a segment in $\IHP(\Delta)$, say $\overline{P_1P_2}$, such that $P_0:=(\mathbbm{x}_k, h(\TT_k))$ is either a point strictly above $\overline{P_1P_2}$ or an interior point on $\overline{P_1P_2}$.
	From Proposition~\ref{better form for IHP}, we know that $$P_1=\Big(\mathbbm x_k-i_1, \min_{\SS\in \mathscr M_{\mathbbm{x}_k-i_1}}(h(\SS)) \Big)\quad\textrm{and}\quad P_2=\Big(\mathbbm x_k+i_2, \min_{\SS\in \mathscr M_{\mathbbm{x}_k+i_2}}(h(\SS)) \Big)$$
for some positive integers $i_1$ and $i_2$.

	Put $\SS_1$ to be an element of $\mathscr M_{\mathbbm{x}_k-i_1}$ such that
	\begin{equation}
	\sum_{P\in\SS_1}w(P)=\min_{\SS\in \mathscr M_{\mathbbm{x}_k-i_1}}\Big(\sum_{P\in\SS}w(P)\Big).
	\end{equation}
	By Lemma~\ref{S}, we get 
	\begin{equation}
	h(\SS_1)=\min_{\SS\in \mathscr M_{\mathbbm{x}_k-i_1}}\big(h(\SS)\big).
	\end{equation}
It is easy to know that $\SS_1$ is a subset of $\TT_k$. We denote its complement in $\TT_k$ by $\SS_1'$, which is of cardinality $i_1$. By Lemma~\ref{basic lemma}, we know that each point $P$ in $\TT_k$ satisfies $$w(P)\leq k-\frac{\gcd(a_1-a_2,b_1-b_2)}{a_2b_1-a_1b_2}.$$ Combining it with $$h_2(\SS_1')\leq \#\SS_1'=i_1,$$ we have
	\[\begin{split}
	h(\TT_k)&\leq h(\SS_1)+h(\SS'_1)= h(\SS_1)+h_1(\SS'_1)+h_2(\SS'_1)\\
	&\leq h(\SS_1)+i_1(p-1)\left[k-\frac{\gcd(a_1-a_2,b_1-b_2)}{a_2b_1-a_1b_2}\right]+i_1.
	\end{split}
	\]
	It simply implies that the slope of $\overline{P_1P_0}$ is less than or equal to $$(p-1)\left[k-\frac{\gcd(a_1-a_2,b_1-b_2)}{a_2b_1-a_1b_2}\right]+1.$$
	
	On the other hand, 	by a similar argument, we choose an element $\SS_2$ from  $\mathscr M_{\mathbbm{x}_k+i_2}$ such that $$\sum_{P\in\SS_2}w(P)=\min_{\SS\in \mathscr M_{\mathbbm{x}_k+i_2}}\big(\sum_{P\in\SS}w(P)\big).$$ By Lemma~\ref{S} again, we have 
	\begin{equation}
	h(\SS_2)=\min_{\SS\in \mathscr M_{\mathbbm{x}_k+i_2}}\big(h(\SS)\big).
	\end{equation} 
	We also easily know that $\TT_k$ is included in $\SS_2$.
	Let $\tau$ be a minimal permutation of $\SS_2$. We shall construct below a finite sequence of permutations of $\SS_2$, denoted by $(\tau_0=\tau, \tau_1,\dots, \tau_m)$, satisfying 
	\begin{itemize}
		\item[(a)] the length $m$ of this sequence is less than or equal to $i_2$,
		\item[(b)] $h(\tau_{i+1})\leq h(\tau_i)+1$, and
		\item[(c)] $\tau_m$ fixes every point in $\SS_2\backslash \TT_k$. 
	\end{itemize}

	Put $\tau_0=\tau$.
Assume that we have $\tau_i$ already. If it fixes each point in $\SS_2\backslash\TT_k$, then we are done. Otherwise put $P_i$ to be a point in $\SS_2\backslash\TT_k$ such that $\tau_i(P_i)\neq P_i$. Then we define $\tau_{i+1}$ the same permutation as $\tau_{i}$ except we swap images of $P_i$ and $\tau_i^{-1}(P_i)$. Iterating this process gives us a sequence of permutations of $\SS_2$. If $\tau_m$ is the last element in this sequence, we know that it fixes each point in $\SS_2\backslash\TT_k$, namely,
	$$\tau_m(P)=P\quad \textrm{for each}\ P\in \SS_2\backslash\TT_k.$$ 
	Since there are at most $i_2$ points in $\SS_2$ whose images are changed by these modifications, we know that $m\leq i_2$. Put $Q_i=\tau_i(P_i)$ and $P_i'=\tau_i^{-1}(P_i)$. We compute 
	\[\begin{split}
	&h(\tau_i)-h(\tau_{i+1})\\=&\lceil w(pQ_i-P_i)\rceil+\lceil w(pP_i-P_i')\rceil- \lceil w(pQ_i-P_i)\rceil-\lceil w(pP_i-P_i')\rceil\\
	\geq &1
	\end{split}\]
	Then we simply prove that the constructed sequence of permutations of $\SS_2$ satisfies conditions (a)-(c) above.
	Moreover, we have $$ h_2(\SS_2, \tau)\geq h_2(\SS_2, \tau_m)-i_2.
	$$ 
	As $\tau$ is minimal, we get
	\begin{equation}\label{lower bound for h2SS}
	\begin{split}
	h(\SS_2)
	=&h_1(\SS_2)+h_2(\SS_2, \tau)\\
	\geq& h_1(\SS_2)+h_2(\SS_2, \tau_m)-i_2\\
	=&h(\SS_2, \tau_m)-i_2.
	\end{split}
	\end{equation}
	Since the restriction of $\tau_m$ on $\TT_k$ is a permutation of $\TT_k$ and $w(P)\geq k$ for any point $P$ in $\SS_2\backslash\TT_k$, we have 	
	\begin{equation}\label{hS2sigma'}
	\begin{split}
	h(\SS_2, \tau_m)&= h\left(\TT_k, \tau_m\big|_{\TT_k}\right)+h\left(\SS_2\backslash\TT_k,\tau_m\big|_{\SS_2\backslash\TT_k}\right)\\
	&\geq h\left(\TT_k\right)+h_1(\SS_2\backslash\TT_k)\\
	&\geq h\left(\TT_k\right)+i_2(p-1)k.
	\end{split}
	\end{equation}
	By \eqref{lower bound for h2SS} and \eqref{hS2sigma'}, the slope of $\overline{P_0P_2}$ is greater than or equal to $$k(p-1)-1.$$
	 Under the assumption $p>\frac{2(a_2b_1-a_1b_2)}{\gcd(a_1-a_2,b_1-b_2)}+1$  in Theorem~\ref{main}, it is easy to check that the slope of $\overline{P_0P_2}$ is strictly greater than the slope of $\overline{P_1P_0}$, which is a contradiction.
	 
	 By a similar argument, we know that $(\mathbbm{x}'_k, h(\TT'_k))$ is also a vertex of $\IHP(\Delta)$. 
	 
	 (2) Let $0\leq i\leq \mathbbm x_k'-\mathbbm x_k$. By Lemma~\ref{S}, there exists $\TT_k\subset\SS_i'\subset \TT_k'$ such that $$h(\SS_i')=\min_{\SS\in \mathscr M_{\mathbbm x_k+i}}(h(\SS)).$$  Since all points in $\TT_k'\backslash \TT_k$ have integer weight $k$, the we have $$h(\SS_i')=h(\TT_k)+ik(p-1),$$ which implies (2) immediately. 
\end{proof}
\begin{lemma}\label{coincide then unit}
	Let $\chi_{1}$ be a nontrivial finite character. Suppose that $\NP(f, \chi_{1})_C$ coincides with $\IHP(\Delta)$ at point $\big(\mathbbm{x}_{k_1}, h(\TT_{k_1})\big)$ for a positive integer $k_1$. Then
\begin{itemize}
	\item[$(1)$] $\big(\mathbbm{x}_{k_1}, h(\TT_{k_1})\big)$ is also a vertex of $\NP(f, \chi_{1})_C$, and
	\item[$(2)$] $u_{\mathbbm{x}_{k_1}, nh(\TT_{k_1})}$ is a $\Zp$-unit. 
\end{itemize}
\end{lemma}

\begin{proof}
	(1) This follows from the fact that the normalized Newton polygon $\NP(f, \chi_1)_C$ always lies above the improved Hodge polygon $\IHP(\Delta)$ by Proposition~\ref{IHP is below NPchi} and that $\big(\mathbbm{x}_{k_1}, h(\TT_{k_1})\big)$ is a vertex of $\IHP(\Delta)$ by Lemma~\ref{vertex for IHP}.
	
	(2) Suppose that $u_{\mathbbm{x}_{k_1}, nh(\TT_{k_1})}$ is not a $\Zp$-unit. Since we know that $(\mathbbm{x}_{k_1}, h(\TT_{k_1}))$ is a vertex of $\NP(f, \chi_{1})_C$, then specializing $\NP(f, T)_C$ to $T=\chi_{1}(1)-1$ makes $\NP(f, \chi_{1})_C$ strictly higher than $\NP(f, T)_C$ at $x=\mathbbm{x}_{k_1}$. By Lemma~\ref{IHP is below NPchi}, it is also strictly higher than $\IHP(\Delta)$ at this point, which leads to a contradiction.
\end{proof}

%

\begin{proof}[Proof of Theorem \ref{main}]
	Let $\chi_{0}:\ZZ_p\to \CC_p^\times$ be a nontrivial character of conductor $p$. 
	Since $\NP(f, \chi_{0})_C$ is not below $\IHP(\Delta)$ and the expression in \eqref{upper bound} is an upper bound for $\IHP(\Delta)$ at $x=\mathbbm x_k$ for each $k\geq 1$, then we have
	\begin{equation}\label{inequality1}
	\begin{split}
	&k\big[\frac{h(\TT_1)}{p-1}+\mathbbm{x}_1(k-1)\big]+\sum_{i=0}^{k-2}[(a_2b_1-a_1b_2)(i+1)-\frac{1}{2}(l_1+l_2)](i+1)\\
	\leq&\mathbbm{x}_1(k-1)k+k\Sigma(\chi_{1})+(a_2b_1-a_1b_2)\sum\limits_{i=1}^{k-1}i^2-\big[\frac{1}{2}(l_1+l_2)+1\big]\sum\limits_{i=1}^{k-1}i.\\
	\end{split}
	\end{equation}
	
	A simplification of these inequalities above shows that they all equivalent to 
	\begin{equation}\label{coincide condition} 
	h(\TT_1)\leq (p-1)\Sigma(\chi_{0}),
	\end{equation}
an equality independent of $k$.
	
	Since $\NP(f, \chi_{1})_C$ coincides with $\IHP(\Delta)$ at $\big(\mathbbm{x}_{k_1}, h(\TT_{k_1})\big)$ for a finite character $\chi_{1}$ and an integer $k_1$, by Lemma~\ref{coincide then unit}, we know that $u_{\mathbbm{x}_{k_1}, nh(\TT_{k_1})}$ is a $\Zp$-unit. It implies that $\NP(f, \chi_{0})_C$ also coincides with $\IHP(\Delta)$ at $(\mathbbm{x}_{k_1}, h(\TT_{k_1}))$. Therefore, by Lemma~\ref{vertex for IHP}~(2), the slopes of segments in $\NP(f, \chi_{0})_C$ after points $x=\mathbbm x_{k_1}$ are all greater than or equal to $(p-1)k_1$. 
	
	One the other hand, recall that in Notation~\ref{Newton slopes} we put $\{\alpha_1, \alpha_2,\dots,\alpha_{a_2b_1-a_1b_2}\}$ (in a non-decreasing order) to be the set of $q$-adic Newton slopes for $L_f^*(\chi_{0}, s)$, which is contained in $[0,2)$. Therefore, each $q$-adic Newton slope of  $L_f^*(\chi_{0},q^is)$ belongs to $[i, i+2)$. Then from the decomposition of $C_f^*(\chi_{0}, s)$  in \eqref{C-function1}, we know that $\alpha_j+k_1-1\geq k_1$ for all $j\geq \mathbbm x_1+1$. For otherwise there are more than $\mathbbm x_{k_1}$ roots of $C_f^*(\chi_{0}, s)$ whose $q$-adic valuations are less than or equal to $k_1$, which is a contradiction to the statement in the previous paragraph.   
	
	From the argument above, we see that \eqref{inequality1} must be an equality, and when $k=k_1$, the height of $\NP(\chi_0, s)_C$ coincides with its upper bound given in \eqref{upper bound}. Hence, we have
	$$h(\TT_1)= (p-1)\Sigma(\chi_{0}).$$
		Notice that \eqref{coincide condition} is independent of $k$. Then the inequalities in \eqref{inequality1} are equalities for all $k\geq 0$. 
		Combining it with Proposition~\ref{coincide then unit}, we have that $u_{\mathbbm{x}_k, nh(\TT_k)}$ is a $\Zp$-unit for each $k\geq 0$.
		Therefore, it is not hard to show that $\NP(f, \chi)_C$ coincides with $\IHP(\Delta)$ at $(\mathbbm{x}_{k}, h(\TT_k))$ for any integer $k$ and nontrivial finite character $\chi$. Then by Lemma~\ref{vertex for IHP} again, we know that $\alpha_i\geq 1$ for $\mathbbm x_1\leq i\leq \mathbbm x'_1$. Combining it with Poincar\'e duality, we have $\alpha_i= 1$ for $\mathbbm x_1+1\leq i\leq \mathbbm x'_1$. It implies that $\IHP(\Delta)$ and $\NP(f, \chi)_C$ coincide at $x=\mathbbm{x}_k+i_k$ for any $0\leq i_k\leq\mathbbm x_k'-\mathbbm x_k$.

		By a similar argument to Lemma~\ref{coincide then unit}~(2), we know that  $u_{\mathbbm{x}'_k, nh(\TT'_k)}$ is also a $\Zp$-unit. 
\end{proof}

\section{The case when $\Delta$ is an isosceles right triangle I.}

In order to apply Theorem~\ref{main} we need that $\NP(f, \chi_{1})_C$ coincides with $\IHP(\Delta)$ at $x=\mathbbm{x}_{k_1}$ for some character $\chi_{1}$ and some integer $k_1$. This however seems to be a very difficult question. We have the following folklore conjecture.

\begin{conjecture}\label{conjecture}
	Let $\Delta$ be a triangle with vertices at $(0,0)$, $\mathbf P_1=(a_1,b_1)$, $\mathbf P_2=(a_2,b_2)$. We assume the hypotheses (as in Theorem~\ref{main}) on the prime $p$.
In the moduli space of all polynomials $f(x_1,x_2)$ of convex hull $\Delta$, there exists an open dense subspace over which the corresponding Newton polygon $\NP(f, \chi)_C$ agrees with $\IHP(\Delta)$ at $x=\mathbbm{x}_k+i_k$ for all finite characters $\chi$, integers $k\geq 1$ and $0\leq i_k\leq \mathbbm x_k'-\mathbbm x_k$.
\end{conjecture}
Generically, the Newton polygon of $C_f^*(\chi, s)$ should be as low as possible, namely, coinciding with the improved Hodge bound $\IHP(\Delta)$.

In this section, we will study a special case when $\Delta$ is an isosceles right triangle with vertices at $(0,0)$, $(d,0)$ and $(0,d)$, where $p\nmid d$. We claim that Conjecture~\ref{conjecture} holds true when the residue of $p$ modulo $d$ is small enough. More precisely, we will prove the  following.
\begin{theorem}\label{Thm for 4}
	Let $p_0$ be the residue of $p$ modulo $d$, and let $d_0$ be the residue of $d$ modulo $p_0$. Conjecture~\ref{conjecture} holds when 
	\begin{equation}\label{a weaker condition}
	d\geq\begin{cases}
 4p_0^{\frac{3}{2}}\Big[\ln 3+\frac{27}{4}+(2+\frac{3}{\ln 2})\ln p_0\Big]+13p_0 &\textrm{if}\ h\geq\frac{1}{4},\\
 4p_0^{\frac{5-2h}{3}}\Big[\ln 3+\frac{27} 4+(2+\frac{3}{\ln 2})\ln p_0\Big]+13p_0 & \textrm{if}\ h<\frac{1}{4},
	\end{cases}
		\end{equation}
	where $h:=\log_{p_0} (p_0-d_0).$
\end{theorem}

In particular, the condition $d \geq 24(2p_0^2+p_0)$ implies \eqref{a weaker condition}; so Theorem~\ref{generic newton polygon} follows from this. We will complete the proof at the end of this section.
 \subsection{An interpretation of Theorem~\ref{Thm for 4}.}

First, we consider the ``\emph{universal polynomial}''  $$f_{\univ}(x_1,x_2)=\sum\limits_{P\in \Delta\cap \MM(\Delta)}\widetilde{a}_{P}x_1^{P_x}x_2^{P_y}$$ whose coefficients are treated as variables. 
%

\begin{notation}\label{definition for det}

Recall the infinite matrix $N$ defined in Proposition~\ref{IHP is below NPchi}. Let $\widetilde N$ be the matrix given by substituting $\hat a_P$ in $N$ by $\widetilde a_P$. More rigorously, we put $$ E_{f_{\univ}}(x_1,x_2):=\sum\limits_{P\in \ZZ^2_{\geq 0}}\widetilde{e}_{P}(\underline{\widetilde a}, T)x_1^{P_x}x_2^{P_y}\in \ZZ_p[ \underline{\widetilde a}] \llbracket T,x_1,x_2\rrbracket,$$ and write $$\widetilde{e}_{P}:=\widetilde{e}_{P}(\underline{\widetilde a}, T).$$ 

Similar to Lemma~\ref{L:estimate of Ef(x)}, we have
\begin{equation}\label{ei}
\widetilde{e}_{P} \in T^{\lceil w(P)\rceil} \ZZ_p[\underline{\widetilde a}]\llbracket T\rrbracket\quad \textrm{and}\quad \widetilde{e}_O=1.
\end{equation}
Then, using the list $(P_1, P_2, \dots)$ of points in $\MM(\Delta)$ given in Notation~\ref{Pi}, we define $\widetilde N$ to be the infinite matrix whose $(i,j)$ entry is $\widetilde e_{pP_i-P_j}\in \ZZ_p[\underline{\widetilde a}]\llbracket T \rrbracket.$


	Similar to Definition~\ref{det TTk}, we put \[\begin{split}
	\det(\TT_1)_\univ=&\sum_{\tau\in \Iso(\TT_1)}\sgn(\tau)\prod_{P\in \TT_1}\widetilde{e}_{p\tau(P)-P}\\
	=&\sum_{i=h(\TT_1)}^{\infty}\widetilde{v}_iT^i\in \ZZ_p[\underline{\widetilde a}]\llbracket T \rrbracket,
	\end{split}\]
	where $\sgn(\tau)$ is the sign of $\tau$ as a permutation.
\end{notation}

Since all results in Section 3 for the ``fixed'' $f$ actually hold for on general polynomials $f(x_1,x_2)\in \overline\FF_p[x_1, x_2]$, we have the following.

%


\begin{proposition}\label{reduce to Ef}
	The polygons $\GNP(\Delta)$ and $\IHP(\Delta)$ coincide at $\big(\mathbbm x_1, h(\TT_1)\big)$ if and only if $\widetilde{v}_{h(\TT_1)}$ is not divisible by $p$.
	
	 Moreover, when either condition holds, Conjecture~\ref{conjecture} holds for that $\Delta$.
\end{proposition}	
\begin{proof}
	We first prove the ``only if'' part. Suppose that $\widetilde{v}_{h(\TT_1)}$ is divisible by $p$. For any pair of a two-variable polynomial $\mathring f(x_1,x_2)\in \overline\FF_p[x_1,x_2]$ with convex hull $\Delta$ and a finite character $\mathring\chi$ of conductor $p^{m_{\mathring\chi}}$,  we write $\mathring v_{h(\TT_1)} \in \ZZ_{p^{n(\mathring f)}}[\zeta_{p^{m_{\mathring \chi}}}]$ as the specialization of $\widetilde{v}_{h(\TT_1)}$ at $T = \mathring \chi(1)-1$ and at $\widetilde{a}_P$ equals to the Teichmuller lifts of the coefficients of $f$, where  $\zeta_{p^{m_{\mathring \chi}}}$ is a primitive $p^{m_{\mathring \chi}}$-th root of unity.
	Then we have $$v_{p}(\mathring v_{h(\TT_1)})\geq \frac{h(\TT_1)}{p^{m_{\mathring \chi}-1}(p-1)}+1.$$ 
	As in \eqref{C(chi, s)}, we denote $$C^*_{\mathring f}(\mathring\chi, s)=\prod\limits_{j=0}^{\infty}L^*_{\mathring f}(\mathring\chi, p^{jn(\mathring f)}s)^{-(j+1)}=\sum_{k=0}^\infty\mathring u_ks^k\in \ZZ_{p}[\zeta_{p^{m_{\mathring \chi}}}]\llbracket s \rrbracket.$$

	By Proposition~\ref{important corollary}, we know that $$v_p\Big(\prod_{i=0}^{n(f)-1}\sigma_{\Frob}^{i}(\mathring v_{h(\TT_1)})-\mathring u_{\mathbbm x_1}\Big)\geq \frac{n(\mathring f)h(\TT_1)+1}{p^{m_{\mathring \chi}-1}(p-1)},$$
	where $\sigma_\Frob$ represents the arithmetic Frobenius acting on the coefficients.
	
Combining the equalities above, we get that
	$$v_{p}(\mathring u_k)
	\geq \frac{n(\mathring f)h(\TT_1)+1}{p^{m_{\mathring \chi}-1}(p-1)}.$$ Since we choose $\mathring f$ and $\mathring\chi$ arbitrarily, 
	we know that $\GNP(\Delta)$ is strictly above $\IHP(\Delta)$ at $x=\mathbbm x_1$, a contradiction.

We prove the ``if'' part.
Let $\overline{\widetilde{v}}_{h(\TT_1)}$ be the image of $\widetilde{v}_{h(\TT_1)}$ in the quotient ring $$\ZZ_p[ \underline{\widetilde a}] / p\ZZ_p[ \underline{\widetilde a}] \cong \FF_p[ \underline{\widetilde a}].$$
Since $\widetilde{v}_{h(\TT_1)}$ is not divisible by $p$, we know that $\overline{\widetilde{v}}_{h(\TT_1)}\neq 0$.

  Recall that we defined $\TT_1'=\{P\in \MM(\Delta)\;|\; w(P)\leq 1\}$ and $\mathbbm{x}_1'$ to be its cardinality in Notation~\ref{cone}.
Let $f_1(x_1,x_2)=\sum\limits_{P\in \TT_1'}b_Px_1^{P_x}x_2^{P_y}\in \overline{\FF}_p[x_1,x_2]$ be a polynomial satisfies that
\begin{itemize}
	\item $f_1$ has convex hull $\Delta$.
	\item  $\overline{\widetilde{v}}_{h(\TT_1)}|_{\widetilde a_P=b_P}\neq 0$.
\end{itemize}  
It is easy to check that for any finite character $\chi$ of conductor $p$, $\NP(f_1, \chi)_L$ coincides with $\IHP(\Delta)$ at $x=\mathbbm x_1$. Since $\GNP(\Delta)$ is not below $\IHP(\Delta)$ and the set of polynomials which satisfy these conditions forms a Zariski open subset in the affine space $\AAA_{\FF_p}^{\mathbbm x'_1}$, we complete the proof.
\end{proof}
Now we are left to show that $\widetilde{v}_{h(\TT_1)}$ is not divisible by $p$.

\begin{definition}
	We label the elements in $\TT_1'$ as $$\TT_1':=\big\{Q_1, Q_2,\dots,Q_{\mathbbm{x}_1'}\big\}$$ such that $Q_1:=(d,0)$ and $Q_2:=(0,d)$. Each point $P\in \MM(\Delta)$ can be written as a linear combination of points in $\TT_1'$ with non-negative integer coefficients, namely $$P=\sum_{i=1}^{\mathbbm{x}'_1}b_{P,i}Q_i.$$ We call the vector $\vec{b}_P\in \ZZ_{\geq 0}^{\mathbbm{x}'_1}$ (or the linear combination)  \emph{$P$-minimal} if it satisfies
	\[
	 \sum_{i=1}^{\mathbbm{x}'_1}b_{P,i}=\big\lceil w(P)\big\rceil.
	\]
\end{definition}
\begin{definition}\label{definition for optimal combo}
A \emph{combo}, denoted by $(\tau, \vec{b}_{\bullet, \tau})$, is a pair consisting of an arbitrary permutation $\tau$ of $\TT_1$ together with, for each $P \in \TT_1$, a vector $\vec{b}_{P, \tau}\in \ZZ_{\geq 0}^{\mathbbm{x}'_1}$ such that 
\begin{equation}\label{linear combo}
\sum_{i=1}^{\mathbbm{x}'_1}b_{P,\tau,i}Q_i=p\tau(P)-P.
\end{equation}

	A combo $(\tau, \vec{b}_{\bullet, \tau})$ is called \emph{optimal} if $\tau$ is minimal and for each $P\in\TT_1$ vector $\vec{b}_{P, \tau}$ is $\big(p\tau(P)-P\big)$-minimal.
	
		A combo $(\tau, \vec{b}_{\bullet, \tau})$ is optimal if and only if $$\sum_{P\in \TT_1}\sum_{i=1}^{\mathbbm x_1'}b_{P,\tau,i}=h(\TT_1).$$
\end{definition}

%

%


%

We have the following explicit expression of the leading coefficient $\widetilde{v}_{h(\TT_1)}$.
\begin{lemma}\label{combo}
We have
	 \begin{equation}\label{Ux1}
	\widetilde{v}_{h(\TT_1)}=\sum_{(\tau, \vec{b}_{\bullet, \tau}) \textrm{ optimal}}\sgn(\tau)\prod_{P\in \TT_1} \prod\limits_{i=1}^{\mathbbm{x}'_1}\frac{(\widetilde a_{Q_i})^{b_{P,\tau,i}}}{b_{P,\tau,i}!},
	\end{equation}
	where the sum runs over all optimal combos, and $\sgn(\tau)$ is the sign of $\tau$.
\end{lemma}

\begin{proof}
	
	We let $$\prod\limits_{Q\in\TT_1'} e^{\widetilde{a}_Q\pi x_1^{Q_x}x_2^{Q_y}}=\sum\limits_{p\in \ZZ^2_{\geq 0}}\widetilde{e}'_Px_1^{P_x}x_2^{P_y}.$$
	
	For any optimal combo $(\tau, \vec{b}_{\bullet, \tau})$, we have $$b_{P,\tau,i}\leq p-1\quad \textrm{for each}\ P\in \TT_1\ \textrm{and}\ 1\leq i\leq \mathbbm x_1',$$ which implies that
	\begin{equation}\label{det TT}
	\begin{split}
	\det(\TT_1)_\univ=&	\sum_{\tau\in \Iso(\TT_1)}\sgn(\tau)\prod_{P\in \TT_1}\widetilde{e}'_{p\tau(P)-P}+ O(T^{h(\TT_1)+1})\\
	=&\sum_{(\tau, \vec{b}_{\bullet, \tau})}\Big(\sgn(\tau)\prod_{P\in \TT_1} \prod\limits_{i=1}^{\mathbbm{x}'_1}\frac{(\widetilde a_{Q_i})^{b_{P,\tau,i}}}{b_{P,\tau,i}!}\Big)T^{\sum\limits_{P\in \TT_1}\sum\limits_{i=1}^{\mathbbm x_1'}b_{P,\tau,i}}+O(T^{h(\TT_1)+1}),
	\end{split}
\end{equation}
	where $(\tau, \vec{b}_{\bullet, \tau})$ runs over all combos for $\TT_1$ and $O(T^{h(\TT_1)+1})$ represents for a power series in $\ZZ_p[\underline{\widetilde{a}}]\llbracket T \rrbracket$ of $T$-adic valuation greater than or equal to $h(\TT_1)+1$.
Then this lemma follows from the last statement in Definition~\ref{definition for optimal combo}.
\end{proof}

\begin{definition}	\label{construction of monomial}
	Lemma~\ref{combo} gives an explicit expression of $\widetilde{v}_{h(\TT_1)}$ as the sum of terms labeled by optimal combos. For a combo $(\tau, \vec{b}_{\bullet, \tau})$, we call $$\sgn(\tau)\prod_{P\in \TT_1} \prod\limits_{i=1}^{\mathbbm{x}'_1}\frac{(\widetilde a_{Q_i})^{b_{P,\tau,i}}}{b_{P,\tau,i}!}$$ its \emph{corresponding monomial}.
	
	Two combos $(\tau, \vec{b}_{\bullet, \tau})$ and $(\tau', \vec{b}'_{\bullet, \tau'})$ have a same corresponding monomial (with possibly different coefficients) if and only if $$\sum_{P\in \TT_1}b_{P,\tau, i}=\sum_{P\in \TT_1}b'_{P,\tau, i}$$ for all $1\leq i\leq \mathbbm{x}'_1.$
\end{definition}

Recall that our task is to prove that $\widetilde{v}_{h(\TT_1)}$ is not divisible by $p$. For this it is enough to show that $\widetilde{v}_{h(\TT_1)}$ has a monomial
 whose coefficient is not divisible by $p$. To this end, we restrict our study to those monmials corresponding to some ``extreme'' optimal combos. 

\begin{lemma}
	For each combo $(\tau, \vec{b}_{\bullet, \tau})$ for $\TT_1$, we have \begin{equation}\label{upper bound for bP12}
	\sum_{P\in \TT_1}b_{P,\tau, 1}\leq \sum_{P\in \TT_1} \Big\lfloor\frac{pP_x}{d}\Big\rfloor\quad\textrm{and}\quad \sum_{P\in \TT_1}b_{P,\tau, 2}\leq \sum_{P\in \TT_1} \Big\lfloor\frac{pP_y}{d}\Big\rfloor,
	\end{equation}
	where $P_x$ and $P_y$ are the $x,y$-coordinates of $P$.
\end{lemma}
\begin{proof}
	We will prove the first inequality, and the proof of the second is similar
	
	Recall that $Q_1=(d,0)$. By equality~\eqref{linear combo}, we get 
	\begin{equation}\label{bound for coefficient of (d,0)}
	b_{\tau^{-1}(P),\tau,1}\leq \Big\lfloor \frac{pP_x}{d}\Big\rfloor.
	\end{equation} 
	Hence, we have that
	\[
	\sum_{P\in \TT_1}b_{P,\tau, 1}=\sum_{P\in \TT_1}b_{\tau^{-1}(P),\tau, 1}\leq \sum_{P\in \TT_1} \Big\lfloor\frac{pP_x}{d}\Big\rfloor.\qedhere
	\]
\end{proof}
\begin{definition}\label{special combo}
	 We call a combo \emph{special} if it is optimal and both inequalities in \eqref{upper bound for bP12} are equalities.
\end{definition}
 Notice that these two sums are the exponents of $\widetilde a_{Q_1}$ and $\widetilde a_{Q_2}$ in the corresponding monomial; so special combos contribute to terms in $\widetilde{v}_{h(\TT_1)}$ with maximal degrees in the coefficients $\widetilde a_{Q_1}$ and $\widetilde a_{Q_2}$ at the two vertices of $\Delta$.

Recall that for each point $P$ in $\MM(\Delta)$, we denoted by $P\%$ its residue in $\square_\Delta$. 

\begin{notation}
	We put $\TT_1=\TT_{1,1}\sqcup\TT_{1,2}$, where $\TT_{1,1}=\{P\in \TT_1\;|\;(pP)\%\in \TT_1\}$ and $\TT_{1,2}$ is the complement of $\TT_{1,1}$ in $\TT_1$. In other words, we have $\TT_{1,2}=\{P\in \TT_1\;|\;(pP)\%\notin \TT_1\}$.
\end{notation}

\begin{example}\label{example}
	When $d=7$ and $p=17$, the following graph shows the distribution of $\TT_{1,1}$ and $\TT_{1,2}$ in $\TT_1$, where ``$\times$'' and ``$\bullet$'' represent points in $\TT_{1,1}$ and $\TT_{1,2}$ respectively.
	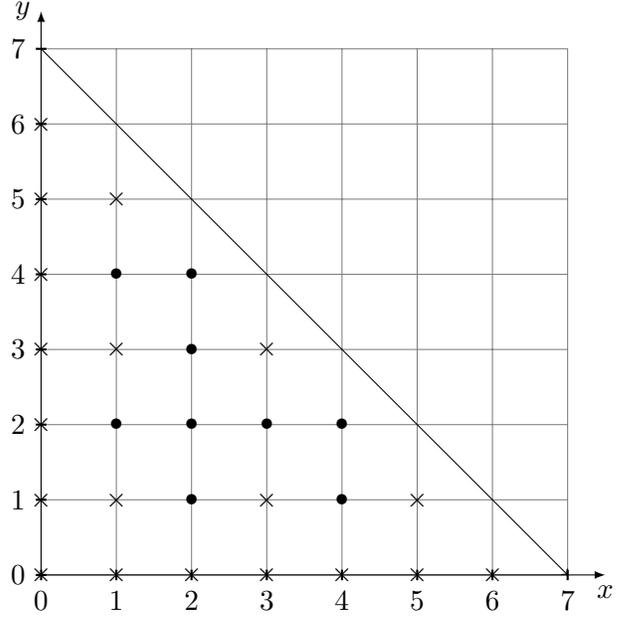
\begin{figure}[h]
		\begin{tikzpicture}
		\tkzInit[xmax=7,ymax=7,xmin=0,ymin=0]
		\tkzGrid
		\tkzAxeXY
		\foreach \Point in {(0,0), (0,1), (0,2), (0,3), (0,4), (0,5), (0,6), (1,0), (1,1),  (1,3), (1,5), (2,0), (3,0), (3,1), (3,3), (4,0), (5,0),(5,1), (6,0)}{
			\node at \Point {$\times$};
		}
		\foreach \Point in {(1,2), (1,4), (2,1), (2,2), (2,3), (2,4), (3,2), (4,1),(4,2)}{
			\node at \Point {$\bullet$};
		}
		\foreach \Point in {}{
			\node at \Point {$\circ$};
		}
		\draw[->] (0,7) -- (7,0);
		\end{tikzpicture},
		\caption{The distributions of $\TT_{1,1}$ and $\TT_{1,2}$ when $d=7$ and $p=17$.} \label{fig: triangle1}
	\end{figure}

\end{example}

\begin{lemma}\label{maximum of exponent}
	 A combo $(\tau, \vec{b}_{\bullet, \tau})$ is special if and only if it satisfies the following two conditions.
	\begin{itemize}
		\item[$(1)$]	For each $P\in \TT_{1,1}$, we have $$\tau^{-1}(P)=(pP)\%$$ and	all other $b_{\tau^{-1}(P),\tau,*}$'s are zero except $b_{\tau^{-1}(P),\tau,1}$ and $b_{\tau^{-1}(P),\tau,2}$ which are equal to $i_{P,1}$ and $i_{P,2}$.

		\item[$(2)$] For each $P\in \TT_{1,2}$, assume that $Q_{j_P}=(pP)\%-\tau^{-1}(P)$, then we have \begin{equation*}
		(pP)\%-\tau^{-1}(P)\in \TT_1'
		\end{equation*}\
		and	all other $b_{\tau^{-1}(P),\tau,*}$'s are zero except $b_{\tau^{-1}(P),\tau,1}, b_{\tau^{-1}(P),\tau,2}$ and $b_{\tau^{-1}(P),\tau,j_{P}} $ which are equal to $i_{P,1}$, $i_{P,2}$ and $1$.
	\end{itemize}

	In particular, if $(\tau, \vec{b}_{\bullet,\tau})$ is a special combo, then $\tau$ uniquely determines $\vec{b}_{\bullet,\tau}$. 
\end{lemma}
\begin{proof}
	
	``$\Longleftarrow$''. Let $(\tau, \vec{b}_{\bullet, \tau})$ be any combo for $\TT_1$ which satisfies these two conditions.

	We easily see that for $(\tau, \vec{b}_{\bullet, \tau})$ to be special, it is enough to show that $\tau$ is minimal, which follows directly from $$h(\TT_1)\geq \sum_{P\in \TT_1}\lfloor w(P)\rfloor\quad \textrm{and}\quad \lceil w(p\tau(P)-P)\rceil=\lfloor w(P)\rfloor.$$

	``$\Longrightarrow$''. Assume that $(\tau, \vec{b}_{\bullet, \tau})$ fails one of these conditions. Then it is easy to check that the monomial corresponding to this combo either has degree greater than equal to $h(\TT_1)$ or the exponent of $\widetilde a_{Q_1}$ or $\widetilde a_{Q_2}$ is not maximal, a contraction to $(\tau, \vec{b}_{\bullet, \tau})$ being special.
\end{proof}
\begin{example}\label{an example for minimal permutation}
	The following gives an example of a minimal permutation of $\TT_1$ in the case of Example~\ref{example}.
Let \begin{eqnarray*}
 \tau^{-1}(0, 0)=(0, 0), &  \tau^{-1}(0, 1)=(0, 3), & \tau^{-1}(0, 2)=(0, 6),\\
 \tau^{-1}(0, 3)=(0, 2), &  \tau^{-1}(0, 4)=(0, 5), &  \tau^{-1}(0, 5)=(0, 1),\\
 \tau^{-1}(0, 6)=(0, 1), & \tau^{-1}(1, 0)=(3, 0),& \tau^{-1}(1, 1)=(3, 3),\\
   \tau^{-1}(1, 5)=(3, 1),& \tau^{-1}(2, 0)=(6, 0), &  \tau^{-1}(3, 0)=(2, 0), \\
\tau^{-1}(3, 1)=(2, 3), &\tau^{-1}(3, 3)=(2, 2), &  \tau^{-1}(4, 0)=(5, 0), \\
\tau^{-1}(5, 0)=(1, 0),&\tau^{-1}(5, 1)=(1, 3), &  \tau^{-1}(6, 0)=(4, 0), \\
 \tau^{-1}(1, 2)=(2, 1),& \tau^{-1}(1, 4)=(1, 1), &  \tau^{-1}(2, 1)=(4, 1).  \\
\tau^{-1}(2, 2)=(2, 4), &  \tau^{-1}(2, 3)=(5, 1), & \tau^{-1}(2, 4)=(1, 4),\\
\tau^{-1}(3, 2)=(1, 2), &  \tau^{-1}(4, 1)=(4, 2), & \tau^{-1}(4, 2)=(1, 5).\\
\end{eqnarray*}
 From the last statement in Lemma~\ref{maximum of exponent}, it determines a unique special combo. We leave it to the reader to complete its corresponding special combo.
\end{example}

\begin{lemma}\label{like term of special combo}
	There is at least one special (optimal) combo among all combos for $\TT_1$. 
\end{lemma}
\begin{proof}
	In Definition~\ref{maximal combo to YY}, we will give a correspondence between the set of special combos and the set of special bijections (See Definition~\ref{maximal combo to YY}); and in \eqref{Construction of beta}, we  construct an explicit special bijection $\widetilde \beta$. This lemma follows from an easy check that the combo corresponding to $\widetilde \beta$  is special.  

	Since the construction of $\widetilde \beta$ requires nothing but $p>2d+1$ and $d$ to be relatively large with respect to $p_0$ (the residue of $p$ modulo $d$), this is not a circular argument.
\end{proof}

\begin{definition}
	We write $\widetilde{v}_{h(\TT_1)}^\sp$ for $$\sum_{(\tau, \vec{b}_{\bullet, \tau})\ \textrm{special}}\sgn(\tau)\prod_{P\in \TT_1} \prod\limits_{i=1}^{\mathbbm{x}'_1}\frac{( \widetilde a_{Q_i})^{b_{P,\tau,i}}}{b_{P,\tau,i}!},$$
	where the sum runs over all special combos.
\end{definition}
By Lemma~\ref{like term of special combo}, for Theorem~\ref{Thm for 4} to hold, it is enough to prove the following.
\begin{proposition}\label{subcombo}

There is a monomial in $\widetilde{v}_{h(\TT_1)}^\sp$ with coefficient not divisible by $p$.
\end{proposition}
Its proof will be given later.


%
%
%

	By the last statement of Lemma~\ref{maximum of exponent}, we are reduced to studying minimal permutations in special combos, which will be further reduced by the correspondence given in Definition~\ref{maximal combo to YY} soon.

	\begin{definition}
		For each point $P\in \square_\Delta$, we call point $(d,d)-P$ its \emph{mirror reflection} and denote it by $m(P)$.
	\end{definition}
	\begin{notation}
		Let $\YY$ be the set consisting of all lattice points strictly inside the upper right triangle of $\square_\Delta$.
		We put $$\YY_0:=\Big\{(pQ)\%\;\big|\;Q\in \TT_{1,2}\Big\}$$ to be a subset of $\YY$.
	\end{notation}


\begin{lemma}\label{symmetric to y=d-x}
	We have $$\Big\{(pP)\%\
	\;\big|\;P\in \TT_1\Big\}=\YY_0\sqcup (\TT_1\backslash m(\YY_0)).$$
\end{lemma}

\begin{proof}
	Suppose that there exists a point $P_0\in \TT_1$ such that $P_0 = (pQ_0)\%$ and $m(P_0) = pQ_1\%$ for two points $Q_0, Q_1\in \TT_1$.
Let $p'$ be an integer such that $p'p\equiv1\pmod d$. We know that $(p'P_0)\%=Q_0$ and $[p'm(P_0)]\%=Q_1$ are mirror reflections, a contradiction.
\end{proof}

Figure~\ref{picture for distribution of dp} shows the distribution of $\{(pP)\%\
|P\in \TT_1\}$ in the case of Example~\ref{example}, where ``$\bullet$'' and ``$\times$'' represent points in $\YY_0$ and $\{(pP)\%\
|P\in \TT_1\}\backslash \YY_0$ respectively.

	\begin{figure}[h]
		\begin{tikzpicture}
		\tkzInit[xmax=7,ymax=7,xmin=0,ymin=0]
		\tkzGrid
		\tkzAxeXY
		\foreach \Point in {(0,0), (0,1), (0,2), (0,3), (0,4), (0,5), (0,6), (0,1), (1,0) , (1,3), (2,0), (2,2), (2,3), (3,0), (3,1), (3,2), (3,3), (4,0), (5,0), (6,0)}{
			\node at \Point {$\times$};
		}
		\foreach \Point in {(2,6), (3,5), (3,6), (5,3), (5,6), (6,2), (6,3), (6,5), (6,6)}{
			\node at \Point {$\bullet$};
		}
		\foreach \Point in {(1,1), (1,2), (1,4), (1,5), (2,1), (2,4), (4,1), (4,2), (5,1)}{
			\node at \Point {$\circ$};
		}
		\draw[-] (0,7) -- (7,0);
		\end{tikzpicture}.
		\caption{The distributions of $\{(pP)\%\
			|P\in \TT_1\}$ when $d=7$ and $p=17$.}\label{picture for distribution of dp}
	\end{figure}
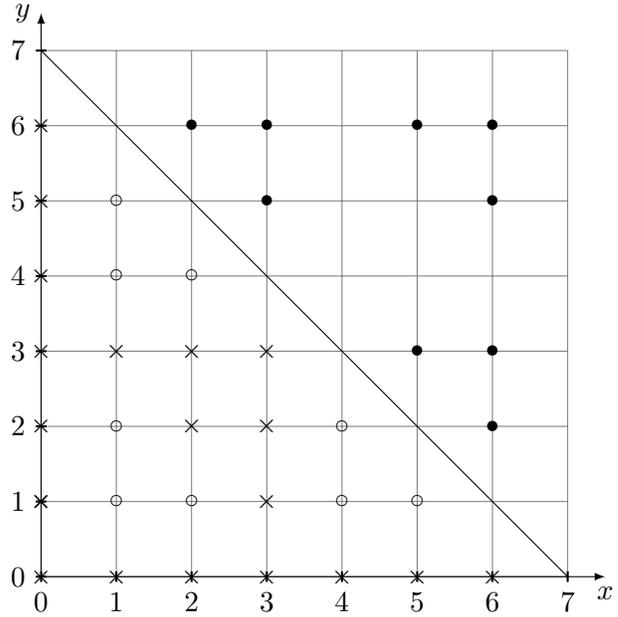

\begin{definition}\label{maximal combo to YY}
	A bijection $\beta:\YY_0\to m(\YY_0)$ is called \emph{special} if $P-\beta(P)\in \TT_1'$ for each point $P\in \YY_0$.
	
	We define a one-to-one correspondence between special combos and special bijections $\beta:\YY_0\to m(\YY_0)$ is defined as follows:

	For a special combo $\big(\tau, \vec{b}_{\bullet, \tau}\big)$, we assign it a special bijection $\beta$ from $\YY_0$ to $m(\YY_0)$ given by	$$\beta(P)=\tau^{-1}((p'P)\%)$$
	for each $P\in \YY_0$.

	In the opposite direction, for a special bijection $\beta: \YY_0\to m(\YY_0)$, we assign to it a special combo $\tau$ given by
	$$\tau^{-1}(P)=\begin{cases}
	(pP)\%& \textrm{if}\ P\in  \TT_{1,1},\\
	\beta((pP)\%)& \textrm{if}\  P\in  \TT_{1,2}.
	\end{cases}$$
	
	Since the composite of these map is the identity map, it is truly a one-to-one correspondence.
\end{definition}
\begin{example}
	The special bijection $\beta: \YY_0\to m(\YY_0)$ corresponding to the special combo in Example~\ref{an example for minimal permutation} is given by
\begin{eqnarray*}
\beta(2, 6)=(1,2),&\beta(3,5)=(1,1),&\beta(3,6)=(2,1),\\
\beta(5, 3)=(4,2),&\beta(5,6)=(1,5),&\beta(6,2)=(5,1),\\
\beta(6, 3)=(4,1),&\beta(6,5)=(1,4),&\beta(6,6)=(2,4).\\
\end{eqnarray*}
\end{example}


\begin{notation} 
	\begin{enumerate}
		\item For a special bijection $\beta$, we let $\tau(\beta) \in \Iso(\TT_1)$ denote the minimal permutation of the corresponding special combo. In view of Lemma~\ref{maximum of exponent}, $\tau(\beta)$ uniquely determines $\beta$.
		\item The composite $m\circ \beta$ can be viewed as a permutation of $\YY_0$. Then we denote by $\sgn(\beta)$ the sign of this permutation and also call it the \emph{sign of $\beta$}.
	\end{enumerate}
\end{notation}

\begin{lemma}\label{same terms}
	We have \begin{equation}\label{vsp}
	\widetilde{v}_{h(\TT_1)}^\sp=\sum_{\beta\ \textrm{special}}\sgn(\beta)\prod_{P\in  \TT_{1}} \prod\limits_{i=1}^{\mathbbm{x}'_1}\frac{(\widetilde a_{Q_i})^{b_{P,\tau(\beta),i}}}{b_{P,\tau(\beta),i}!}.
	\end{equation}
\end{lemma}
\begin{proof}
	First, by the one-to-one correspondence in Definition~\ref{maximal combo to YY}, we know that the sum of $\widetilde{v}_{h(\TT_1)}^\sp$ over all special combos is the same as the sum over all special $\beta$'s.
	Let $\beta$ be special and let $\tau(\beta)$ be the corresponding minimal permutation of $\TT_1$. Since the restriction of $\tau(\beta)$ to $\tau(\beta)^{-1}(\TT_{1,1})=\TT_1\backslash m(\YY_0)$ is symmetric about $y=x$, we know that $\sgn(\tau(\beta))$ depends only on $\sgn(\beta)$. More precisely, we have \[\sgn(\beta)=\sgn\big(\tau(\beta)\big),\qedhere\]
	which completes the proof of this lemma.
\end{proof}
\begin{lemma}\label{monomial}
	$(1)$ The contribution to $\widetilde{v}_{h(\TT_1)}^\sp$ in \eqref{vsp} of terms coming from $P \in \TT_{1,1}$ is same for all special bijections, namely, for two special bijections $\beta_1,\beta_2: \YY_0\to m(\YY_0)$, we have $$\prod_{P\in  \TT_{1,1}} \prod\limits_{i=1}^{\mathbbm{x}'_1}\frac{(\widetilde a_{Q_i})^{b_{\tau(\beta_1)^{-1}(P),\tau(\beta_1),i}}}{b_{\tau(\beta_1)^{-1}(P),\tau(\beta_1),i}!}=\prod_{P\in \TT_{1,1}} \prod\limits_{i=1}^{\mathbbm{x}'_1}\frac{(\widetilde  a_{Q_i})^{b_{\tau(\beta_2)^{-1}(P),\tau(\beta_2),i}}}{b_{\tau(\beta_2)^{-1}(P),\tau(\beta_2),i}!}.$$
\noindent$(2)$ For the contributions of terms coming from $P \in \TT_{1,2}$, we have that the equality \begin{equation*}
\prod_{P\in  \TT_{1,2}} \prod\limits_{i=1}^{\mathbbm{x}'_1}\frac{(\widetilde  a_{Q_i})^{b_{\tau(\beta_1)^{-1}(P),\tau(\beta_1),i}}}{b_{\tau(\beta_1)^{-1}(P),\tau(\beta_1),i}!}=\prod_{P\in  \TT_{1,2}} \prod\limits_{i=1}^{\mathbbm{x}'_1}\frac{(\widetilde a_{Q_i})^{b_{\tau(\beta_2)^{-1}(P),\tau(\beta_2),i}}}{b_{\tau(\beta_2)^{-1}(P),\tau(\beta_2),i}!}
\end{equation*}  holds if and only if we have the following equality 
\begin{equation}\label{congruence}
\Big\{P-\beta_1(P)\;\big |\;P\in \YY_0\Big\}^\star=\Big\{P-\beta_2(P)\;\big |\;P\in \YY_0\Big\}^\star\end{equation}
	 as multisets.
\end{lemma}
\begin{proof}

The first statement directly follows from condition (1) for a special permutation in Lemma~\ref{maximum of exponent}.

For any special bijection $\beta$, we have \[\begin{split}
&\prod_{P\in  \TT_{1,2}} \prod\limits_{i=1}^{\mathbbm{x}'_1}\frac{(\widetilde  a_{Q_i})^{b_{\tau(\beta)^{-1}(P),\tau(\beta),i}}}{b_{\tau(\beta)^{-1}(P),\tau(\beta),i}!}\\
=&\prod_{P\in  \TT_{1,2}} \prod\limits_{i=1}^{2}\frac{(\widetilde  a_{Q_i})^{b_{\tau(\beta)^{-1}(P),\tau(\beta),i}}}{b_{\tau(\beta)^{-1}(P),\tau(\beta),i}!}\times \prod_{P\in  \TT_{1,2}}\widetilde a_{(pP)\%-\tau(\beta)^{-1}(P)}\\
=&\prod_{P\in  \TT_{1,2}} \prod\limits_{i=1}^{2}\frac{(\widetilde  a_{Q_i})^{b_{\tau(\beta)^{-1}(P),\tau(\beta),i}}}{b_{\tau(\beta)^{-1}(P),\tau(\beta),i}!}\times \prod_{P\in  \YY_0}\widetilde a_{P-\beta(P)}.
\end{split}\]
Since $$\prod_{P\in  \TT_{1,2}} \prod\limits_{i=1}^{2}\frac{(\widetilde  a_{Q_i})^{b_{\tau(\beta)^{-1}(P),\tau(\beta),i}}}{b_{\tau(\beta)^{-1}(P),\tau(\beta),i}!}$$ is same to all special permutations, we complete the proof of the second statement.
\end{proof}

\begin{definition}\label{definition for congruence}
	We call $\beta, \beta':\YY_0\to m(\YY_0)$ \emph{related} if they satisfy equality~\eqref{congruence}.
\end{definition}

\begin{corollary}\label{same monomial}
	If $\beta_1,\beta_2: \YY_0\to m(\YY_0)$ are two related special bijections and $\sgn(\beta_1)=\sgn(\beta_2)$, then they contribute to a same monomial in $\widetilde{v}_{h(\TT_1)}$.
\end{corollary}

\begin{proposition}\label{core proposition}
	 $(1)$ There exists a special $\widetilde\beta:\YY_0\to m(\YY_0)$ such that every $\beta'$ related to $\widetilde\beta$ is even, i.e. $\sgn(\beta')=1$, and the number of such $\beta'$ is equal to $2^i$ for some integer $i$. 
	 
	 $(2)$ Therefore, there exists a monomial in $\widetilde v_{\mathbbm{x}_1, h(\TT_1)}$ such that its coefficient is in the form of $\frac{2^i}{\mathscr N_1}$, where $N_1$ is an integer which is not divisible by $p$.
\end{proposition}
Its proof will be completed in section 5.

Theorems~\ref{Thm for 4}, \ref{generic newton polygon} and \ref{theorem for L} would follow from this proposition.
	\begin{proof}[proof of Theorem~\ref{Thm for 4} assuming Proposition~\ref{core proposition}]
This theorem follows directly from Proposition~\ref{core proposition}, Proposition~\ref{reduce to Ef} and Theorem~\ref{main}.
\end{proof}

\begin{proof}[Proof of Theorem~\ref{generic newton polygon}]
	It is easy to check that $d\geq 24(2p_0^2+p_0)$ satisfies \eqref{a weaker condition}. Therefore, the only task left is to compute $\mathbbm x_k$, $\mathbbm x'_k$ and $h(\mathbbm x_k)$ explicitly. It follows directly from applying Lemma~\ref{xk and hk} to this specific $\Delta$. 
\end{proof}

\begin{proof}[Proof of Theorem~\ref{theorem for L}]
	Its proof follows from Theorem~\ref{generic newton polygon} and a consideration of Poincar\'e duality.
\end{proof}

\section{The case when $\Delta$ is an isosceles right triangle II.}

\subsection{Overview} 
The goal of this section is to prove Proposition~\ref{core proposition} by constructing explicitly the special bijection $\widetilde\beta: \YY_0 \to m(\YY_0)$. This is done in several steps. First, for a large subset $\LL_1$ of $\YY_0$, we shall define a bijection $\widetilde\beta_1\ (\textrm{i.e. }\widetilde\beta|_{\LL_1}): \LL_1 \to m(\LL_1)$ which is ``diagonal'', namely the line segment $\overline{\widetilde\beta_1(P)P}$ is parallel to the line $y=x$. For the remaining points in $\YY_0$, we divide them into two subsets $\LL_2$ and $\LL_3$ as in \eqref{L2}, where  $\LL_2$ is contained in $\KK_1$ (see the blue region in Figure~\ref{K1 and K2}) as will be proved in Proposition~\ref{inclusion}, and $\LL_3$ is contained in the green region in Figure~\ref{K1 and K2} by definition.

The map  $\widetilde\beta_2\ (\textrm{i.e. }\widetilde{\beta}|_{\LL_2})$  will map $\LL_2$ into the union of appropriate shifts of the subset $\KK_2$ (see the yellow region in Figure~\ref{K1 and K2}). More precisely, we write $\LL_2$ as the disjoint union $\LL_{2, i_1}\sqcup \cdots \sqcup\LL_{2, i_r}$ (for some non-negative integers $i_1,\dots,i_r$) and $\widetilde \beta_2$ is the union of maps $L_{2, i_k} \to (\KK_2+(i_kp_0,-i_kp_0))\cap m(\YY_0)$  such that the line segments $\overline{\widetilde \beta_2(P) P}$ are parallel for all points $P$ in a fixed $L_{2, i_k}$.
We extend $\widetilde \beta_2$ to a map $\mathbbm s(\widetilde \beta_2)$ on $\LL_2 \sqcup m(\ran(\LL_2))$ by requiring $\mathbbm s(\widetilde \beta_2)(P)=m\circ\beta^{-1}\circ m(P)$ for any point $P\in m(\ran(\beta))$ and hence determine the preimages of points in $m(\LL_2)$ under the map $\widetilde{\beta}$ (see the pink region in Figure~\ref{K1 and K2}).  
At last, we write $\widetilde\beta_3\ (\textrm{i.e.\ } \widetilde{\beta}|_{\YY_0\backslash (\LL_1\cup \dom(\mathbbm s(\widetilde\beta_2))})$ for the unique diagonal symmetric bijection from $\YY_0\backslash \Big(\LL_1\cup \dom(\mathbbm s(\widetilde\beta_2))\Big)$ to $m(\YY_0)\backslash \Big(m(\LL_1)\cup \ran(\mathbbm s(\widetilde\beta_2))\Big)$. Finally we will show that $\widetilde\beta_1, \mathbbm s(\widetilde\beta_2)$ and $\widetilde\beta_3$ altogether define the needed special bijection $\widetilde\beta: \YY_0 \to m(\YY_0)$.

 \subsection{Construction of $\widetilde \beta_1$.}


%
%
%
%
%
%
%
%
%
%
%
%
%
%
 
%

 \begin{hypothesis}\label{first hypothesis}
 Recall that we put $p_0=p\%d$. From now on we assume that $p>2d+1$ and	$p_0< \frac{d}{6}$.
 \end{hypothesis}

 \begin{notation}\label{basic notations}
 Here is a list of symbols:
 	\begin{itemize}
 		\item ${\mathscr D}_k$: the set consisting of all lattice points on the diagonal line $y=x+k$.
 		\item ${\mathscr W}_k$: the set consisting of all lattice points on the anti-diagonal line $x+y=k$.
 		\item For an interval $I \subseteq \RR$, we write $\mathscr D_I:=\coprod_{i \in I \cap \ZZ} \mathscr D_i$ and  ${\mathscr W}_I:=\coprod_{i \in I \cap \ZZ} {\mathscr W}_i$.		
 		\item $\KK_1:={\mathscr W}_{[2d-3p_0, 2d]}\cap {\mathscr D}_{[-p_0, p_0)}\cap \YY$ (see an example in Figure~\ref{K1 and K2}).
 	\end{itemize}
 	\begin{figure}[h]
 	\begin{tikzpicture}
 	\filldraw[color=blue!50] 
 (8,8) -- (7.25,8) -- (6.5,7.25)--(7.25,6.5)--(8,7.25)--cycle;
 	\filldraw[color=yellow] 
 (6,2) -- (4.875,0.875) -- (5.625,0.175)--(6.75,1.25)--cycle;
 	\filldraw[color=green!20] 
 (8,0) -- (8,4) -- (4,8)--(0,8)--cycle;
 \filldraw[color=pink] 
 (0,0) -- (0.75,0) -- (1.5,0.75)--(0.75,1.5)--(0,0.75)--cycle;
 
 \draw[->] (-0,0) -- (10,0) node[right] {$x$};
 \draw[->] (0,-0) -- (0,10) node[above] {$y$};
 \draw[-] (0,8) -- (8,8)--(8,0)--(0,8);
 \draw(7.5,7.5)node[below]{$\KK_1$};
  \draw(0.5,0.3)node[above]{$m(\KK_1)$};
 \draw(6,1.5)node[below]{$\KK_2$};
 
 	\end{tikzpicture}.
 	\caption{Regions $\KK_1$ and $\KK_2$ when $d=16$, $p=19$.} \label{K1 and K2}
 \end{figure}
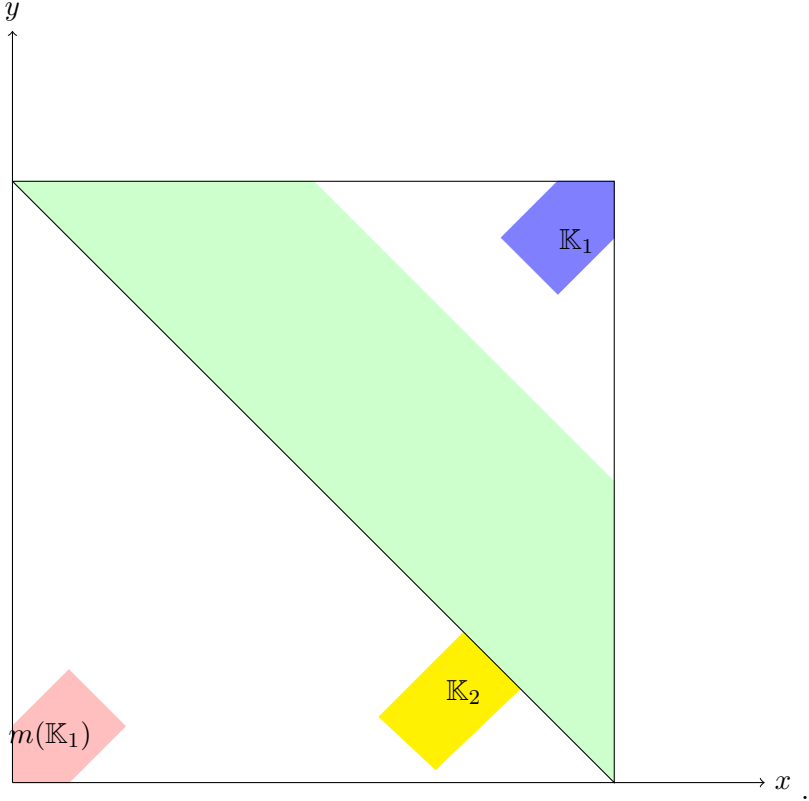

  \end{notation}
		\begin{definition}\label{symmetric}
			Let $\beta$ be an injection from a subset of $\YY_0$ to $m(\YY_0)$. 
			
		\begin{enumerate}
				\item We call $\beta$ \emph{weakly symmetric} if its domain and image are symmetric about the line $y=d-x$; i.e. $\dom(\beta) = m(\dom(\beta))$ and $\ran(\beta) = m(\ran(\beta))$.
				\item We call $\beta$ \emph{symmetric} if each point $P\in \dom(\beta)$ satisfies $$\beta\big(m(\beta(P))\big)=m(P).$$
				\item 	
				If there exists some symmetric map $\beta'$ with domain included in $\YY_0$ such that  $$\dom(\beta')=\dom(\beta)\cup m\big(\ran(\beta)\big)\quad\textrm{and}\quad \beta'\big|_{\dom(\beta)}=\beta,$$then we call $\beta'$ the \emph{symmetric closure} of $\beta$ and denoted it by $\mathbbm s(\beta)$.
		\end{enumerate}	
		\end{definition}

\begin{lemma}\label{symmetric closure}
	Let $\beta$ be an injection from a subset of $\YY_0$ to $m(\YY_0)$. If $\dom(\beta)\cap m(\ran(\beta))=\emptyset$, then  $\mathbbm s(\beta)$ exists.
\end{lemma}
\begin{proof}
	Let \begin{itemize}
		\item $\mathbbm s(\beta)\big|_{\dom(\beta)}=\beta$, and
		\item for any point $P\in m(\ran(\beta))$ let $\mathbbm s(\beta)(P)=m\circ\beta^{-1}\circ m(P).$
	\end{itemize} 
	It is a trivial check that $\mathbbm s(\beta)$ is the symmetric closure of $\beta$.
\end{proof}

\begin{definition}
	A vector is called \emph{diagonal} if it is parallel to the line $y=x$. Let $V^\star$ be a multiset of vectors. We call it \emph{diagonal} if each $\vec{v}\in V^\star$ is diagonal.
	We write $$\mathscr V:=\Big\{V^\star\;\big|\;V^\star\ \textrm{is a diagonal multiset}\Big\}.$$
	We define the weight on vectors on $\RR^2$ so that $w(\overrightarrow{OP}) = w(P).$

	For each $V^\star\in \mathscr V$ and each $r\in \RR$, we write $$(V^\star)^{\geq r}:=\Big\{\vec{v}\in V^\star\;\big|\; w(\vec{v})\geq r\Big\}.$$
\end{definition}
We define a total order ``$\prec_1$'' on $\mathscr V$ as follows:
\begin{definition}\label{partial order for diagonal vectors}
	For any two sets $V^\star_1, V^\star_2\in \mathscr V$, we denote $V^\star_1\prec_1 V^\star_2$ if one of the following cases happens:
	\begin{itemize}
		\item[Case 1:] $\#V^\star_1<\#V^\star_2;$
		\item[Case 2:] $\#V^\star_1=\#V^\star_2$ and there exists a real number $r_0$ such that $\#(V^\star_1)^{\geq r}=\#(V^\star_2)^{\geq r}$	for all $ r>r_0$ but  $\#(V^\star_1)^{\geq r_0}<\#(V^\star_2)^{\geq r_0}$.
	\end{itemize}
\end{definition}
To construct $\widetilde \beta$ needed for Proposition~\ref{core proposition}, we shall construct it so that for a largest possible subset $\LL_1 \subseteq \YY_0$, $\overrightarrow{ \widetilde \beta(P)P}$ is diagonal for all $P \in \LL_1$, or equivalently, the set $\{P-\beta(P)\;|\; P \in \YY_0\}$ contains as many diagonal vectors (and as highest weight) as possible.
\begin{definition}\label{Vstar}
	Let $\SS$ be an arbitrary subset of $\YY_0$, and let $\beta:\SS\to m(\YY_0)$ be an injection. We set $$V^\star(\beta):=\Big\{P-\beta(P)\;\big |\;P\in \SS\Big\}^{\star}.$$ If $V^\star(\beta)$ is diagonal, then we also call $\beta$ \emph{diagonal}.
\end{definition}

\begin{definition}\label{definition of beta1}
	We call a pair $(P, Q)$ in $\YY_0\times m(\YY_0)$ \emph{eligible} if it satisfies the following two conditions:
	\begin{itemize}
		\item [(a)]$\overrightarrow{QP}$ is diagonal with weight less than or equal to $1$, and
		\item [(b)]either $w(P)>\frac{3}{2}$ or $w(Q)<\frac{1}{2}$.
	\end{itemize}
\end{definition}

We write $$\mathscr E_1:=\bigcup_{\SS\subset \YY_0}\Big\{\beta:\SS \to m(\YY_0)\;|\; (P, \beta(P))\ \textrm{is eligible for each}\ P\in \SS\Big\},$$where $\SS$ runs over all subsets of $\YY_0$. For simplicity of notation, we put $$V(\mathscr E_1):=\{V^\star(\beta)\;\big|\;\beta\in \mathscr E_1\}.$$
\begin{definition}\label{definition for beta1}
Define $\widetilde \beta_1$ to be an element in $\mathscr E_1$ such that $V^\star(\widetilde \beta_1)$ is a maximal element in $\big(V(\mathscr E_1), \prec_1\big)$.
\end{definition}
In fact, $\widetilde \beta_1$ can be constructed in the following way:
 
 	Assume that $\widetilde \beta_1$ has been defined on some subset of $\YY_0$, say $\YY_0'$. If there is no eligible pair in $\YY_0\backslash(\YY_0)'\times m(\YY_0)\backslash \widetilde \beta_1(\YY_0')$, then we call the definition of $\widetilde \beta_1$ is completed. Otherwise, we choose an eligible pair $(P_0, Q_0)$ in $\YY_0\backslash\YY_0'\times m(\YY_0)\backslash \widetilde \beta_1(\YY_0')$ which maximizes the weight $w(P_0-Q_0)$, and define $\widetilde \beta_1(P_0)=Q_0$. 
\begin{lemma}
	The map $\widetilde \beta_1$ is the unique maximal element in the totally ordered set \\$(V(\mathscr E_1), \prec_1)$.
\end{lemma}
\begin{proof}
	Assume that we have defined $\beta_1$ on the subset $\YY_0'$ by the construction above and $(P_0, Q_0)$ and $(P'_0, Q'_0)$ are two eligeble pairs in $\YY_0\backslash(\YY_0)'\times m(\YY_0)\backslash \widetilde \beta_1(\YY_0')$ which maximize $$w(P_0-Q_0)=w(P'_0-Q'_0).$$ Since $\overrightarrow{Q_0P_0}$ and $\overrightarrow{Q_0'P_0'}$ are required to be diagonal, we know that $P_0\neq P_0'$ and $Q_0\neq Q_0'$. Therefore, the definition of $\widetilde\beta_1$ is independent of the choices of pairs.
\end{proof}
 	Write $\LL_1$ for the domain of $\widetilde \beta_1$. Since we require $\widetilde \beta_1$ to be the maximal element in $V(\mathscr E_1)$, tt is easily known that $\widetilde \beta_1$ is symmetric. Put 
 	\begin{equation}\label{L2}
 	\LL_2:=\Big\{P\in \YY_0\backslash \LL_1 \;\big|\;w(P)>\frac{3}{2}\Big\}\quad \textrm{and}\quad \LL_3:=\Big\{P\in \YY_0\backslash \LL_1\;\big |\;w(P)\leq\frac{3}{2}\Big\}.
 	\end{equation} Then we obtain a disjoint decomposition of $\YY_0$ as 
$$\YY_0=\LL_1\sqcup \LL_2 \sqcup \LL_3.$$ 
We next will\begin{enumerate}
	\item define a map $\widetilde \beta_2$ on $\LL_2$,
	\item find its symmetric closure $\mathbbm s(\widetilde \beta_2)$,
	\item define a map $\widetilde \beta_3$ on the complement of $\LL_1 \cup \dom(\mathbbm s(\widetilde \beta_2))$ in $\YY_0$, and
	\item put together the maps $\widetilde \beta_1, \mathbbm s(\widetilde \beta_2),$ and $\widetilde \beta_3$ to get a bijection $\widetilde\beta:\YY_0\to m(\YY_0)$ which satisfies the conditions in Proposition~\ref{core proposition}.
\end{enumerate}

  \subsection{Study of $\LL_2$.}
  Recall that $\LL_2$ defined in \eqref{L2} is the subset where we cannot define $\widetilde\beta$ diagonally and where the weight of the points is strictly bigger than $\frac 3 2$."

 In this subsection, we will complete the definition of $\widetilde \beta_2:\LL_2\to m(\YY_0)$. We start with a proposition about the distribution of its domain $\LL_2$ in $\YY_0$, which plays an important role in its construction.
 
 \begin{proposition}\label{inclusion}
 The subset $\LL_2$ is included in $\KK_1$ (See Notation~\ref{basic notations}). 
 \end{proposition}
\begin{proof}														
	The proof will occupy the entire Section 5.2 and it will follow from Propositions~\ref{Class self resolve} and \ref{second restriction} below.
\end{proof}


 \begin{lemma}\label{periodic as mathscr C}
 	Recall that $p_0 = p\%d$. Let $P$ be any point in $\YY_0$. If $P+p_0(i, j)$ is contained in  $\YY$ for a pair of integers $(i, j)$, then it is also contained in $\YY_0$.
 \end{lemma}
 \begin{proof}
 	Before proving the lemma, we refer to Figure\ref{picture for distribution of dp}, where the bullet points in the upper-right triangle are periodic with period $3$.
 	Since the upper-right triangle in $\square_\Delta$ is convex, it is enough to show that the lemma holds for $(i,j) = (1,0),  (-1, 0), (0, 1), (0,-1), (1, -1)$, and $(-1, 1).$
 	We will just prove the case when $i=1$ and $j=-1$, and the rest can be handled similarly.
 	
 	Let $Q$ be the point in $\TT_1$ such that $(pQ)\%=P$. It is easy to check that  \[\begin{split} 
 	P+p_0(1, -1)	&\equiv pQ+p_0(1, -1)\\&\equiv pQ+p(1, -1)\\&=p(Q+(1,-1))\pmod d.
 	\end{split}\]
 	In fact, the point $Q+(1, -1)$ is strictly contained in $\Delta_f$, for otherwise $P+p_0(1, -1)$ is on the boundary of $\square_{\Delta_f}$, which is a contradiction to $P+p_0(1, -1)\in \YY$. Then by the definition of $\YY_0$, we know that $P+p_0(1, -1)$ belongs to  $\YY_0$.
 \end{proof}
\begin{notation}
	We call the square with vertices $(d-p_0, d-p_0)$, $(d-p_0, d-1)$, $(d-1, d-p_0)$ and $(d-1, d-1)$ \emph{the fundamental cell}, denoted by $\mathscr C$, and write $$\mathscr C_0:=\mathscr C\cap \YY_0.$$
\end{notation}
Back to the example in Figure~\ref{picture for distribution of dp}, the corresponding subset $$\mathscr C_0=\{(5,6),(6,5),(6,6)\}.$$

 \begin{corollary}\label{period}
 We know that $\YY_0$ distributes periodically in $\YY$ of period $p_0$.
 	More precisely, each point in $\YY_0$ is a shift of some point in $\mathscr C_0$ by $(ip_0, jp_0)$, where $(i, j)$ is a pair of integers. 
 \end{corollary}
 \begin{proof}
 	It follows directly from Lemma~\ref{periodic as mathscr C}.
 \end{proof}
 \begin{corollary}\label{period1}
 	Let $k_1,k_2$, $j_1$ and $j_2$ be integers satisfying $k_1,k_2\geq d$ and $p_0|(k_2-k_1)$. If both ${\mathscr W}_{k_1}\cap {\mathscr D}_{[j_1, j_1+2p_0)}$ and  ${\mathscr W}_{k_2}\cap {\mathscr D}_{[j_2, j_2+2p_0)}$ are contained in $\square_\Delta$, then\begin{enumerate}
 		\item  $\#\Big(\YY_0\cap {\mathscr W}_{k_1}\cap {\mathscr D}_{[j_1, j_1+2p_0)}\Big)=\#\Big(\YY_0\cap {\mathscr W}_{k_2}\cap {\mathscr D}_{[j_2, j_2+2p_0)}\Big).$
 	
 	\item If moreover we have $p_0|(j_2-j_1)$, we have the following equality of sets $$\YY_0\cap {\mathscr W}_{k_1}\cap {\mathscr D}_{[j_1, j_1+2p_0)}=\YY_0\cap {\mathscr W}_{k_2}\cap {\mathscr D}_{[j_2, j_2+2p_0)}+(k_1-k_2)(1,1)+(j_1-j_2)(-1,1).$$
 	\end{enumerate}
 \end{corollary}
\begin{proof}
	In fact, this corollary follows directly from previous corollary.
\end{proof}
Since $\YY_0$ is distributed periodically of period $p_0$, by Corollary~\ref{period}, it is enough for us to understand $\mathscr C_0 $.
The following two lemmas show the details.

 \begin{lemma}\label{fundamental cell}
 	The distribution of $\mathscr C_0$ in $\mathscr C$ has the following properties:
 	\begin{itemize}
 		\item[(1)] There is no point of $\YY_0$ (or $\mathscr C_0$) on the top row or the first column of $\mathscr C$.
 		\item[(2)] If the point $(i,j)$ is in $\mathscr C$ and $i+j=2d-p_0$, then it is also in $\mathscr C_0$.
 		\item[(3)] For each point $P$ in $\mathscr C$, either $P$ or $(2d-p_0, 2d-p_0)-P$ is contained in $\mathscr C_0$.
 
 	\end{itemize}
 \end{lemma}

 \begin{proof}
 	The first two statements are straightforward. Therefore, we only prove Property (3).

 	Let $(i_1, j_1)$ and $(i_2, j_2)$ be two points in $\mathscr C$ symmetric about $y=2d-p_0-x$. Without loss of generality, we assume that the weight of $(i_1, j_1)$ is less than the weight of  $(i_2, j_2)$. Then it is easy to  check that they satisfy
 	\begin{itemize}
 		\item $2(d-p_0)\leq i_1+j_1<2d-p_0$,
 		\item $i_1+j_2=2d-p_0$, and 
 		\item $i_2+j_1=2d-p_0.$
 	\end{itemize} 
 	Suppose that both $(i_1, j_1)$ and $(i_2, j_2)$ are in $\mathscr C_0$. Then there are two points $(i'_1, j'_1)$ and $(i'_2, j'_2)$ in $\TT_1$ such that $$(p(i'_1, j'_1))\%=(i_1, i_1)\quad \textrm{and} \quad(p(i'_2, j'_2))\%= (i_2, i_2).$$
 	It is easy to show that $$p(i'_1+j'_2)\equiv 2d-p_0 \pmod d\quad\textrm{and}\quad p(i'_2+j'_1)\equiv 2d-p_0 \pmod d.$$ 
 	Since $(d, p)=1$ and $p\equiv p_0 \pmod d$, we have $$i'_1+j'_2\equiv 
 	-1 \pmod d\quad \textrm{and}\quad i'_2+j'_1\equiv 
 	-1 \pmod d.$$ 
 	
 	Combining these two congruence equations with $$d\cdot w(i'_1, j'_1)=i'_1+j'_1\leq d-1\quad \textrm{and}\quad d\cdot w(i'_2, j'_2)=i'_2+j'_2\leq d-1,$$ we get that 
 	$$d\cdot w(i'_1, j'_1)=d-1\quad\textrm{and}\quad d\cdot w(i'_2, j'_2)=d-1.$$ 
 	It forces $p(i'_1+j'_1)\equiv -p_0\pmod d$, which is a contradiction to $$2(d-p_0)\leq i_1+j_1<2d-p_0.$$
 	
 	By a similar argument, we check that at least one of $(i_1, j_1)$ and $(i_2, j_2)$ belongs to $\mathscr C_0$, which completes the proof. 
 \end{proof}
\begin{notation}\label{d1 d2}
	(1) Let $d=d_1p_0+d_0$ and let $0\leq d_2<p_0$ be the integer such that $d_0d_2\equiv 1 \pmod {p_0}$.
	
	(2) For any two points $P_1, P_2$ in $\ZZ_{\geq 0}^2$, if they satisfy that $P_1-P_2=p_0Q$ for some point $Q$ in $\ZZ_{\geq 0}^2$, then we denote $P_1\equiv P_2 \pmod {p_0}$. 
\end{notation}

   \begin{proposition}\label{lemma2 for destribution}
 	For any $0< k\leq p_0$, we have $$\#\big(({\mathscr W}_{2d-k}\cup {\mathscr W}_{2d-p_0-k})\cap \mathscr C_0\big)=\begin{cases}
 	p_0-1& \textrm{if}\ k=p_0;\\
 	kd_2\% p_0-1& \textrm{otherwise}.
 	\end{cases}$$
 \end{proposition}
We need some preparations before giving the proof of this proposition after Lemma~\ref{0 or p0}.

  \begin{lemma}\label{lemma 1 for s}
  	Let $i$, $j$ be two positive integers. If $i+j\leq p_0$, then 
  	$$\big(p(id_1, jd_1)\big)\% \in \YY_0.$$
  \end{lemma}
\begin{proof}
Since $p(id_1, jd_1)\equiv (d-id_0, d-jd_0) \pmod d,$
it is enough to prove $$1<w\big((d-id_0, d-jd_0)\big)<2,$$
which follows directly from $i+j\leq d$.
\end{proof}

 \begin{notation}
 	Put $$\AAA:=\Big\{(id_1,jd_1)\;\big|\;i, j>0\ \textrm{and}\ i+j\leq p_0\Big\}.$$
 \end{notation}

 By Lemmas~\ref{lemma 1 for s} and \ref{periodic as mathscr C}, for any point $P\in \AAA$ there exists a point $P'$ in $\mathscr C_0$ such that $P'\equiv(pP)\%\pmod {p_0}$. It automatically gives us a map from $\AAA$ to $\mathscr C_0$, denoted by $\gamma$. Now we will show that $\gamma$ is a bijection.
 \begin{notation}
 		For simplicity of notation, we put $P_{i,j}:=(id_1,jd_1)$. 
 \end{notation}
 \begin{lemma}
 	 The map $\gamma$ is a bijection.
 \end{lemma}
 \begin{proof}
 	Any two points $P_{i_1,j_1}$ and $P_{i_2,j_2}$ in $\AAA$ satisfy
 	\[\begin{split}
 	\gamma(P_{i_1,j_1})-\gamma(P_{i_2,j_2})\equiv& (pP_{i_1,j_1})\%-(pP_{i_2,j_2})\% \\=&(d-i_1d_0, d-j_1d_0)-(d-i_2d_0, d-j_2d_0)
 	\\=&((i_2-i_1)d_0, (j_2-j_1)d_0)\pmod {p_0}.
 	\end{split}\]
 	Now if $\gamma(P_{i_1,j_1})=\gamma(P_{i_2,j_2})$, we know that $((i_2-i_1)d_0, (j_2-j_1)d_0)\equiv O\pmod {p_0}$. 
 	
 	Since $$(d_0, p_0)=1, |i_2-i_1|<p_0\ \textrm{and}\ |j_2-j_1|<p_0,$$ we have $i_1=i_2$ and $j_1=j_2$, which implies $\gamma$ is an injection. By Lemma~\ref{fundamental cell}, there are $\frac{p_0(p_0-1)}{2}$ points in $\mathscr C_0$, which is equal to the cardinality of $\AAA$. Therefore, $\gamma$ is a bijection.
 \end{proof}
 \begin{lemma}\label{0 or p0}
 	Any two points $P_{i_1, j_1}$ and $P_{i_2, j_2}$ of the same weight satisfy $$\Big|d\cdot w\big(\gamma(P_{i_1,j_1})-\gamma(P_{i_2,j_2})\big)\Big|=0\ \textrm{or} \ p_0.$$

 \end{lemma}
 \begin{proof}

We know easily that $$\gamma(P_{i_1,j_1})\equiv(d-i_1d_0, d-j_1d_0)\quad \textrm{and}\quad \gamma(P_{i_2,j_2})\equiv(d-i_2d_0, d-j_2d_0)\pmod {p_0},$$ which implies that $$d\cdot w\big(\gamma(P_{i_1,j_1})\big)-(d-i_1d_0+d-j_1d_0)\quad \textrm{and}\quad d\cdot w\big(\gamma(P_{i_2,j_2})\big)-(d-i_2d_0+d-j_2d_0)$$ are both divisible by $p_0$.

 Since $P_{i_1, j_1}$ and $P_{i_2, j_2}$ have the same weight, we know $i_1+j_1=i_2+j_2$. Therefore, we have
 \begin{equation}\label{equation 1}p_0\;\big|\;d\cdot w\big(\gamma(P_{i_1,j_1})-\gamma(P_{i_2,j_2})\big). \end{equation}
 
On the other hand, $\gamma(P_{i_1,j_1})$ and $\gamma(P_{i_2,j_2})$ both belong to $\mathscr C_0$, which together with \eqref{equation 1} force  $\big|d\cdot w\big(\gamma(P_{i_1,j_1})-\gamma(P_{i_2,j_2})\big)\big|$ to be $0$ or $p_0$.
 \end{proof}

 \begin{proof}[Proof of Proposition~\ref{lemma2 for destribution}]
 	By Lemma~\ref{0 or p0}, we know that \[\begin{split}
 	\#\big(({\mathscr W}_{2d-k}\cup {\mathscr W}_{2d-p_0-k})\cap \mathscr C_0\big)=&\#\Big\{(i, j)\;\big |\;(i+j)d_0\equiv k\pmod {p_0}, i,j>0\ \textrm{and}\ i+j<p_0\Big\}\\
 	=&\#\Big\{(i, j)\;\big|\;(i+j)\equiv kd_2\pmod {p_0}, i,j>0\ \textrm{and}\ i+j<p_0\Big\}\\
 	=&\begin{cases}
 	p_0-1& \textrm{if}\ k=p_0;\\
 	kd_2\% p_0-1& \textrm{otherwise}.
 	\end{cases}\\
 	\end{split}\]
 \end{proof}

 The following proposition is the first stepstone of Theorem~\ref{inclusion}.
 \begin{proposition}\label{Class self resolve}
 	For every integer $k$ with $|k|\geq p_0$, we have ${\mathscr D}_k\cap \LL_2=\emptyset$.
 \end{proposition}
  The proof of the proposition will be given after some lemmas.
 \begin{lemma}\label{Pigeonhole principle}
 	Let $\big(P_1,P_2,\dots, P_{p_0}\big)$ be a sequence of consecutive points in ${\mathscr D}_k\cap \YY$ for some $k$. 
 	\begin{enumerate}
 		\item There are exact $ \lfloor\frac{p_0}{2}\rfloor$ points in this sequence belonging to $\YY_0$.
 		\item In particular, if $\ell$ is an integer with $\lceil\frac{p_0}{2}\rceil<\ell\leq p_0$, then at least $\ell-\lceil\frac{p_0}{2}\rceil$  points in the set $\{P_1, P_2,\dots, P_\ell\}$ belong to $\YY_0$.
 	\end{enumerate}
 \end{lemma}
 \begin{proof}
 
(1) It follows directly from Lemma~\ref{fundamental cell} (1)-(3).
 	
 (2) Combining (1) with Pigeonhole principle, we complete the proof of (2).
 \end{proof}

 \begin{lemma}\label{no good pair of points left}
 	There do not exist two points $P_0\in \LL_2$ and $Q_0\in m(\LL_3)$ such that $\overrightarrow{Q_0P_0}$ is a diagonal vector of weight less or equal to $1$.
 \end{lemma}
 \begin{proof}
 	Suppose the lemma were false. Then there exists an integer $k$ such that 
 	$$\Big\{(P, Q)\in (\LL_2\cap {\mathscr D}_k)\times  (m(\LL_3)\cap {\mathscr D}_k)\;\big|\;w(\overrightarrow{QP})\leq 1\Big\}$$
 	 is not empty. We put $(P_1, Q_1)$ to be a pair of points in this set which maximize the weight $w(\overrightarrow{Q_1P_1})$. By the inductive definition of $\widetilde \beta_1$, we can define $\widetilde \beta_1$ on $P_1$ by $\widetilde \beta_1(P_1)=Q_1$, which contradicts to the assumption that  $P_1$ does not belong to $\LL_1$. 
 \end{proof}
\begin{lemma}\label{useful lemma}
	For any integer $k$, there do not exist two points $P\in(\YY_0\backslash \LL_1)\cap {\mathscr D}_k$ and $P'\in\LL_1\cap {\mathscr D}_k$ such that \begin{equation}\label{equation}
	w(P'-\widetilde \beta_1(P'))<w(P-\widetilde \beta_1(P'))\leq 1.
	\end{equation}
\end{lemma}
\begin{proof}
	Suppose that $P$ and $P'$ are two points which satisfy conditions in this lemma. We easily see that inequality~\eqref{equation} violates the requirement in construction of $\widetilde \beta_1$ that $(P', \widetilde \beta_1(P'))$ maximizes $w(P'-\widetilde \beta_1(P'))$, a contradiction.
\end{proof}

\begin{lemma}\label{Pin L11}
For an arbitrary point $P$ in $\mathscr D_k\cap \YY_0\cap \mathscr W_{(\tfrac{3d}{2}, 2d]}$, if it satisfies
\begin{equation}\label{interesting equation}
\#\Big\{Q\in {\mathscr D}_k\cap m(\YY_0)\;\big|\;w(P-Q)\leq 1\Big\}\geq \# \Big\{Q\in {\mathscr D}_k\cap \LL_1\;\big|\;w(Q)> w(P)\Big\}+1,
\end{equation} 
then it belongs to $\LL_1$.
\end{lemma}
\begin{proof}
	Let $P$ be a point which satisfies conditions in this lemma. Suppose that $P$ does not belong to $\LL_1$. Then by Lemma~\ref{no good pair of points left}, each element in $$\Big\{Q\in {\mathscr D}_k\cap m(\YY_0)\;\big|\;w(P-Q)\leq 1\Big\}$$ is equal to $\widetilde \beta_1(P')$ for some $P'\in \LL_1\backslash \{P\}$. From equality~\eqref{interesting equation}, we know that at least one of these $P'$ does not belong to $$\Big\{Q\in {\mathscr D}_k\cap \LL_1\;\big|\;w(Q)\geq w(P)\Big\}.$$ Then we obtain a contradiction directly from Lemma~\ref{useful lemma}.
\end{proof}

\begin{corollary}\label{Pin L1}
For an arbitrary point $P$ in $\mathscr D_k\cap \YY_0\cap \mathscr W_{(\tfrac{3d}{2}, 2d]}$, if it satisfies
	 \begin{equation}\label{compare the number of points}
	\#\Big\{Q\in {\mathscr D}_k\cap m(\YY_0)\;\big|\;w(P-Q)\leq 1\Big\}\geq \#\Big\{Q\in {\mathscr D}_k\cap \YY_0\;\big|\;w(Q)\geq w(P)\Big\},
	\end{equation}  
then it belongs to $\LL_1$.
\end{corollary}

\begin{proof}
	Let $P$ be a point satisfying condition in this lemma. Suppose that $P$ does not belong to $\LL_1$. Then we have $$\#\Big\{Q\in {\mathscr D}_k\cap \YY_0\;\big|\;w(Q)\geq w(P)\Big\}\geq \# \Big\{Q\in {\mathscr D}_k\cap \LL_1\;\big|\;w(Q)> w(P)\Big\}+1.$$ Combining it with Lemma~\ref{Pin L11}, we get $P\in \LL_1$, a contradiction.
\end{proof}

Therefore, in order to show that each point $P$ in ${\mathscr D}_k\cap \YY_0\cap\mathscr W_{(\tfrac{3d}{2}, 2d]}$ for $|k|\geq p_0$ belongs to $\LL_1$, it is enough to prove that $P$ satisfies inequality~\eqref{compare the number of points}. The following functions give a lower bound for cardinality of the first set 
in \eqref{compare the number of points} and an upper bound for the second one.

\begin{definition}\label{g1 bound}
	We define $$g_1(2p_0i+j):=\begin{cases}
	i\lfloor\frac{p_0}{2}\rfloor& \textrm{if}\ 0\leq j \leq p_0\\
	i\lfloor\frac{p_0}{2}\rfloor+\lfloor\frac{j}{2}\rfloor-\lceil\frac{p_0}{2}\rceil&\textrm{if}\ p_0<j<2p_0.
	\end{cases}$$ 
\end{definition}
\begin{lemma}\label{g1 bound1}
		For any integers $B_1$, $B_2$ and $k$ with $0<B_1<B_2< d$ and $|k|\leq B_1$, by Lemma~\ref{Pigeonhole principle}, we have
	$$\#\Big(\big\{Q\in {\mathscr D}_k\;\big|\;B_1\leq d\cdot w(Q)\leq B_2\big\}\cap m(\YY_0)\Big)\geq g_1(B_2-B_1).$$ 
\end{lemma}

\begin{definition}\label{g2 bound}
	Define $$g_2(2p_0i+j):=\begin{cases}
	i \lfloor\frac{p_0}{2}\rfloor+\lfloor\frac{j}{2}\rfloor& \textrm{if}\ 0\leq j \leq p_0\\
	(i+1)\lfloor\frac{p_0}{2}\rfloor&\textrm{if}\ p_0<j<2p_0.
	\end{cases}$$
\end{definition}
\begin{lemma}\label{g2 bound1}
	For any integers $B_1$, $B_2$ and $k$ with $d<B_1<B_2< 2d$, by Lemma~\ref{Pigeonhole principle}, we have $$\#\Big(\big\{Q\in {\mathscr D}_k\;\big|\;B_1\leq d\cdot w(Q)\leq B_2\big\}\cap \YY_0\Big)\leq g_2(B_2-B_1).$$
\end{lemma}
 
\begin{lemma}\label{g1 and g2}
	Both $g_2$ and $g_1$ are non-decreasing and $g_1(k+p_0)\geq g_2(k)$ for every $k>0$.
\end{lemma}
\begin{proof}
	It follows from their definitions. 
\end{proof}

 \begin{proof}[\textbf{Proof of Proposition~\ref{Class self resolve}}]
 Consider a point $P_0$ in ${\mathscr D}_k\cap \YY_0\cap \mathscr W_{(\tfrac{3d}{2}, 2d]}$. We have 
 	 that $$\Big\{Q\in {\mathscr D}_k\cap m(\YY_0)\;\big|\;w(P_0-Q)\leq 1\Big\}=\Big\{Q\in {\mathscr D}_k\;\big|\;d\cdot w(P_0)-d\leq d\cdot w(Q)\leq d\Big\}\cap m(\YY_0).$$ By Lemma~\ref{g1 bound1}, we know that 
 	\begin{equation}
 	\#\Big\{Q\in {\mathscr D}_k\cap m(\YY_0)\;\big|\;w(P_0-Q)\leq 1\Big\}\geq g_1(2d-w(P_0)d).
 	\end{equation}
 	
 	On the other hand, since  $$\Big\{Q\in {\mathscr D}_k\cap \YY_0\;\big|\;w(Q)\geq w(P_0)\Big\}=\Big\{Q\in {\mathscr D}_k\;\big|\;d\cdot w(P_0)\leq d\cdot w(Q)\leq 2d-|k|\Big\}\cap \YY_0,$$
 	by Lemma~\ref{g2 bound1}, we know that
  \begin{equation}
  \#\Big\{Q\in {\mathscr D}_k\cap \YY_0\;\big|\;w(Q)\geq w(P_0)\Big\}\leq g_2\big(2d-|k|-d\cdot w(P_0)\big).
  \end{equation}
 	
 By Lemma~\ref{g1 and g2} and $|k|\geq p_0$, the terms on the right side of these inequalities above satisfy
 	\[\begin{split}
 	g_2\big(2d-|k|-w(P_0)d\big)\leq& g_2\big(2d-p_0-w(P_0)d\big)\\
 	\leq& g_1\big(2d-w(P_0)d\big).
 	\end{split}\]
 
 	Hence, we have 
 	\begin{equation}
 	\#\Big\{Q\in {\mathscr D}_k\cap m(\YY_0)\;\big|\;w(P_0-Q)\leq 1\Big\}\geq \#\Big\{Q\in {\mathscr D}_k\cap \YY_0\;\big|\;w(Q)\geq w(P_0)\Big\}.
 	\end{equation} 
 	Combining it with Corollary~\ref{Pin L1}, we prove this proposition.
  \end{proof}

 \begin{proposition}\label{second restriction}
 	The intersection ${\mathscr W}_{[d, 2d-3p_0]}\cap \LL_2$ is empty.
 \end{proposition}
 \begin{proof}
 	By Lemma~\ref{Pin L11}, it is enough to prove that each point $P_0$ in ${\mathscr W}_{(\frac{3}{2}d, 2d-3p_0]}$, say $P_0 \in \mathscr D_k$, satisfies \begin{equation}\label{interesting equation1}
 	\# \Big\{Q\in {\mathscr D}_k\cap \LL_1\;\big|\;w(Q)> w(P_0)\Big\}+1\leq \#\Big\{Q\in {\mathscr D}_k\cap m(\YY_0)\;\big|\;w(P_0-Q)\leq 1\Big\}.
 	\end{equation}

 	Now we estimate the size of the two sets in \eqref{interesting equation1} as follows.
 	
 	By Proposition~\ref{Class self resolve} and the assumption of $P_0$, we are reduced to proving \eqref{interesting equation1} for $P_0$ in ${\mathscr D}_{(-p_0,p_0)}\cap {\mathscr W}_{(\frac{3}{2}d, 2d-3p_0]}$, which guarantees us a point $P'_0$ in $\mathscr C_0$ such that 
 	\begin{equation}\label{P0P'0}
 	P'_0=P_0+(ip_0, ip_0)
 	\end{equation}
 	for some integer $i\geq 1$. Assume that $P_0'$ belongs to $\mathscr W_{j}$.
 	Put $$\BB(P_0):=\Big\{P\in \mathscr C_0\cap \LL_1\cap \mathscr D_k\;\big|\;w(P)\geq w(P'_0)\Big\}.$$
  Then we give the following estimations.
 
\noindent	\textbf {1.  Estimation of $\#\Big\{Q\in {\mathscr D}_k\cap m(\YY_0)\;\big|\;w(P_0-Q)\leq 1\Big\}$.}
	
	By definition of $\BB(P_0)$, we know that $\widetilde \beta_1(\BB(P_0))$ is contained in $$\Big\{Q\in {\mathscr D}_k\cap m(\YY_0)\;\big|\;w(P_0-Q)\leq 1\Big\}\cap {\mathscr W}_{[j-d, d)}.$$

     On the other hand, by \eqref{P0P'0}, we know easily that $${\mathscr D}_k\cap m(\YY_0)\cap {\mathscr W}_{[j-d-2ip_0, j-d)}\subset\Big\{Q\in {\mathscr D}_k\cap m(\YY_0)\;\big|\;w(P_0-Q)\leq 1\Big\}.$$
     Therefore, $\widetilde \beta_1(\BB(P_0))$ and ${\mathscr D}_k\cap m(\YY_0)\cap {\mathscr W}_{[j-d-2ip_0, j-d)}$ are two disjoint subsets of \\
     $\Big\{Q\in {\mathscr D}_k\cap m(\YY_0)\;\big|\;w(P_0-Q)\leq 1\Big\}$. 
     By Lemma~\ref{Pigeonhole principle}, we know that $$\#\Big({\mathscr D}_k\cap m(\YY_0)\cap {\mathscr W}_{[j-d-2ip_0, j-d)}\Big)=i\lfloor\frac{p_0}{2}\rfloor,$$ which implies $$\#\Big\{Q\in {\mathscr D}_k\cap m(\YY_0)\;\big|\;w(P_0-Q)\leq 1\Big\}\leq \#\BB(P_0)+i\lfloor\frac{p_0}{2}\rfloor.$$
\noindent	\textbf {2. Estimation of $	\# \Big\{Q\in {\mathscr D}_k\cap \LL_1\;\big|\;w(Q)> w(P)\Big\}$.}

Consider the disjoint decomposition
 \begin{equation}\label{deomposition 2}
 \Big\{Q\in {\mathscr D}_k\cap \LL_1\;\big|\;w(Q)> w(P)\Big\}=\Big({\mathscr D}_k\cap \LL_1\cap \mathscr W_{(j-2ip_0, j]}\Big)\cup \Big(\BB(P_0)\backslash P'_0\Big).
 \end{equation}
 We need consider the following two cases:
 
 Case 1: When $P_0'$ belongs to $\LL_1$, we have $\#\Big(\BB(P_0)\backslash P'_0\Big)= \#\BB(P_0)-1$, which implies \[\begin{split}
 &\#\Big\{Q\in {\mathscr D}_k\cap \LL_1\;\big|\;w(Q)> w(P)\Big\}\\
 \leq& \#\Big({\mathscr D}_k\cap \mathscr W_{(j-2ip_0, j]}\Big)+\#\Big(\BB(P_0)\backslash P'_0\Big)\\
 = & i\lfloor\frac{p_0}{2}\rfloor+\#\BB(P_0)-1.
 \end{split}\]

 Case 2: When $P_0'$ does not belong to $\LL_1$, we have $$\#\Big({\mathscr D}_k\cap \LL_1\cap \mathscr W_{(j-2ip_0, j]}\Big)\leq \#\Big({\mathscr D}_k\cap \mathscr W_{(j-2ip_0, j]}\Big)-1,$$ which implies 
\[\begin{split}
&\#\Big\{Q\in {\mathscr D}_k\cap \LL_1\;\big|\;w(Q)> w(P)\Big\}\\
\leq& \#\Big({\mathscr D}_k\cap \mathscr W_{(j-2ip_0, j]}\Big)-1+\#\Big(\BB(P_0)\backslash P'_0\Big)\\
= & i\lfloor\frac{p_0}{2}\rfloor+\#\BB(P_0)-1.
\end{split}\]

In either case, it is easy to check \eqref{interesting equation1}, which completes this proposition.
 \end{proof}

%
%
%
 
 \subsection{Definition of $\widetilde{\beta}$.}
We next construct a map $\overline{\beta}_2: \LL_2\to m(\LL_3)$.
Put $\widetilde J=\{d-3p_0,d-3p_0+1,\dots, d-1\}$.  Write \begin{equation}\label{K2}
\KK_2:=\Big\{ P \;\big|\; P\in {\mathscr W}_{\widetilde J}\cap {\mathscr D}_{\big[\lceil \frac{d}{2}\rceil, \lceil \frac{d}{2}\rceil+2p_0\big)}\Big\}\quad \textrm{and}\quad \KK_2^0:=\KK_2\cap m(\YY_0).
\end{equation}

The general idea of constructing $\overline \beta_2$ is to map $\LL_2$ to disjoint sets $(\KK_2+(i_kp_0,-i_kp_0))\cap m(\YY_0)$, where $(i_1, \dots,i_{\#\LL_2})$ is a certain sequence of numbers in some range, such that for any two points $P_1$ and $P_2$ in $\LL_2$ if  $\overline{\beta}_2(P_1)$ and  $\overline{\beta}_2(P_2)$ belong to $(\KK_2+(i_kp_0,-i_kp_0))\cap m(\YY_0)$ for a same $k$, then $P_1-\overline{\beta}_2(P_1)=P_2-\overline{\beta}_2(P_2)$.

\begin{remark}
	An easy computation shows that ${\mathscr W}_{d-1}\cap m(\YY_0)$ is not empty, say that $Q$ is a point in it. The most naive construction of $\overline{\beta}_2$ is to make an injection from $\LL_2$ to $\{Q +(2ip_0,-2ip_0)\}_{i=1}^{\#\LL_2}$. However, the construction requires a very stronge condition that $d=O(p_0^3)$. In order to weaken this condition, we need a more detailized construction (see Construction~\ref{Construction of tilde beta2}).
\end{remark}
We start the construction of $\overline{\beta}_2$ with giving more details of its codomain. Recall that we defined the numbers $d_0$, $d_1$ and $d_2$ in Notation~\ref{d1 d2}. 
\begin{lemma}\label{distribution of R2}
 		We have $$\#(\KK_2^0\cap {\mathscr W}_{d-i})=\begin{cases}
 		p_0-1& \textrm{if}\ i\equiv d_0\pmod{p_0};\\
 		\big(i(p_0-d_2)\big)\%p_0& \textrm{otherwise}
 		\end{cases}$$

 		for all $1\leq i\leq p_0-1$.
\end{lemma}
 \begin{proof}
 Since $\YY_0$ and $m(\YY_0)$ are symmetric about $y=d-x$, we have $$\#\Big(\KK^0_2\cap {\mathscr W}_{d-i}\Big)=\#\Big({\mathscr W}_{d+i}\cap \YY_0\cap {\mathscr D}_{\big[\lceil \frac{d}{2}\rceil, \lceil \frac{d}{2}\rceil+2p_0\big)}\Big).$$
 Find $p_0< j\leq 2p_0$ such that $2d-j\equiv d+i \pmod {p_0}$. By Corollary~\ref{period1}, we have $$\#\Big({\mathscr W}_{d+i}\cap \YY_0\cap {\mathscr D}_{\big[\lceil \frac{d}{2}\rceil, \lceil \frac{d}{2}\rceil+2p_0\big)}\Big)=\#\Big({\mathscr W}_{2d-j}\cap {\mathscr D}_{[-p_0,p_0)}\cap\YY_0\Big).$$
 From Corollary~\ref{period}, we know $$\#\Big({\mathscr W}_{2d-j}\cap {\mathscr W}_{[-p_0,p_0)}\Big)=\#\Big(({\mathscr W}_{2d-j}\cup {\mathscr W}_{2d-p_0-j})\cap \mathscr C_0\Big).$$
 Therefore, by Proposition~\ref{lemma2 for destribution}, if $j=p_0$, then we have $$\#(({\mathscr W}_{2d-j}\cup {\mathscr W}_{2d-p_0-j})\cap \mathscr C_0)=p_0-1.$$
 It is not hard to see from the relation between $i$ and $j$ that $i\equiv d_0 \pmod {p_0}$. Combining all these equlities above, we get $\#(\KK^0_2\cap {\mathscr W}_{d-i})=p_0-1$.
 
 For the case that $j\neq p_0$, we have 
 	\[\begin{split}
 &\#(({\mathscr W}_{2d-j}\cup {\mathscr W}_{2d-p_0-j})\cap \mathscr C_0)\\=&jd_2\% p_0-1\\
 =&(d_0-i)d_2\% p_0-1\\
 =&i(p_0-d_2)\%p_0.
 \end{split}\]
 By a similar argument, we complete the proof immediately.
\end{proof}

In order to support our construction of $\overline{\beta}_2$, we need several technical lemmas. 
 \begin{notation}
 	For any subset $\KK_2'$ of $\KK_2$ and any integer $k\in [0, p_0-1]$, we put 
 	$$\KK_2'(k):=\begin{cases}
 P+(-k, k)& \textrm{if} \ P+(-k, k)\in \KK_2;\\
 P+( p_0-k, -(p_0-k))&\textrm{otherwise.}
 	\end{cases}$$

 \end{notation}

%
%
%
%

 \begin{lemma}\label{drawer}
 	Let $J$ be a subset of $\widetilde J$. Suppose that there are at least $\mathbbm n$ points in ${\mathscr W}_j\cap \KK_2'$ for each $j\in J$. Then
 	
 	\noindent $(1)$ for every subset $\SS$ of ${\mathscr W}_J\cap \KK_2$ of cardinality $\mathbbm m$, there exists at least an integer $i$ in $[0,p_0-1]$ such that 
 	$$\#\Big(\SS\cap \KK_2'(i)\Big)\geq \big\lceil \frac{\mathbbm m\mathbbm n}{p_0}  \big\rceil.$$
 	
 	\noindent $(2)$ For the set of lattice points ${\mathscr W}_J\cap \KK_2$, there exists a subset $I$ of $\{1,2, \dots, p_0\}$ of cardinality less than or equal to $\big\lceil-\log_{(1-\frac{\mathbbm n}{p_0})}\big(p_0(\#J)\big)\big\rceil$ such that 
 	\begin{equation}\label{union}
 	\bigcup_{i\in I}\KK_2'(i)\cap {\mathscr W}_J ={\mathscr W}_J\cap \KK_2.
 	\end{equation}

 \end{lemma}
 \begin{proof}
 	(1). Since $\biguplus\limits_{i=0}^{p_0-1}({\mathscr W}_J\cap \KK_2'(i))^\star$ covers $\SS$ at least $\mathbbm n$ times, by Pigeonhole principle, there exists some $i$ such that  $\SS\cap \KK_2'(i)\geq \big\lceil \frac{\mathbbm m\mathbbm n}{p_0}  \big\rceil$.
 	
 	(2).  By (1), we can choose a sequence $(i_1,i_2,\dots)$ from $\{1,2,\dots, 3p_0\}$ such that 
 	\begin{equation}\label{recursive of m}
 	\mathbbm{m}_k\leq \mathbbm{m}_{k-1}-\big\lceil \frac{\mathbbm n\mathbbm{m}_{k-1}}{p_0}\big\rceil\leq \mathbbm{m}_{k-1}(1-\frac{\mathbbm n}{p_0}),
 	\end{equation}
 	where $\mathbbm{m}_k:=\#({\mathscr W}_J\cap \KK_2-\bigcup\limits_{j=1}^{k} \KK_2'(i_j)\cap {\mathscr W}_J)$.
 	
 	Write $t=\Big\lfloor-\log_{\big(1-\frac{\mathbbm n}{p_0})}(p_0(\#J)\big)\Big\rfloor+1$. Repeated application of \eqref{recursive of m} gives  $$\mathbbm{m}_{t}\leq \mathbbm{m}_0(1-\frac{\mathbbm n}{p_0})^{t}= p_0(\#J)(1-\frac{\mathbbm n}{p_0})^{t}<1.$$
 	It implies $\mathbbm{m}_t=0$. Therefore, the length of this sequence cannot be longer than $t-1$, which completes the proof of (2).
 \end{proof}

Let $u$ be a real number in $(0,1)$. Depending on $u$, we decompose $\widetilde J$ into three groups:
 \begin{enumerate}
 	\item $J_1(u)=\Big\{j\in \widetilde J\;\big|\; j>d-\frac{p_0^u}{p_0-d_2}\Big\}$,
 	\item $J_2(u)=\Big\{j\in \widetilde J\;\big|\; \#({\mathscr W}_j\cap \KK_2^0)\geq p_0^u\Big\}$, and
 	\item $J_3(u):=\widetilde J\big\backslash \Big(J_2(u)\cup J_1(u)\Big)$.
\end{enumerate}
 By Lemma~\ref{distribution of R2}, we know that \begin{equation}\label{card of J2}
 3p_0-3p_0^u\leq \#J_2(u) \leq 3p_0.
 \end{equation}
%
%

\begin{notation}
	Set $h = \log_{p_0}(p_0-d_2)$.
\end{notation}

 \begin{construction}[Construction of $\overline{\beta}_2$]\label{Construction of tilde beta2}
We construct $\overline{\beta}_2$ in three steps:
 
 \noindent \textbf{Step 1.} Lemma~\ref{distribution of R2} shows that $\KK_2^0\cap {\mathscr W}_{d-1}$ is not empty, say that it contains a point $Q_1$.  We put 
 $${\mathscr W}_{J_1+d}\cap \LL_2:=\{P_1, P_2,\dots,P_{t_1}\},$$ where $t_1$  is its cardinality, and define $\overline{\beta}_2$ on ${\mathscr W}_{J_1+d}\cap \LL_2$ as 
 \[\begin{split}
 \overline{\beta}_2: &{\mathscr W}_{J_1+d}\cap \LL_2\to m(\LL_3)\\
 &	P_i\mapsto Q_1+\big(p_0(i-1), -p_0(i-1)\big).
 \end{split}\]
 Namely, $\overline{\beta}_2$ maps ${\mathscr W}_{J_1+d}\cap \LL_2$ into a disjoint union of $\KK_2+(p_0i, -p_0i)$ for $0\leq i\leq t_1-1$.

By the definition of $J_1$, we know that $\#J_1=\lfloor\frac{p_0^u}{p_0-d_2}\rfloor$. Since ${\mathscr W}_{J_1+d}\cap \LL_2$ is in an isosceles right triangle with side lengths $\#J_1$, we have $t_1\leq\frac{1}{2}(\lfloor\frac{p_0^u}{p_0-d_2}\rfloor+1)\lfloor\frac{p_0^u}{p_0-d_2}\rfloor.$
It is easily check that $t_1\leq \frac{1}{2}({p_0}^{2(u-h)}+ {p_0}^{u-h})$. 
  
 \noindent \textbf{Step 2.}  We denote by $\theta$ the  unique map from $\KK_1$ to $\KK_2$ given by parallel transform. By Lemma~\ref{drawer}~(2), there is a sequence $(i_1,i_2,\dots,i_{t_2})$ such that $$\bigcup\limits_{k=1}^{t_2} \KK_2^0(i_k)\cap {\mathscr W}_{J_2}= \KK_2\cap {\mathscr W}_{J_2}\quad\textrm{and}\quad t_2\leq \big\lfloor-\log_{(1-\frac{p_0^u}{p_0})}\big(p_0(\#J_2)\big)\big\rfloor.$$
Since $\bigcup\limits_{k=1}^{t_2} \KK_2^0(i_k)$ depends only on the elements in set $\{i_1,i_2,\dots,i_{t_2}\}$, we can in fact require $(i_1,i_2,\dots,i_{t_2})$ to be increasing.

 It is easily seen that 
  $$t_2\leq \big\lfloor-\log_{(1-\frac{p_0^u}{p_0})}\big(p_0(\#J_2)\big)\big\rfloor\leq \big\lfloor-\log_{(1-\frac{p_0^u}{p_0})}(3p_0^2)\big\rfloor.$$ 
 
 For each point $P$ in ${\mathscr W}_{J_2+d}\cap \LL_2$, we put $k(P)$ to be the smallest number such that $\KK_2^0(i_{k(P)})\cap {\mathscr W}_{J_2}$ contains $\theta(P)$. Then 
 we define \[\overline{\beta}_2(P):=\theta(P)+(x_2(P), -x_2(P)), \]
 where $x_2(P)=t_1p_0+i_{k(P)}+k(P)p_0$.x
  Namely, $\overline{\beta}_2$ maps ${\mathscr W}_{J_2+d}\cap \LL_2$ into a disjoint union of $$\KK_2+\big[t_1p_0+i_{k}+kp_0\big](1,-1)\quad \textrm{for}\ 1\leq k\leq t_2.$$ 
 
 \noindent \textbf{Step 3.} 
Write $J_3(u)=\{j_1,j_2, \dots,j_{s_3}\}$. By \eqref{card of J2}, we know that 
$s_3\leq 3p_0^u$. 
 Let $d_3$ be the largest number in $\widetilde J$ such that  $\#({\mathscr W}_j\cap \KK_2^0)\geq \tfrac{p_0}{2}$. By Lemma~\ref{distribution of R2}, we have $$J_3\in (d-3p_0, d_3]\quad\textrm{and}\quad \frac{p_0}{2}\leq(d-d_3)(p_0-d_2)<p_0.$$ 
 Replacing $J$ in Lemma~\ref{drawer} by $\{d_3\}$, we obtain a sequence  $(i'_1,i'_2,\dots,i'_{t_3})$ from $\{1,2,\dots, p_0\}$ such that $$\bigcup\limits_{k=1}^{t_3} \KK_2^0(i_k)\cap {\mathscr W}_{d_3}= \KK_2\cap {\mathscr W}_{d_3}\quad \textrm{and}\quad t_3\leq \big\lfloor-\log_{(1-\frac{p_0/2}{p_0})}(p_0)\big\rfloor=\big\lfloor\log_2(p_0)\big\rfloor.$$ 
 Similar to \textbf{Step 2}, we assume that $(i'_1,i'_2,\dots,i'_{t_3})$ is increasing. Then we define $\overline{\beta}_2$ on ${\mathscr W}_{J_3+d}\cap \LL_2$ as follows:
 
Consider each $P\in {\mathscr W}_{J_3+d}\cap \LL_2$. Suppose that $P$ belongs to ${\mathscr W}_{j_l+d}$ for some $j_l\in J_3$. Let $k(P)$ be the smallest number such that $\KK_2^0(i'_{k(P)})\cap {\mathscr W}_{d_3}$ contains $\theta(P)+(d_3-j_l, d_3-j_l)$. Then we define \[\overline{\beta}_2(P):=
 \theta(P)+(d_3-j_l+x_3(P), d_3-j_l-x_3(P)),\]
 where $x_3(P)=p_0[k(P)+(l-1)(t_3+2)+t_1+t_2+2]+i'_{k(P)}.$
 
 	 Namely, $\overline{\beta}_2$ maps ${\mathscr W}_{J_2+d}\cap \LL_2$ into a disjoint union of $$\KK_2+\big\{p_0[k+(l-1)(t_3+2)+t_1+t_2+2]+i'_{k}\big\}(1,-1)$$
 	 for $1\leq k\leq t_3$ and $1\leq l\leq s_3$.
\end{construction}

 Notice that the codomain of $\overline{\beta}_2$ is both a disjoint union of shifts of $\KK_2$ and a subset of $\square_\Delta$.  Then for a fixed $d$, the residue $p_0$ of $p$ modulo $d$ cannot be too large. The following computation gives $p_0$ an upper bound such that the construction for $\overline\beta_2$ above is realizable. In fact, the complicated conditions in Theorem~\ref{Thm for 4} are also from this computation.
 
From the constructing above, we know that the image of $\overline{\beta}_2$ is included in a union of disjoint shifts of $\KK_2$. Moreover, the number of these shifts, denoted by $\mathcal N$, is counted and estimated as follows:
 \begin{equation}\label{estimate the order}
 \begin{split}
 \mathcal N=&t_1+t_2+t_3\times s_3\\
 \leq&\frac{1}{2}({p_0}^{2(u-h)}+ {p_0}^{u-h})+\big\lfloor-\log_{(1-\frac{p_0^u}{p_0})}(3p_0^2)\big\rceil+\big\rfloor\log_2(p_0)\big\rceil\times 3p_0^u\\
 \leq& \frac{3}{4}{p_0}^{2(u-h)}+\frac 1 4+2\ln(p_0)p_0^{1-u}+\ln3p_0^{1-u}+\log_2(p_0)\times3 p_0^u.\\
 \end{split}
 \end{equation}
Recall $\KK_2=\Big\{ P \;\big|\; P\in {\mathscr W}_{\widetilde J}\cap {\mathscr D}_{\big[\lfloor \frac{d}{2}\rfloor, \lfloor \frac{d}{2}\rfloor+2p_0\big)}\Big\}$.
It is easy to see that the largest $x$-coordinate of points in the codomain of $\overline{\beta}_2$ is equal to $\frac{3}{4}d+ p_0 (\mathcal N+2s_3+3),$ which obvious is controled by $d$. Then we get a necessary condition:
\begin{equation}\label{estimate the order1}
d\geq 4p_0 (\mathcal N+2s_3+3).
\end{equation}
 
 We write $G(h, u)=\max\{2(u-h), 1-u, u\}$. Recall that $h=\log_{p_0}(p_0-d_2)$ is fixed by $p_0$ and $d$. Therefore, our next goal is to determine the minimum of $G(h, u)$ by varying the value of $u$ inside $[0, 1]$. 
 \begin{definition}
 	We call $u'$ an \emph{optimizer} of $h$ if 
 $G(h, u')=\min\limits_{u\in [0,1]}\big(G(h, u)\big)$.
 \end{definition}
 
In order to get an optimizer of a given $h\in (0, 1]$. We need to consider two cases:.
 \begin{figure}
 	\centering
 	\begin{tikzpicture}[scale=4]
 	\draw[->] (-0,0) -- (1.2,0) node[right] {$x$};
 	
 	\draw[->] (0,-0) -- (0,1.2) node[above] {$y$};
 	\draw[-] (0,1) -- (1,0) node[right] {};
 	\draw[-] (0,0) -- (1.2,1.2) node[right] {};
 	\draw[dashed] (0,-0.5) -- (1,1.5) node[right] {};
 	\draw[dashed] (0,-0.25) -- (1,1.75) node[right] {};
 	\draw[dashed] (0.25,-0.5) -- (1,1) node[right] {};
 	\draw(1.2,1.2)node[right]{$y=x$};
 	\draw(1,1.75)node[right]{$y=2(x-\frac{1}{8})$};
 	\draw(1,1.5)node[right]{$y=2(x-\frac{1}{4})$};
 	\draw(1,1)node[right]{$y=2(x-\frac{1}{2})$};
 	\draw(0,1)node[right]{$y=1-x$};
 	\end{tikzpicture}
 	\caption{Determine the optimizer of $h$.} \label{fig: optimizer}
 \end{figure}
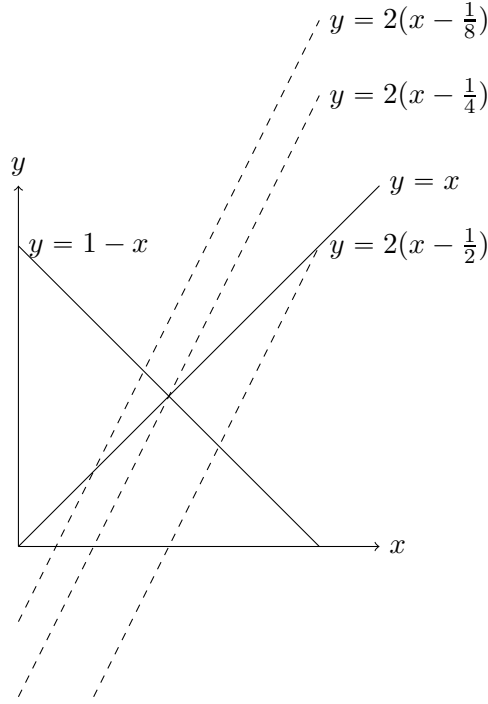

 \noindent \textbf{Case 1.} When $\tfrac{1}{4}\leq h\leq 1$. From Figure~\ref{fig: optimizer}, the optimizer $u$ of $h$ is the $x$-coordinate of the point of intersection of lines $y=x$ and $y=1-x$. Therefore, we know that $u=\tfrac{1}{2}$ is the optimizer of this $h$. Plugging $u=\tfrac{1}{2}$ into equations~\eqref{estimate the order} and \eqref{estimate the order1}, we have
 $$ 4p_0(\mathcal N+2s_3+3)\leq 4p_0^{\frac{3}{2}}\Big[\ln 3+\frac{27}{4}+(2+\frac{3}{\ln 2})\ln p_0\Big]+13p_0.$$
 
 \noindent \textbf{Case 2.} When $0<h<\frac{1}{4}$. Based on the same observation of Figure~\ref{fig: optimizer}, an optimizer $u$ of $h$ is the $x$-coordinate of the point of intersection of lines $y=2(x-h)$ and $y=1-x$. An easy computation shows that $u=\tfrac{1+2h}{3}$. Combining it with \eqref{estimate the order} and \eqref{estimate the order1}, we have $$4p_0(\mathcal N+2s_3+3)\leq 4p_0^{\frac{5-2h}{3}}\Big[\ln 3+\frac{27} 4+(2+\frac{3}{\ln 2})\ln p_0\Big]+13p_0.$$
 
\begin{notation} 
	
(1)	We write \begin{equation}\label{posible vectors}
 V(\overline{\beta}_2):=\Big\{P- \overline{\beta}_2(P)\;\big|\;P\in \LL_2\Big\}.
 \end{equation}

 (2) The reflection of a vector $\vec{v}$ through a diagonal line $y=x$ is denoted by $\vec{v}^\vee$. Let $V$ be a set of vectors. We put $V^{\vee}:=\{\vec{v}^\vee\;|\;\vec{v}\in V\}$.
 \end{notation}

 \begin{lemma}\label{no inter}
	We know that $$\ran(\widetilde{\beta}_1)\cap\Big\{P-\vec{v}\;\big|\;P\in \LL_2\ \textrm{and}\ \vec{v}\in V(\overline{\beta}_2)\cup V(\overline{\beta}_2)^\vee\Big\}$$  is empty.
\end{lemma}
	\begin{proof}
It is easy to check that any point $Q$ in $$\Big\{P-\vec{v}\;\big|\;P\in \LL_2\ \textrm{and}\ \vec{v}\in V(\overline{\beta}_2)\cup V(\overline{\beta}_2)^\vee\Big\}$$ belongs to $\mathscr D_k$ for some $|k|>\frac{d}{2}$. Then this lemma follows simply from the definition of $\widetilde{\beta}_1$.
\end{proof}
 \begin{construction}[Construction of $\widetilde{\beta}_2$]
 	
 	\textbf{Step 1}. Write 
 	$$\mathscr E_2:=\big\{\beta:\LL_2\hookrightarrow m(\YY_0)\;\big|\;P-\beta(P)\in V(\overline{\beta}_2)\cup V(\overline{\beta}_2)^\vee\ \textrm{for all}\ P\in \LL_2\big\}.$$
 	We know that $\mathscr E_2$ is non-empty, for $\overline{\beta}_2$ is automatically contained in it. 
 	
 	\noindent \textbf{Step 2}. We line up the elements in $V(\overline{\beta}_2)$ to form a sequence, denoted by $(\vec{v}_1, \vec{v}_2,\dots, \vec{v}_{\mathscr N})$.
 
 	\noindent \textbf{Step 3}. Define a partial order over $\mathscr E_2$ as follows: 
 	
	For any two maps $\beta_1, \beta_2\in \mathscr E_2$, we denote $\beta_1\prec_2 \beta_2$, if there exists an integer $1\leq k\leq \mathcal N$ such that 
 	\begin{itemize}
 		\item $\#\{P\;|\;P- \beta_1(P)=\vec{v}_i\ \textrm{or}\ \vec{v}_i^\vee\}=\#\{P\;|\;P-\beta_2(P)=\vec{v}_i \ \textrm{or}\  \vec{v}_i^\vee\}$	for all $ 1\leq i\leq k-1$, and\\
 		\item $\#\{P\;|\;P-\beta_1(P)=v_k \ \textrm{or}\  v_k^\vee\}<\#\{P\;|\;P-\beta_2(P)=v_k \ \textrm{or}\  v_k^\vee\}.$
 	\end{itemize}
    
   \noindent  \textbf{Step 4}. Let $\widetilde \beta_2$ be a maximal element in $\mathscr E_2$.
\end{construction}
By Hypothesis~\ref{first hypothesis}, we know that $m(\LL_2)\cap \widetilde \beta_2(\LL_2)=\emptyset.$ Combining it with Lemma~\ref{symmetric closure}, we can simply prove the existence of $\mathbbm s(\widetilde{\beta}_2)$.
   Since maps $\widetilde \beta_1$ and $\mathbbm s(\widetilde{\beta}_2)$ are symmetric, we define \[\begin{split}
 \widetilde \beta_3: \YY_0\Big\backslash\Big( \dom(\widetilde \beta_1)\cup \dom(\mathbbm s(\widetilde \beta_2))\Big)&\to m(\YY_0)\Big\backslash \Big(\ran(\widetilde \beta_1)\cup \ran(\mathbbm s(\widetilde \beta_2))\Big)\\
P &\mapsto m(P).
 \end{split}\] 
 
 Putting $\widetilde\beta_1, \mathbbm s(\widetilde{\beta}_2)$ and $\widetilde\beta_3$ together, we define  a bijection $\widetilde{\beta}: \YY_0\to m(\YY_0)$ such that 
 \begin{equation}\label{Construction of beta}
 \widetilde{\beta}(P)=\begin{cases}
 \widetilde \beta_1(P)& \textrm{if}\ P\in \LL_1;\\
 \mathbbm s(\widetilde \beta_2)(P)& \textrm{if}\ P\in \dom(\mathbbm s(\widetilde \beta_2)); \\
 \widetilde \beta_3(P)& \textrm{otherwise}.
 \end{cases}
 \end{equation}

Since $ \YY_0\Big\backslash\Big( \dom(\widetilde \beta_1)\cup \dom(\mathbbm s(\widetilde \beta_2))\Big)\subset \LL_3,$ we know that $w(P-m(P))\leq 1$ for each point $P\in \LL_3$. Combining it with the constructive definition of $\widetilde \beta_1$ and $\mathbbm s(\widetilde \beta_2)$, we easily check that $\widetilde{\beta}$ is a special bijection (see Definition~\ref{maximal combo to YY}).

\subsection{Completion of the proofs.}
\begin{notation}
	For a bijection $\beta: \YY_0\to m(\YY_0)$, we write $$\LL_{\beta}:=\LL_2\cup \beta^{-1}(m(\LL_2)).$$
\end{notation}
Recall that we defined the meaning of two bijections  $\beta, \beta':\YY_0\to m(\YY_0)$ to be related in Definition~\ref{definition for congruence}. 
  \begin{proposition}\label{thm 2k}
 	\noindent$(a)$ A bijection $\beta':\YY_0\to m(\YY_0)$ is related to $\widetilde{\beta}$
 	if and only if 
 	\begin{itemize}
 	\item[$(1)$] $P-\beta'(P)=P-\widetilde{\beta}(P)$ or $(P-\widetilde{\beta}(P))^\vee$ for all $P$ in $\LL_2$;
 	\item[$(2)$] $\beta'\big|_{\LL_1}=\widetilde\beta\big|_{\LL_1}$;
 	\item[$(3)$] $\beta'(P)=m(P)$ for $P\in \YY_0\backslash(\LL_1\cup\LL_{\beta'})$;
 	\item[$(4)$] $\beta'$ is symmetric.
 	\end{itemize}

 	\noindent$(b)$ The number of bijections related to $\widetilde{\beta}$ is equal to $2^{k}$, where $$k=\#\big\{P\in \LL_2\;|\; P-(P-\widetilde{\beta}_2(P))^\vee\in m(\YY_0)\big\}.$$
 \end{proposition}
 \begin{proof}[Proof of Proposition~\ref{thm 2k}]
 	
 	%
 	%
 	``$\Longrightarrow$''. It is straightforward.
 	
 	``$\Longleftarrow$''. By the construction of $\widetilde \beta_1$, we know that $P-\beta'(P)$ is not diagonal for each point $P$ in $\LL_2$; and ${\beta'}^{-1}(Q)-Q$ is not diagonal for each point $Q$ in $m(\LL_2)$. On the other hand, since $\beta'$ and $\widetilde\beta$ are related, there are exact $2\bullet\#\LL_2$ non-diagonal vectors in $\Big\{P-\beta'(P)\;\big|\; P\in \YY_0\Big\}$. Therefore, we have 
	\begin{equation}\label{equation 3}
	\Big\{ P-\beta'(P)\;\big|\; P\in \LL_{\beta'}\Big\}^\star=\Big\{ P-\widetilde\beta(P)\;\big|\; P\in \LL_{\widetilde\beta}\Big\}^\star.
	\end{equation}

 	Recall that we denote $V(\overline{\beta}_2)=\{\vec{v}_1,\vec{v}_2,\dots\}$. Assume that $\beta'$ does not satisfy Property~(1). We put $i$ be the smallest number such that there exists some point $P_0$ which satisfies $$P_0-\widetilde\beta(P_0)=\vec{v}_i\ \textrm{or}\ \vec{v}_i^\vee\quad \textrm{and}\quad P_0-\beta'(P_0)\neq \vec{v}_i\ \textrm{or}\ \vec{v}_i^\vee.$$ It is easy to see that for each point $Q$ in $m(\YY_0)\cap \mathscr D_k$, there exists at most one vector $\vec{v}$ in $V(\overline{\beta}_2)\cup V(\overline{\beta}_2)^\vee$ such that $P_0+\vec{v}$ belongs to $\LL_2$.
 	Combining it with  Lemma~\ref{no inter} allows us to induce a injection $\beta_2': \LL_2\to m(\YY_0)$ from $\widetilde\beta_2$ such that $$\beta_2'(P)=\begin{cases}
 		\widetilde\beta_2(P_0)&\textrm{if}\ P=P_0;\\
 		\beta'(P)& \textrm{else}.
 	\end{cases}$$
  It is easy to check that $\beta_2'$ is greater than $\widetilde\beta_2$ with respect to ``$\prec_2$'', a contradiction.
 	Therefore, $\beta'$ satisfies Property~(1).
 	
 	Apply the same argument to $m\circ {\beta'}^{-1}\circ m$, we know that $${\beta'}^{-1}(Q)-Q={\widetilde{\beta}}^{-1}(Q)-Q\ \textrm{or}\ (\widetilde{\beta}^{-1}(Q)-Q)^\vee\ \textrm{for all}\ Q\ \textrm{in}\ m(\LL_2).$$
 	Recall that we define the partial order ``$\prec_1$'' and $V^\star(\beta)$ in Definitions~\ref{partial order for diagonal vectors} and \ref{Vstar}.
 	One can check that $V^\star(\widetilde\beta|_{\YY_0\backslash \LL_{\widetilde \beta}})$ is actually the only maximal element in the set $$\Big\{V^\star(\beta)\;\big|\; \beta:\YY_0\backslash \LL_{\widetilde \beta}\to m(\YY_0\backslash \LL_{\widetilde \beta})\ \textrm{and}\ \beta\ \textrm{is diagonal}\Big\}$$
 	 with respect to the partial order $``\prec_1$''.
 
 	As a corollary of Lemma~\ref{no inter}, we know that for each $k$, there does not exist two points in $\mathscr D_k$ such that one is from $\LL_1$ and the other is from $\beta'(\LL_2)$. Therefore, if we put $\beta''$ to be a bijection from $\YY_0$ to $m(\YY_0)$ such that $\beta''$ is an element in $$\Big\{\beta:\YY_0\to m(\YY_0)\;|\; \LL_{\beta}=\LL_{\beta'}\ \&\ \beta\ \textrm{is diagonal}\Big\},$$ which maximizes $V^\star(\beta'')$ with respect to ``$\prec_1$'', then we have 
 	$$\beta''|_{\LL_1}=\widetilde \beta|_{\LL_1}.$$ Moreover, we can check that 
 	$V^\star(\beta'')	\preceq_1 V^\star(\widetilde \beta) $ and the equation hold if and only if $\LL_{\beta'}$ is weakly symmetric and $\beta'(P)=m(P)$ for each $P\in \YY_0\backslash(\LL_1\cup\LL_{\beta'}).$ 
 	
 	On the other hand, since $\beta'$ and $\widetilde{\beta}$ are related, we know that $V^\star(\beta')=_1 V^\star(\widetilde \beta)$. Hence, we have $$\beta'|_{\YY_0\backslash \LL_{\beta'}}=\beta''|_{\YY_0\backslash \LL_{\beta''}},$$ which implies that $\beta'$ satisfies Property~(2) and (3).
 	
 	Then we are left to show that $\beta'|\LL_{\beta'}$ is symmetric. First, from the argument above, we know that it is weakly symmetric. Therefore, for any point $P$ in $\LL_2$, if we put $P':=m(\beta'(P))$, we know that $\beta'(P')\in m(\LL_2)$. Since we checked Property~(1) already, we know that $P'-\beta'(P')\in V(\overline{\beta}_2)\cup V(\overline{\beta}_2)^\vee.$ As the argument above, there are at most one vector $\vec{v}$ in $V(\overline{\beta}_2)\cup V(\overline{\beta}_2)^\vee$ such that $P'-\vec{v}\in m(\LL_2)$, which obviously is $P'-m(P)$. Therefore, we know that $\beta'$ is symmetric. 
 	
	(b) It follows directly from (a).
 \end{proof}

 \begin{proof}[Proof of Proposition~\ref{core proposition}]
 	(1) By Proposition~\ref{thm 2k}, we check that $\widetilde{\beta}$ constructed in  \eqref{Construction of beta} is exactly the needed $\widetilde\beta$ in this proposition. Moreover, we know that the integer $i$ in this proposition is equal to \[\#\big\{P\in \LL_2\;|\; P-(P-\widetilde{\beta}_2(P))^\vee\in m(\YY_0)\big\}.\]
 	
 	(2) From Lemma~\ref{same monomial}, we know that two special bijections 
 	contribute a same monomial to $\widetilde{v}_{h(\TT_1)}^\sp$ in Lemma~\ref{same terms} if and only if they are related. By Corollary~\ref{same monomial} and part (1) of this proposition, we have \begin{equation}\label{vh}
 	\widetilde v_{h(\TT_1)}=	\frac{2^i}{\prod\limits_{P\in  \TT_{1}} \prod\limits_{i=1}^{\mathbbm{x}'_1}b_{P,\tau(\widetilde\beta),i}!}
 	\prod\limits_{i=1}^{\mathbbm{x}'_1}\widetilde  a_{Q_i}^{\sum\limits_{P\in \TT_1}b_{P,\tau(\widetilde\beta),i}}+\textrm{``other terms''},
 	\end{equation}
 	where ``other terms'' is a power series in $\ZZ_p[\underline{\widetilde{a}}]$ which contains no term like $\widetilde  a_{Q_i}^{\sum\limits_{P\in \TT_1}b_{P,\tau(\widetilde\beta),i}}$. Since for any $P\in \TT_1$ and any $1\leq i\leq \mathbbm x_1'$, we know that $b_{P,\tau(\widetilde\beta),i}$ in \eqref{vh} is less than $p$, we complete the proof of this proposition.
 \end{proof}

%
%
%

\end{document}